\documentclass[a4paper,11pt]{article}

\usepackage{mathrsfs}
\usepackage{amsmath}
\usepackage{bm}
\usepackage{amssymb}
\usepackage{makeidx}
\makeindex
\usepackage[dvipsnames]{xcolor}
\usepackage[colorlinks=true, linkcolor=blue, citecolor=blue, urlcolor=blue]{hyperref}
\usepackage{titlesec}
\usepackage{amsthm}
\usepackage{setspace}
\usepackage{pifont}
\usepackage{esint}
\usepackage{graphicx}
\usepackage{float}
\usepackage{enumerate}
\numberwithin{equation}{section}
\DeclareMathOperator{\Ent}{Ent}

\usepackage[backend=biber,style=alphabetic,sorting=nyt]{biblatex}
\DeclareFieldFormat[article]{title}{#1}
\renewbibmacro{in:}{}
\DeclareFieldFormat{pages}{#1}
\DeclareFieldFormat[article]
{journaltitle}{\mkbibemph{#1}\addcomma}
\renewbibmacro*{volume+number+eid}{%
	\printfield{volume}%
	\setunit{\addspace}%
	\printfield{eid}%
}
\renewbibmacro*{eprint}{%
	\iffieldundef{eprint}
	{}
	{%
		\setunit{\addcomma\space}% <- 这里决定前面的标点是逗号
		\printtext{%
			\printfield[eprint:arxiv]{eprint}}%
	}%
}

\titleformat{\section}[block]{\bfseries\Large}{\thesection}{1em}{}

\newtheoremstyle{mystyle}
{}{}{\upshape}{}{\bfseries}{.}{.5em}{}
\theoremstyle{mystyle}

\newtheorem{theorem}{Theorem}[section]
\newtheorem{definition}{Definition}[section]
\newtheorem{rmk}{Remark}[section]
\newtheorem{proposition}{Proposition}[section]

\newtheorem{lemma}{Lemma}[section]

\theoremstyle{plain}

\newcommand{\I}{{\cal I}}

\newcommand{\R}{\mathbb{R}}

\newcommand{\N}{\mathbb{N}}

\newcommand{\cL}{\mathcal{L}}
\newcommand{\cF}{\mathcal{F}}

\newcommand{\cJ}{\mathcal{J}}

\begin{document} 
	
	\title{\large
         The Conformal Fractional--Logarithmic Laplacian on the Sphere: \\[2mm] Yamabe Problems and Sharp Inequalities
         }
	
	\author{\small Huyuan Chen\qquad Rui Chen\qquad Daniel Hauer }
	\date{}
	\maketitle
	\thispagestyle{empty}
	\pagenumbering{arabic}
	
\noindent\textbf{Abstract:}
In this paper, we introduce the conformal fractional--logarithmic Laplacian on the unit sphere, defined as the derivative of the conformal fractional Laplacian with respect to the order parameter \(s\in(0,1)\). We investigate its fundamental analytic and spectral properties, including its relation to the conformal logarithmic Laplacian, its spectral representation, and the explicit form of its eigenvalues and eigenfunctions. We further establish its conformal covariance law and derive the associated Yamabe-type equation, proving its equivalence to the corresponding conformal equation in \(\mathbb R^N\) through stereographic projection. Finally, we apply this framework to sharp Sobolev-type inequalities, recovering the sharp logarithmic Sobolev inequality, revealing the failure of a naive fractional--logarithmic analogue, and establishing new sharp fractional--logarithmic inequalities.

	%\bigskip

	% ========== Table of contents ==========
	
	\thispagestyle{empty}
		\bigskip
	
	
\medskip

	%\noindent {\small\bf MSC Subject Classifications: }  
\tableofcontents
    
	% ========== Main text ==========
	\setcounter{page}{1}

\section{Introduction and Main Results}

\subsection{Conformally Covariant Operators}
In recent years, conformally covariant operators on Riemannian manifolds, together with their
associated curvature quantities and curvature prescription problems, have attracted significant
attention in conformal geometry and geometric analysis.
Classical examples include the conformal Laplacian, which
underlies the Yamabe problem; see, e.g., \cite{leeParker1987,aubin1976,schoen1984,hang2016q,gursky2015strong,chen1997priori}.

To be more precise, let $(M,g)$ be a smooth Riemannian manifold of dimension $N\in\mathbb N$, and let
$\mathscr P_g:C^\infty(M)\to C^\infty(M)$ be a pseudo-differential operator.
We say that $\mathscr P_g$ is \emph{conformally covariant} if, under the conformal change of metric
$\tilde g=\eta\,g$ with $\eta\in C^\infty(M)$ positive, it satisfies a transformation law of the form
\begin{equation}\label{eq:conf-law-general}
\mathscr P_{\eta g}(u)=\eta^{-b}\,\mathscr P_g\bigl(\eta^{a}u\bigr),
\quad u\in C^\infty(M),
\end{equation}
for some constants $a,b\in\mathbb R$.
Associated with $\mathscr P_g$ one defines the curvature quantity
\[
Q^{\mathscr P_g}:=\mathscr P_g(1),
\]
which satisfies a $Q$-curvature type equation obtained by inserting $u\equiv 1$ in
\eqref{eq:conf-law-general}. Writing $\eta=w^{1/a}$, we may express this as
\begin{equation}\label{eq:Q-curv-type}
\mathscr P_g(w)=Q^{\mathscr P_{\tilde g}}\,w^{\frac{b}{a}},
\qquad \tilde g=\eta g=w^{1/a}g,\quad w>0.
\end{equation}

\medskip

Let $\Delta_g$ denote the Laplace--Beltrami operator and $R_g$ the scalar curvature of $(M,g)$.
A fundamental example is the conformal Laplacian
\[
\mathscr P_g^{1}:=-\Delta_g+\frac{N-2}{4(N-1)}\,R_g,
\]
for which
\[
Q^{\mathscr P_g^1}:=\mathscr P_g^{1}(1)=\frac{N-2}{4(N-1)}\,R_g.
\]
In the notation of \eqref{eq:conf-law-general}, the operator $\mathscr P_g^{1}$ satisfies
\eqref{eq:conf-law-general} with
\[
a=\frac{N-2}{4},
\qquad
b=\frac{N+2}{4}.
\]
In this case, the $Q$-curvature type equation \eqref{eq:Q-curv-type} takes the form of the classical
\emph{Yamabe equation}. More precisely, if $\tilde g=u^{\frac{4}{N-2}}\,g$ with $u>0$, then
\[\mathscr P_g^{1}(u)
= \frac{N-2}{4(N-1)}\,R_{\tilde g}\,u^{\frac{N+2}{N-2}}
\qquad\text{on }M,\]
or equivalently,
\[-\Delta_g u+\frac{N-2}{4(N-1)}\,R_g\,u
= \frac{N-2}{4(N-1)}\,R_{\tilde g}\,u^{\frac{N+2}{N-2}}
\qquad\text{on }M.\]
In particular, prescribing constant scalar curvature $R_{\tilde g}\equiv \kappa\in\mathbb R$ yields the
constant-curvature Yamabe problem
\[-\Delta_g u+\frac{N-2}{4(N-1)}\,R_g\,u
= \frac{N-2}{4(N-1)}\,\kappa\,u^{\frac{N+2}{N-2}}
\qquad\text{on }M.\]

Beyond this second-order model, one has the higher-order GJMS operators
constructed by Graham--Jenne--Mason--Sparling \cite{grahamJenneMasonSparling1992} and their extensions,
as well as the \emph{fractional conformal Laplacians} $\{\mathscr P_g^{s}\}_{s\in(0,N/2)}$ arising from
scattering theory on conformally compact Einstein (or more generally asymptotically hyperbolic)
fillings \cite{grahamZworski2003,chang2011fractional,gonzalez2016recent}.
These families satisfy \eqref{eq:conf-law-general} with
\[
a=\frac{N-2s}{4},
\qquad
b=\frac{N+2s}{4},
\]
and the associated curvature $Q^{\mathscr P_g^s}:=\mathscr P_g^{s}(1)$ is referred to as the fractional $Q$-curvature.

The case of the sphere with its round metric, denoted by $(\mathbb S^N,g)$, is particularly
important: stereographic projection identifies conformal operators on $\mathbb S^N$ with Euclidean
models on $\mathbb R^N$. In particular, $\mathscr P_g^{s}$ intertwines with the Euclidean fractional Laplacian
$(-\Delta)^s$ on $\mathbb R^N$; see \cite{chang2011fractional,gonzalez2016recent,GMQ13,KSMW18,KMW21} and more general topics see \cite{LX19,JLX14,wei2009prescribing,schoen1996prescribed,jin2014fractional,chen2016existence,case2016fractional,ao2019higher}.

\medskip
Throughout this paper, $(\mathbb S^N,g)$ with $N\ge 1$ denotes the unit sphere
$\mathbb S^N\subset\mathbb R^{N+1}$ endowed with the standard round metric $g=g_{\mathbb S^N}$,
i.e.\ the pullback of the Euclidean metric on $\mathbb R^{N+1}$ via the inclusion map.
For $s\in(0,1)$ and $N>2s$, the conformal fractional Laplacian $\mathscr P_g^{s}$ on $(\mathbb S^N,g)$
admits the singular-integral representation
\cite[Proposition 4.3]{gonzalez2016recent}
\begin{equation}\label{eq:frac-sphere}
\mathscr P_g^{s}u(z)
= c_{N,s}\,\mathrm{p.v.}\!\int_{\mathbb S^N}
\frac{u(z)-u(\zeta)}{|z-\zeta|^{N+2s}}\,dV_g(\zeta)
+ A_{N,s}\,u(z),
\qquad u\in C^\infty(\mathbb S^N),
\end{equation}
where $\mathrm{p.v.}$ denotes the Cauchy principal value and
\begin{equation}\label{eq:constants}
A_{N,s}
:=\frac{\Gamma\!\left(\tfrac N2+s\right)}{\Gamma\!\left(\tfrac N2-s\right)},
\qquad
c_{N,s}
:=4^{s}\pi^{-N/2}s(1-s)\,
\frac{\Gamma\!\left(\tfrac N2+s\right)}{\Gamma(2-s)}.
\end{equation}
Here $\Gamma$ denotes the Gamma function.

Naturally, one is led to investigate the behaviour of $\mathscr P_g^s$ at the endpoints of the
parameter range and, in particular, to understand the logarithmic regime arising as $s\to0^+$.
In this direction, the \emph{conformal logarithmic Laplacian} on the sphere can be introduced as
the derivative of $\mathscr P_g^s$ at $s=0$. More precisely, \cite[Theorem~1.1]{fernandez2025conformal}
provides the following pointwise representation:
\begin{equation}\label{eq:log-def}
\mathscr P_g^{\ln} u(z)
:=\left.\frac{d}{ds}\mathscr P_g^s u(z)\right|_{s=0}
= c_N \!\int_{\mathbb S^N}
\frac{u(z)-u(\zeta)}{|z-\zeta|^N}\,dV_g(\zeta)
+ A_N\,u(z),
\quad u\in C^\infty(\mathbb S^N),
\end{equation}
where
\[
c_N := \pi^{-N/2}\Gamma\!\left(\tfrac N2\right),
\qquad
A_N := 2\,\psi\!\left(\tfrac N2\right),
\]
and $\psi(x)=\Gamma'(x)/\Gamma(x)$ denotes the digamma function.
The conformal logarithmic Yamabe-type problem associated with $\mathscr P_g^{\ln}$ on $\mathbb S^N$
has been recently investigated in \cite{fernandez2025conformal}; see also the references therein for
further background on fractional conformal Laplacians and related curvature problems. Here we'd like to mention that the logarithmic Laplacian in $\R^N$ was first introduced by \cite{CW19} through an asymptotic expansion of the fractional Laplacian. Since then, it has been systematically investigated in a growing body of literature, including \cite{GDB25,CDKN23,JSW20,CV24,CJS24} and the references therein.

\subsection{Conformal Fractional--Logarithmic Laplacian}
 
Our interest in this article is to investigate the \emph{conformal fractional--logarithmic Laplacian}
on the sphere, defined as the derivative of $\mathscr P_g^t$ at an intermediate order
$t=s\in(0,1)$. In this sense, $\mathscr P_g^{s+\ln}$ may be viewed as an operator lying ``between''
the fractional conformal Laplacian $\mathscr P_g^{s}$ and the conformal logarithmic Laplacian
$\mathscr P_g^{\ln}$: it retains the conformal weight dictated by the order $s$, while exhibiting a
logarithmic correction originating from differentiation with respect to the order. Namely, we set
\begin{equation}\label{eq:def-frac-log-sphere}
\mathscr P_g^{s+\ln} u(z)
:= \left.\frac{d}{dt}\mathscr P_g^t u(z)\right|_{t=s},
\qquad u\in C^\infty(\mathbb S^N).
\end{equation}

Our first result establishes a singular-integral representation for $\mathscr P_g^{s+\ln}$, which
makes the nonlocal logarithmic structure explicit.

\begin{theorem}
\label{prop:frac-log-kernel}
\text{(i).} Let $s\in(0,1)$, integer $N>2s$ and  $u\in C^\beta(\mathbb{S}^N)$ with
$\beta>2s$. Then for every $z\in\mathbb{S}^N$, 
$\mathscr{P}^{s+\ln}_g u(z)$ given by~\eqref{eq:def-frac-log-sphere} is
well-defined and admits the representation
\begin{equation}\label{eq:frac-log-kernel}
  \mathscr{P}^{s+\ln}_g u(z)
  = c_{N,s}\,\mathrm{p.v.}\!\int_{\mathbb{S}^N}
      \frac{u(z)-u(\zeta)}{|z-\zeta|^{N+2s}}
      \bigl(-2\ln|z-\zeta|+b_{N,s}\bigr)\,dV_g(\zeta)
    + A'_{N,s}\,u(z),
\end{equation}
where
\begin{equation}\label{eq:frac-log-constants}
    b_{N,s} := \frac{c'_{N,s}}{c_{N,s}}
    = \ln 4
      + \psi\!\Big(\tfrac N2+s\Big)
      + \psi(2-s)
      + \frac{1}{s}
      - \frac{1}{1-s}
\end{equation}
and
\begin{equation}\label{eq:A-Ns-prime}
    A'_{N,s}
    := \frac{d}{dt}A_{N,t}\big|_{t=s}
    = A_{N,s}\Bigl(
          \psi\!\Bigl(\tfrac N2+s\Bigr)
        + \psi\!\Bigl(\tfrac N2-s\Bigr)
      \Bigr).
\end{equation}
Here $\psi$ denotes the Digamma function and $A_{N,s}$ is given in \eqref{eq:constants}.\smallskip

\text{(ii).} Assume that
$u\in C^\beta(\mathbb{S}^N)$ with $\beta>2s$.  For every $1\le p\le\infty$ we have
  \(
     \mathscr{P}^{s+\ln}_g u \in L^p(\mathbb{S}^N),
  \)
  and
  \[
    \frac{\mathscr{P}^t_g u - \mathscr{P}^s_g u}{t-s}
    \longrightarrow \mathscr{P}^{s+\ln}_g u
    \quad\text{in }L^p(\mathbb{S}^N)\,\quad\text{as } t\to s.
  \]
  
\text{(iii).} Let $u\in C^{\beta}(\mathbb{S}^N)$ with $\beta>0$. Then
\[
  \mathscr{P}^{s+\ln}_g u \longrightarrow \mathscr{P}^{\ln}_g u
  \quad\text{uniformly on }\mathbb{S}^N\ \text{as }s\to0^+
\]
and for every $1\le p\le+\infty$, 
\[
  \mathscr{P}^{s+\ln}_g u \longrightarrow \mathscr{P}^{\ln}_g u
  \quad\text{in }L^p(\mathbb{S}^N)\ \text{as }s\to0^+.
\]
\end{theorem}

Next, we introduce a Dini-type continuity condition adapted to the operator $\mathscr{P}^{s+\ln}_g$.

\begin{definition}
Let $u:\mathbb{S}^N\to\mathbb{R}$ be a measurable function and
$z\in\mathbb{S}^N$. The modulus of continuity of $u$ at $z$ is defined
by
\[
  \omega_{u,z}(r)
  := \sup_{\substack{\zeta\in\mathbb{S}^N\\ d_g(z,\zeta)\le r}}
     |u(z)-u(\zeta)|,
  \qquad r>0,
\]
where $d_g$ denotes the geodesic distance on $(\mathbb{S}^N,g)$.

For $s\in(0,1)$ we say that $u$ is fractional--logarithmic Dini
continuous of order $s$ at $z$ if
\[ \int_0^1
    \frac{\omega_{u,z}(r)}{r^{1+2s}}\,
    \bigl(1+|\ln r|\bigr)\,dr <\infty.\]
We say that $u$ is uniformly fractional--logarithmic Dini
continuous of order $s$ on $\mathbb{S}^N$ if
\[
  \omega_u(r):=\sup_{z\in\mathbb{S}^N}\omega_{u,z}(r)
\]
satisfies
\[\int_0^1
    \frac{\omega_u(r)}{r^{1+2s}}\,
    \bigl(1+|\ln r|\bigr)\,dr <\infty.\]
\end{definition}

\begin{proposition}\label{prop:Dini-frac-log}
Let $s\in(0,1)$ and $u\in L^1(\mathbb{S}^N)$.
\begin{enumerate}
  \item[(i)] If $u$ is fractional--logarithmic Dini continuous of order
  $s$ at $z\in\mathbb{S}^N$, then $\mathscr{P}^{s+\ln}_g u(z)$ is
  well defined.

  \item[(ii)] If $u$ is uniformly fractional--logarithmic Dini
  continuous of order $s$, then
  $\mathscr{P}^{s+\ln}_g u$ is continuous.
\end{enumerate}
\end{proposition}

	Next, we study the eigenvalues of the conformal fractional--logarithmic Laplacian. Let $-\Delta_g$ denote the nonnegative Laplace--Beltrami operator on $(\mathbb{S}^N,g)$.
Since $\mathbb{S}^N$ is compact, $-\Delta_g$ has purely discrete spectrum, and its eigenvalues are
\[
\lambda_k := k(k+N-1),\qquad k\in\mathbb{N}_0:=\mathbb{N}\cup\{0\}.
\]
For each $k\in\mathbb{N}_0$, define the eigenspace
\[
\mathcal{H}_k
:= \Big\{\Phi \in C^\infty(\mathbb{S}^N): -\Delta_g\Phi =\lambda_k\Phi \Big\}.
\]
Then $\mathcal{H}_k$ is precisely the space of spherical harmonics of degree $k$.
Denote by $d_k:=\dim\mathcal{H}_k$ its dimension. One has
\[
d_0=1,\quad d_1=N+1,\quad
d_k=\binom{N+k}{N}-\binom{N+k-2}{N}\quad \text{for } k\ge 2.
\]
We fix, once and for all, for each $k\in\mathbb{N}_0$ an orthonormal basis
\[
\{Y_{k,j}\}_{j=1}^{d_k}\subset L^2(\mathbb{S}^N)
\]
of $\mathcal{H}_k$ consisting of real-valued smooth spherical harmonics, i.e.,
\[
-\Delta_g Y_{k,j}=\lambda_k Y_{k,j}\quad \text{on }\mathbb{S}^N,\qquad j=1,\dots,d_k.
\]
In particular, $\mathcal{H}_0$ is spanned by constants, whereas $\mathcal{H}_1$
can be identified with the restrictions to $\mathbb{S}^N$ of linear functions on $\mathbb{R}^{N+1}$.
For further background on spherical Laplacian, see \cite[Section 4.4]{zelditch2017eigenfunctions}. Moreover, the family $\{Y_{k,j}\}_{k\ge 0,\,1\le j\le d_k}$ forms a complete orthonormal basis of
$L^2(\mathbb{S}^N)$. 
  
For the conformal fractional Laplacian $\mathscr P_g^s$ on $(\mathbb S^N,g)$, its action can be
described spectrally in terms of the spherical Laplace--Beltrami eigenvalues $\lambda_k$
(see \cite[Proposition 1.6 and Lemma 2.6]{chang2025fractional}). Namely, define
\begin{equation}\label{varphins}
\varphi_{N,s}(\lambda)
:=
\frac{
\Gamma\!\big(\frac12+s+\sqrt{\lambda+\frac{(N-1)^2}{4}}\big)
}{
\Gamma\!\big(\frac12-s+\sqrt{\lambda+\frac{(N-1)^2}{4}}\big)
},
\qquad \lambda\ge 0,\ s\in(0,1),\ N>2s,
\end{equation}
so that $\varphi_{N,s}(\lambda)>0$ for all $\lambda\ge 0$.
More precisely, if $\Phi\in\mathcal H_k$, then $\Phi$ is also
an eigenfunction of $\mathscr P_g^s$ with eigenvalue $\varphi_{N,s}(\lambda_k)$, i.e.,
\[
\mathscr P_g^s\Phi(z)=\varphi_{N,s}(\lambda_k)\,\Phi(z),
\qquad z\in\mathbb S^N.
\]

To introduce spectrum of the fractional--logarithmic operator, we consider the corresponding symbol
\begin{align}
\varphi^{s+\ln}_N(\lambda)
&:= \left.\frac{d}{dt}\varphi_{N,t}(\lambda)\right|_{t=s} \nonumber\\[1mm]
&= \varphi_{N,s}(\lambda)\Bigl[
\psi\!\Bigl(\tfrac12+s+\sqrt{\lambda+\tfrac14(N-1)^2}\Bigr)
+\psi\!\Bigl(\tfrac12-s+\sqrt{\lambda+\tfrac14(N-1)^2}\Bigr)
\Bigr],
\label{eq:frac-log-eigvalue}
\end{align}
where $\psi=\Gamma'/\Gamma$ is the digamma function. Our main result on the eigenvalues can be stated
as follows.

\begin{theorem}\label{thm:frac-log-basic}
Let $s\in (0,1)$ and let $N$ be an integer with $N>2s$.

\medskip
\noindent\textnormal{(i)}
If $\Phi\in \mathcal H_k$ for some $k\in\mathbb N_0$, then $\Phi$ is also an eigenfunction of
$\mathscr P^{s+\ln}_g$ with eigenvalue $\varphi^{s+\ln}_N(\lambda_k)$, namely
\begin{equation}\label{eq:eig-frac-log}
\mathscr P^{s+\ln}_g\Phi
= \varphi^{s+\ln}_N(\lambda_k)\,\Phi
\quad\text{on }\mathbb S^N.
\end{equation}

\medskip
\noindent\textnormal{(ii)}
If $\Psi\in C^{\beta}(\mathbb S^N)$ is nonzero for some $\beta>2s$ and satisfies
\[
\mathscr P^{s+\ln}_g\Psi=\mu\,\Psi
\quad\text{on }\mathbb S^N
\]
for some $\mu\in\mathbb R$, then there exists a unique $k\in\mathbb N_0$ such that
\[
\mu=\varphi^{s+\ln}_N(\lambda_k)
\quad\text{and}\quad
\Psi\in \mathcal H_k.
\]
\end{theorem}

For part \textnormal{(ii)} of Theorem~\ref{thm:frac-log-basic}, a key step is to show that the
sequence $\{\varphi^{s+\ln}_N(\lambda_k)\}_{k\in\mathbb N_0}$ is strictly increasing in $k$.
%This requires a rather delicate analysis: a direct computation shows that, for general parameters$s\in(0,1)$ and $N>2s$, the function $\lambda\mapsto \varphi^{s+\ln}_N(\lambda)$ is not monotone on$[0,\infty)$. 
We establish the desired monotonicity along the discrete spectrum
$\{\lambda_k\}_{k\in\mathbb N_0}$ in Proposition~\ref{pr spc func}. Moreover, as shown in the proof
of Proposition~\ref{pr spc func}, we obtain the following finer observation:

\begin{rmk}\label{rem:sign-frac-log-eigs}
\textnormal{(a)}
If $N\ge 4$ and $s\in(0,1)$, then $\varphi^{s+\ln}_N(\lambda_k)>0$ for all $k\ge 0$.
If $N\in\{2,3\}$ and $s\in (0,1)$, then $\varphi^{s+\ln}_N(\lambda_k)>0$ for all $k\ge 1$.
If $N=1$ and $s\in[0,\tfrac12)$, then $\varphi^{s+\ln}_1(\lambda_k)>0$ for all $k\ge 2$.
\smallskip

\textnormal{(b)}
For $N=3$, there exists $s_0\in(0,1)$ such that $\varphi^{s_0+\ln}_3(\lambda_0)=0$.
Moreover,
\[
\varphi^{s+\ln}_3(\lambda_0)>0\quad \text{for } s\in(0,s_0),
\qquad
\varphi^{s+\ln}_3(\lambda_0)<0\quad \text{for } s\in(s_0,1).
\]
\smallskip

\textnormal{(c)}
For $N=2$, one has $\varphi^{s+\ln}_2(\lambda_0)<0$ for all $s\in(0,1)$.
\smallskip

\textnormal{(d)}
For $N=1$, there exists $s_1\in(0,\tfrac12)$ such that $\varphi^{s_1+\ln}_1(\lambda_1)=0$.
Moreover,
\[
\varphi^{s+\ln}_1(\lambda_1)>0\quad \text{for } s\in(0,s_1),
\qquad
\varphi^{s+\ln}_1(\lambda_1)<0\quad \text{for } s\in(s_1,\tfrac12),
\]
and
\[
\varphi^{s+\ln}_1(\lambda_0)<0\quad \text{for all } s\in(0,\tfrac12).
\]
\end{rmk}
 
\smallskip

Next, we introduce the $Q$-curvature and the Yamabe-type problem associated with the
fractional--logarithmic conformal operator $\mathscr P^{s+\ln}_g$ on the sphere.
The starting point is the conformal covariance of $\mathscr P^{s+\ln}_g$, which we establish first.

\begin{proposition}\label{prop:conf-frac-log}
Let $s\in(0,1)$ and let $N$ be an integer with $N>2s$. Let $\eta\in C^\infty(\mathbb S^N)$ be
positive. Then for every $u\in C^\infty(\mathbb S^N)$,
\begin{equation}\label{eq:conf-frac-log-eta}
\mathscr P^{s+\ln}_{\eta g}u
= \eta^{-\frac{N+2s}{4}}
\biggl[
\mathscr P^{s+\ln}_g\bigl(\eta^{\frac{N-2s}{4}}u\bigr)
-\frac12\,(\ln\eta)\,\mathscr P^s_g\bigl(\eta^{\frac{N-2s}{4}}u\bigr)
-\frac12\,\mathscr P^s_g\Bigl((\ln\eta)\,\eta^{\frac{N-2s}{4}}u\Bigr)
\biggr].
\end{equation}
Equivalently, setting $\varphi:=\eta^{\frac{N-2s}{4}}u$, one has
\begin{equation}\label{eq:conf-frac-log-phi}
\mathscr P^{s+\ln}_{\eta g}\bigl(\eta^{-\frac{N-2s}{4}}\varphi\bigr)
= \eta^{-\frac{N+2s}{4}}
\biggl[
\mathscr P^{s+\ln}_g(\varphi)
-\frac12\,(\ln\eta)\,\mathscr P^s_g(\varphi)
-\frac12\,\mathscr P^s_g\bigl((\ln\eta)\,\varphi\bigr)
\biggr].
\end{equation}
\end{proposition}

As in the purely fractional case, it is convenient to parametrize conformal metrics on $\mathbb S^N$
by a positive function $u\in C^\infty(\mathbb S^N)$ via
\begin{equation}\label{eq:conf-metric-u}
\tilde g = u^{\frac{4}{N-2s}}\,g.
\end{equation}
The reason for keeping the same conformal exponent $4/(N-2s)$ is structural:
the operator $\mathscr P^{s+\ln}_g$ is defined as the derivative with respect to the order
of the conformal fractional Laplacian $\mathscr P_g^s$, and therefore it inherits the same conformal
weight. In particular, under the change of metric \eqref{eq:conf-metric-u}, the principal
homogeneity is still dictated by $\mathscr P_g^s$, while the logarithmic nature of
$\mathscr P^{s+\ln}_g$ manifests itself only through additional $\ln u$--terms in the transformation law. In particular, we take
\[
\eta = u^{\frac{4}{N-2s}},
\qquad\text{so that}\qquad
\eta^{\frac{N-2s}{4}}=u,\quad
\eta^{-\frac{N+2s}{4}}=u^{-\frac{N+2s}{N-2s}},\quad
\ln\eta=\frac{4}{N-2s}\,\ln u.
\]
Substituting into \eqref{eq:conf-frac-log-phi} yields the conformal covariance law in the form
\begin{equation}\label{eq:conf-frac-log-u}
\mathscr P^{s+\ln}_{\tilde g}(\varphi)
= u^{-\frac{N+2s}{N-2s}}
\biggl[
\mathscr P^{s+\ln}_g(u\varphi)
-\frac{2}{N-2s}\,(\ln u)\,\mathscr P^s_g(u\varphi)
-\frac{2}{N-2s}\,\mathscr P^s_g\bigl((\ln u)\,u\varphi\bigr)
\biggr]
\end{equation}
for every $\varphi\in C^\infty(\mathbb S^N)$ and every positive $u\in C^\infty(\mathbb S^N)$.

\medskip

We define the \emph{fractional--logarithmic $Q$-curvature} associated with $\mathscr P^{s+\ln}_g$ by
\[
Q^{s+\ln}_{g}:=\mathscr P^{s+\ln}_{g}(1).
\]
Let $\tilde g=u^{\frac{4}{N-2s}}g$ with $u>0$. Taking $\varphi\equiv 1$ in
\eqref{eq:conf-frac-log-u} yields the transformation formula
\begin{equation}\label{eq:Q-frac-log-trans}
Q^{s+\ln}_{\tilde g}
= u^{-\frac{N+2s}{N-2s}}
\biggl[
\mathscr P^{s+\ln}_g(u)
-\frac{2}{N-2s}\,(\ln u)\,\mathscr P^s_g(u)
-\frac{2}{N-2s}\,\mathscr P^s_g\bigl((\ln u)\,u\bigr)
\biggr]
\quad\text{on }\mathbb S^N.
\end{equation}
Equivalently, \eqref{eq:Q-frac-log-trans} can be rewritten as the nonlinear equation for the conformal
factor $u$:
\begin{equation}\label{eq:frac-log-Yamabe-eq}
\mathscr P^{s+\ln}_g(u)
-\frac{2}{N-2s}\Bigl[(\ln u)\,\mathscr P^s_g(u)+\mathscr P^s_g\bigl((\ln u)\,u\bigr)\Bigr]
= Q^{s+\ln}_{\tilde g}\,u^{\frac{N+2s}{N-2s}}
\quad\text{on }\mathbb S^N.
\end{equation}
In particular, prescribing constant fractional--logarithmic $Q$-curvature,
$Q^{s+\ln}_{\tilde g}\equiv\mu\in\mathbb R$, leads to the fractional--logarithmic Yamabe-type problem
\begin{equation}\label{eq:frac-log-Yamabe-SN}
\mathscr P^{s+\ln}_g(u)
-\frac{2}{N-2s}\Bigl[(\ln u)\,\mathscr P^s_g(u)+\mathscr P^s_g\bigl((\ln u)\,u\bigr)\Bigr]
= \mu\,u^{\frac{N+2s}{N-2s}}
\quad\text{on }\mathbb S^N.
\end{equation}

\medskip

Next, we explain how solutions to the fractional--logarithmic Yamabe-type problem converge, as
$s\to0^+$, to solutions of the logarithmic Yamabe-type equation. Moreover, the latter equation admits
a complete classification of solutions; see \cite{fernandez2025conformal,chen2024positive,frank2020classification}.

\begin{proposition}\label{prop:limit-frac-log-yamabe-us}
Let $s_k\to0^+$ as $k\to+\infty$ and  $u_{s_k}\in C^\infty(\mathbb S^N)$ be positive solutions of
\begin{equation}\label{eq:frac-log-Yamabe-s-family-us}
\mathscr{P}^{s_k+\ln}_g(u_{s_k})
= \frac{2}{N-2s_k}\,\mathscr{P}^{s_k}_g(u_{s_k})\,\ln u_{s_k}
+ \frac{2}{N-2s_k}\,\mathscr{P}^{s_k}_g\bigl((\ln u_{s_k})\,u_{s_k}\bigr)
+ \mu_{s_k}\,u_{s_k}^{\frac{N+2s_k}{N-2s_k}}
\quad\text{on }\mathbb{S}^N.
\end{equation}
Assume that
\[
\mu_{s_k}\to\mu_0\in\mathbb R,
\qquad
u_{s_k}\to u_0 \quad\text{in }C^\beta(\mathbb S^N)\quad {\rm as}\ k\to+\infty
\]
for some $\beta>0$, and that $u_0>0$ on $\mathbb S^N$.
Then $u_0$ satisfies the logarithmic Yamabe-type equation
\begin{equation}\label{eq:log-yamabe-limit-us}
\mathscr P_g^{\ln}(u_0)
= \frac{4}{N}\,u_0\ln u_0 + \mu_0\,u_0
\quad\text{on }\mathbb S^N.
\end{equation}
\end{proposition}

\medskip

    Next, we relate the operator on $\mathbb S^N$ to the Euclidean fractional--logarithmic Laplacian
via the following intertwining identity.

\begin{proposition}\label{prop:sphere-RN-frac-log}
Let $s\in(0,1)$ and  $u\in C^\infty(\mathbb S^N)$. Then, for every $x\in\mathbb R^N$,
\begin{equation}\label{eq:intertwining-frac-log}
\mathcal T_s\!\bigl[\mathscr P^{s+\ln}_g u\bigr](x)
= \phi(x)^{-2s}\Bigl[
(-\Delta)^{s+\ln}\,\mathcal T_s[u](x)
-(-\Delta)^s\!\bigl((\ln\phi)\,\mathcal T_s[u]\bigr)(x)
-\bigl(\ln\phi(x)\bigr)\,(-\Delta)^s\mathcal T_s[u](x)
\Bigr].
\end{equation}
\end{proposition}

As a consequence, the constant fractional--logarithmic $Q$-curvature problem on $\mathbb S^N$
is equivalent to a Yamabe-type equation on $\mathbb R^N$. More precisely, for $N>2s$ and $\mu\in\mathbb R$,
one is led to
\begin{equation}\label{eq:frac-log-Yamabe-RN}
(-\Delta)^{s+\ln} v
-\frac{2}{N-2s}\Bigl[(\ln v)\,(-\Delta)^s v+(-\Delta)^s\bigl(v\ln v\bigr)\Bigr]
= \mu\,v^{\frac{N+2s}{N-2s}}
\quad\text{in }\mathbb R^N,
\end{equation}
where $(-\Delta)^{s+\ln}$ denotes the fractional--logarithmic Laplacian on $\mathbb R^N$
introduced in \cite{chen2026fractional,chen2026potential}. More precisely, for $u\in C_{\mathrm{loc}}^2(\mathbb{R}^N)\cap L^\infty(\mathbb{R}^N)$,  the fractional--logarithmic Laplacian is defined as the derivative of the fractional Laplacian with respect to the order:
\[
  (-\Delta)^{s+\ln}u(x)
  :=\left.\frac{d}{dt}\,(-\Delta)^t u(x)\right|_{t=s}.
\]
Equivalently, it admits the singular-integral representation
\begin{equation}\label{op-s+log}
    (-\Delta)^{s+\ln} u(x)
  = c_{N,s}\,\mathrm{p.v.}\!\int_{\mathbb{R}^N}
  \frac{u(x)-u(y)}{|x-y|^{N+2s}}
  \bigl(-2\ln|x-y|+b_{N,s}\bigr)\,dy,
\end{equation}
where \(c_{N,s},b_{N,s}\) are defined in \eqref{eq:constants} \eqref{eq:frac-log-constants}, respectively. 

In fact, thanks to Proposition~\ref{prop:sphere-RN-frac-log}, the spherical operator
$\mathscr P^{s+\ln}_g$ and the Euclidean operator $(-\Delta)^{s+\ln}$ are intertwined through the
conformal pullback $\mathcal T_s$. Consequently, one may equivalently take either side: the
spectral definition on $\mathbb S^N$ or the singular-integral definition on $\mathbb R^N$ as the
starting point, and recover the corresponding operator on the other space via
\eqref{eq:intertwining-frac-log}.

We can now formulate the precise correspondence between the spherical and Euclidean
fractional--logarithmic Yamabe-type problems via stereographic projection and the conformal pullback
$\mathcal T_s$.

 \begin{theorem}
\label{thm:frac-log-correspondence}
  Let $s\in(0,1)$,  $\mu\in\mathbb{R}$, 
  $u,v$ be positive classical solutions of (\ref{eq:frac-log-Yamabe-SN}) on the unit sphere,  of \eqref{eq:frac-log-Yamabe-RN} in $ \mathbb{R}^N$, respectively. Then 
  $\mathcal T_s[u]$,   $\mathcal T_s^{-1}[v]$ are  positive solutions of  \eqref{eq:frac-log-Yamabe-RN},  
  (\ref{eq:frac-log-Yamabe-SN}) respectively, where  $\mathcal{T}_s$ denotes the conformal pullback operator mapping functions defined on $\mathbb{S}^N$ to ones on $\mathbb{R}^N$, as defined in (\ref{taus}), and $\mathcal{T}_s^{-1}$ is its inverse.
   
\end{theorem}

We next exhibit an explicit family of positive solutions to the fractional--logarithmic Yamabe-type
equations. Remarkably, the solutions are given by the standard fractional ``bubble'' profile on
$\mathbb R^N$ and the corresponding constant solution on $\mathbb S^N$, linked by the conformal
pullback $\mathcal T_s$.

\begin{theorem}\label{prop:radial-bubble-frac-log}
Let $s\in(0,1)$, $N>2s$ and $C>0$. Define
\[
v_{s,C}(x):=C\left(\frac{2}{1+|x|^2}\right)^{\frac{N-2s}{2}},
\qquad x\in\mathbb R^N,
\]
and
\[
u_C\equiv C \quad\text{on }\mathbb S^N.
\]
If
\[
\mu
=
C^{-\frac{4s}{N-2s}}
\Bigl(
A'_{N,s}
-\frac{4}{N-2s}\,A_{N,s}\ln C
\Bigr),
\]
where $A_{N,s}$ and $A'_{N,s}$ are given in \eqref{eq:constants} and \eqref{eq:A-Ns-prime}, respectively,
then $u_C$ and $v_{s,C}$ solve \eqref{eq:frac-log-Yamabe-SN} on $\mathbb S^N$ and
\eqref{eq:frac-log-Yamabe-RN} on $\mathbb R^N$, respectively. Moreover,
\[
v_{s,C}=\mathcal T_s[u_C],
\]
where $\mathcal T_s$ denotes the conformal pullback from $\mathbb S^N$ to $\mathbb R^N$.
\end{theorem}

Theorem~\ref{prop:radial-bubble-frac-log} provides explicit positive solutions to the
fractional--logarithmic Yamabe-type equations \eqref{eq:frac-log-Yamabe-SN} and \eqref{eq:frac-log-Yamabe-RN}.
A notable feature is that these solutions coincide with the standard bubbles arising in the
fractional Yamabe equation. This is explained by the mixed-order structure of the operator in
\eqref{eq:frac-log-Yamabe-RN}: the combination
\[
(-\Delta)^{s+\ln}v
-\frac{2}{N-2s}\Bigl[(\ln v)\,(-\Delta)^s v+(-\Delta)^s\bigl(v\ln v\bigr)\Bigr]
\quad\text{in }\mathbb R^N
\]
acts, for the bubble profile, as an effective fractional operator of order $2s$.
Recall that the classical fractional Yamabe equation
\[
(-\Delta)^s w=\mu\, w^{\frac{N+2s}{N-2s}}
\quad\text{in }\mathbb R^N
\]
admits (and, up to the natural conformal symmetries, only admits) the bubble family $v_{s,C}$;
see \cite{CLO06}. Consequently, the leading balance in \eqref{eq:frac-log-Yamabe-RN} is governed by a
fractional-order mechanism closely paralleling the classical problem. Nevertheless, because the
equation involves a genuine interaction between operators of different orders, the uniqueness of
positive solutions for \eqref{eq:frac-log-Yamabe-SN}--\eqref{eq:frac-log-Yamabe-RN} remains open.

\subsection{Sharp Sobolev-Type Inequalities}

In this subsection, we further explore the connection between inequalities associated with conformal operators on the sphere and sharp Sobolev-type inequalities for nonlocal operators in the Euclidean setting. We begin with a proposition showing that the sharp logarithmic Sobolev inequality on \(\mathbb R^N\), formulated in terms of the logarithmic Laplacian, is closely related to the celebrated logarithmic Beckner inequality on the Sphere; see \cite{beckner1997logarithmic,frank2020classification}.

In this paper, $\mathcal S(\mathbb R^n)$ denotes the Schwartz space of all smooth functions on $\mathbb R^n$ that, together with all their derivatives, decay faster than any polynomial at infinity and the Fourier transform is defined by
\[
\widehat{f}(\xi):=(2\pi)^{-N/2}\int_{\mathbb R^N}e^{-ix\cdot \xi}f(x)\,dx.
\]
We first recall the definition of the logarithmic Laplacian in $\mathbb{R}^N$. It was introduced as the derivative of the fractional Laplacian with respect to the order at $t=0$, see \cite{CW19}
\[
(-\Delta)^{\ln}u(x)
:=\left.\frac{d}{dt}\,(-\Delta)^t u(x)\right|_{t=0}.
\]
Equivalently, if $u\in L^1(\mathbb{R}^N)$ is Dini continuous, then $(-\Delta)^{\ln}u$ admits the singular integral representation
\[
(-\Delta)^{\ln}u(x)
=
c_N\,\mathrm{p.v.}\!\int_{\mathbb{R}^N}
\frac{u(x)\mathbf{1}_{B_1(x)}(y)-u(y)}{|x-y|^N}\,dy
+\rho_N\,u(x),
\]
where
\[
c_N:=\pi^{-N/2}\Gamma\!\left(\frac{N}{2}\right),
\qquad
\rho_N:=2\ln 2+\psi\!\left(\frac{N}{2}\right)+\Gamma'(1).
\]

\begin{proposition}\label{prop:logS-beckner-equivalence}
The sharp Euclidean logarithmic Sobolev inequality
\[\frac{2}{N}\int_{\mathbb{R}^N}\frac{|v|^2}{\|v\|_2^2}
\ln\!\left(\frac{|v|^2}{\|v\|_2^2}\right)\,dx
\le a_N+\frac{\langle v,(-\Delta)^{\ln} v\rangle}{\|v\|_2^2}
\quad {\rm for}\ \, v\in C_c^{\infty}(\mathbb{R}^N)\setminus\{0\},\]
holds, and it is equivalent to the Beckner's logarithmic Sobolev inequality on $\mathbb{S}^N$
\[\iint_{\mathbb S^N\times\mathbb S^N}
\frac{(u(z)-u(\zeta))^2}{|z-\zeta|^N}\,dV_g(z)\,dV_g(\zeta)
\ge
C_N\int_{\mathbb S^N}|u|^2
\ln\!\left(\frac{|u|^2\,|\mathbb S^N|}{\|u\|_2^2}\right)\,dV_g,
\quad{\rm for}\ \,  u\in C^\infty(\mathbb S^N)\setminus \{0\},\]
where
\[a_N
:=\frac{2}{N}\ln\!\left(\frac{\Gamma(N)}{\Gamma(\frac N2)}\right)-\ln(4\pi)-2\psi\!\left(\frac N2\right),\quad C_N:=\frac{4}{N}\,\frac{\pi^{N/2}}{\Gamma(\frac N2)}.\]
\end{proposition}

By the sharp fractional Sobolev inequality on \(\mathbb R^N\), we have, for every \(v\in C_c^{\infty}(\mathbb{R}^N)\),
\[
F_v(s):=\kappa_{N,s}\,\|v\|_{\dot H^s(\mathbb R^N)}^2-\|v\|_{L^{2_s^*}(\mathbb R^N)}^2\ge 0,
\qquad s\in[0,1),
\]
and moreover \(F_v(0)=0\); see Section \ref{sharplog} for the relevant notation and further details. It follows that
\[
F_v'(0^+)\ge 0.
\]
As we shall see, this endpoint differential inequality is in fact equivalent to the sharp Euclidean logarithmic Sobolev inequality. In other words, the logarithmic Sobolev inequality arises naturally as the \(s\to0^+\) first-order differential form of the sharp fractional Sobolev inequality.
This naturally leads to the  question: do analogous inequalities hold for the fractional--logarithmic operator in the Euclidean setting, or for the conformal fractional--logarithmic operator on \(\mathbb S^N\)?

We first address this question for the extremals of the sharp fractional Sobolev inequality. Namely, we prove that the corresponding sharp fractional--logarithmic Sobolev-type identity is satisfied by fractional Sobolev extremals. This is formulated in the following proposition.

\begin{proposition}
\label{prop:fraclog-identity-variational}
Let $N>2s,$ the fractional extremal $u_s$ with $s\in (0,1)$ satisfies the sharp fractional--logarithmic identity
\begin{equation}\label{eq:sharp-fraclog-identity-expanded}
\frac{2}{N}\,\Ent_{p(s)}(u_s)
=
\kappa'_{N,s}\,\frac{\int_{\mathbb R^N}|\xi|^{2s}\,|\widehat u_s(\xi)|^2\,d\xi}{\|u_s\|_{L^{p(s)}(\mathbb R^N)}^2}
+\kappa_{N,s}\,
\frac{\int_{\mathbb R^N}|\xi|^{2s}\ln|\xi|^2\,|\widehat u_s(\xi)|^2\,d\xi}{\|u_s\|_{L^{p(s)}}^2}.
\end{equation}
 where \(u_s\) and \(\kappa_{N,s}\) are defined in \eqref{extremal} and \eqref{eq:kappaNs}, respectively, and
\[
\Ent_{p(s)}(v)
:=\int_{\mathbb R^N}\frac{|v|^{p(s)}}{\|v\|_{L^{p(s)}}^{p(s)}}
\ln\!\left(\frac{|v|^{p(s)}}{\|v\|_{L^{p(s)}}^{p(s)}}\right)\,dx,
\qquad
p(s):=\frac{2N}{N-2s},
\]
denotes the entropy functional associated with the exponent \(p(s)\).
\end{proposition}

By Remark \ref{examd}, in the limiting case \(s=0\), the above identity formally becomes
\begin{equation}\label{euclilog}
    \frac{2}{N}\,\Ent_2(u_0)
=
a_N+\frac{\int_{\mathbb R^N}\ln|\xi|^2\,|\widehat u_0(\xi)|^2\,d\xi}{\|u_0\|_2^2},\quad u_0:=(1+|x|^2)^{-\frac{N}{2}}.
\end{equation}
This is exactly the equality case of the sharp Euclidean logarithmic Sobolev inequality.

\begin{theorem}\label{thm:naive-fraclog-fails}
Let \(N>2s\) and \(0<s<1\).
The general sharp  fractional--logarithmic sobolev inequality fails:
\begin{equation}\label{eq:naive-fraclog-ineq}
\frac{2}{N}\,\Ent_{p(s)}(v)\le
\kappa'_{N,s}\,\frac{\int_{\mathbb R^N}|\xi|^{2s}\,|\widehat v(\xi)|^2\,d\xi}{\|v\|_{L^{p(s)}(\mathbb R^N)}^2}
+\kappa_{N,s}\,\frac{\int_{\mathbb R^N}|\xi|^{2s}\ln|\xi|^2\,|\widehat v(\xi)|^2\,d\xi}{\|v\|_{L^{p(s)}(\mathbb R^N)}^2}
\quad v\in \dot H^s(\mathbb R^N)\setminus\{0\}.
\end{equation}
Equivalently, the monotonicity property
\begin{equation}\label{eq:Fprime-nonneg-thm}
F_v'(s)\ge 0
\qquad\text{for all }s\in(0,1),\ \forall\, v\in \dot H^s(\mathbb R^N)\setminus\{0\},
\end{equation}
fails in general.
\end{theorem}

We now transfer the sharp fractional--logarithmic identity for Euclidean extremals to the sphere via stereographic projection. As in the entropy itself, the conformal change of variables produces an additional logarithmic correction term involving the conformal factor. Define
\[
\Ent_{p(s)}^{\mathbb S^N}(v)
:=
\int_{\mathbb S^N}
\frac{|v|^{p(s)}}{\|v\|_{L^{p(s)}(\mathbb S^N)}^{p(s)}}
\ln\!\left(
\frac{|v|^{p(s)}}{\|v\|_{L^{p(s)}(\mathbb S^N)}^{p(s)}}
\right)\,dV_g,\quad p(s)=\frac{2N}{N-2s}.
\]

\begin{proposition}\label{prop:fraclog-identity-sphere-correct}
Let \(0<s<1\), and let \(U_s\) be a spherical extremal corresponding to the Euclidean bubble \(u_s\), namely $u_s=\mathcal T_sU_s,$
where \(u_s\) is defined in \eqref{extremal}. Then
\begin{align}
\frac{2}{N}\,\Ent_{p(s)}^{\mathbb S^N}(U_s)
&=
\kappa'_{N,s}\,
\frac{\langle U_s,\mathscr P_g^sU_s\rangle_{\mathbb S^N}}
{\|U_s\|_{L^{p(s)}(\mathbb S^N)}^2}
+\kappa_{N,s}\,
\frac{\langle U_s,\mathscr P^{s+\ln}_g U_s\rangle_{\mathbb S^N}}
{\|U_s\|_{L^{p(s)}(\mathbb S^N)}^2}
-2\int_{\mathbb S^N}
\frac{|U_s(\omega)|^{p(s)}}{\|U_s\|_{L^{p(s)}(\mathbb S^N)}^{p(s)}}
\ln\phi(\sigma(\omega))\,dV_g(\omega)
\notag\\
&\qquad
+\frac{2\kappa_{N,s}}{\|U_s\|_{L^{p(s)}(\mathbb S^N)}^2}
\int_{\mathbb S^N}
\ln\phi(\sigma(\omega))\,U_s(\omega)\,\mathscr P_g^s U_s(\omega)\,dV_g(\omega).
\label{eq:sharp-fraclog-identity-sphere-correct}
\end{align}
\end{proposition}

At \(s=0\), since
\[
p(0)=2,\qquad \kappa_{N,0}=1,\qquad \kappa'_{N,0}=a_N,
\]
identity \eqref{eq:sharp-fraclog-identity-sphere-correct} reduces to
\[
\frac{2}{N}\,\Ent_{2}^{\mathbb S^N}(U_0)
=
a_N+
\frac{\langle U_0,\mathscr P_g^{\ln}U_0\rangle_{\mathbb S^N}}
{\|U_0\|_{L^{2}(\mathbb S^N)}^2},
\]
since the two \(\ln\phi\)-correction terms cancel exactly. This is precisely the equality case of the sharp logarithmic Sobolev inequality on \(\mathbb S^N\), see the proof of Proposition \ref{prop:logS-beckner-equivalence}:
\[\frac{2}{N}\int_{\mathbb{S}^N}\frac{|u|^2}{\|u\|_2^2}
\ln\!\left(\frac{|u|^2}{\|u\|_2^2}\right)\,dV_g
\le a_N+\frac{\int_{\mathbb{S}^N}u\,\mathscr P_g^{\ln}u\,dV_g}{\|u\|_2^2}
\quad{\rm for}\ \, u\in C^\infty(\mathbb{S}^N)\setminus\{0\}.\]

A natural question, in view of the failure discussed above, is how to derive new sharp inequalities governed by the fractional--logarithmic energy
\[
\langle u,(-\Delta)^{s+\ln}u\rangle.
\]
To address this problem, we develop three different approaches based on Beckner's logarithmic uncertainty principle, entropy-type identities, and Jensen's inequality. These lead to three corresponding sharp fractional--logarithmic Sobolev inequalities.

\begin{proposition}\label{prop:fraclog-beckner}
Let \(0<s<1\), and let \(u\in \mathcal S(\mathbb R^N)\) satisfy
\[
\|(-\Delta)^{s/2}u\|_{L^2(\mathbb R^N)}^2
=
\int_{\mathbb R^N}|\xi|^{2s}\,|\widehat u(\xi)|^2\,d\xi
=1.
\]
Then
\begin{equation}\label{eq:fraclog-ineq-normalized}
\frac{N}{4}\,\langle u,(-\Delta)^{s+\ln}u\rangle
\ge
\int_{\mathbb R^N}\big|(-\Delta)^{s/2}u(x)\big|^2
\ln\big|(-\Delta)^{s/2}u(x)\big|\,dx
+
B_N,
\end{equation}
where
\[
B_N
:=
\frac{N}{2}\psi\!\Big(\frac{N}{2}\Big)
-\frac{N}{4}\ln\pi
-\frac12\ln\!\Big(\frac{\Gamma(N)}{\Gamma(\frac N2)}\Big)
+\frac{N}{2}\ln(2\pi).
\]
Moreover, equality holds in \eqref{eq:fraclog-ineq-normalized} provided that $(-\Delta)^{s/2}u$
is an extremal for Beckner's logarithmic uncertainty principle, namely, up to conformal automorphisms,
\[
(-\Delta)^{s/2}u(x)=\left(\pi^{N/2}\frac{\Gamma(N/2)}{\Gamma(N)}\right)^{-1/2}(1+|x|^2)^{-N/2}.
\]

\end{proposition}
\begin{proposition}
\label{prop:fraclog_moment}
For every $u\in\mathcal S(\R^N)$ such that
\[\|(-\Delta)^{s/2}u\|_2^2=1,\]
one has the strict inequality
\begin{equation}\label{eq:fraclog-moment-normalized}
\langle u,(-\Delta)^{s+\ln}u\rangle
\ >\
-\ln\Bigg(\frac{2\pi e}{N}\int_{\R^N}|x|^2\big|(-\Delta)^{s/2}u(x)\big|^2\,dx\Bigg)
+\frac{4}{N}B_N.
\end{equation}
\end{proposition}

\begin{proposition}\label{prop:fraclog-Lq-bound}
Let \(1\le q<2\). For every \(u\in \mathcal S(\mathbb R^N)\) satisfying
\[\|(-\Delta)^{s/2}u\|_{L^2(\mathbb R^N)}^2
=1,\]
one has
\begin{equation}\label{eq:fraclog-Lq-normalized}
\frac{N}{4}\,\langle u,(-\Delta)^{s+\ln}u\rangle
>
\frac{1}{q-2}\ln\Big\|(-\Delta)^{s/2}u\Big\|_{L^q(\mathbb R^N)}^{q}
+ B_N.
\end{equation}
\end{proposition}

 The remainder of this paper is organized as follows. Section 2 establishes the foundational properties of the conformal fractional logarithmic Laplacian, including its definition via derivatives of fractional operators, continuity and regularity results, and an explicit characterization of its eigenvalues. Section 3 addresses conformal covariance and Yamabe-type problems: we prove their equivalence and establish the existence of positive solutions. Finally, Section 4 is devoted to applications to sharp Sobolev inequalities, where we revisit the sharp logarithmic Sobolev inequality, show the failure of a naive fractional--logarithmic analogue, and derive new sharp fractional--logarithmic inequalities.
    
  \smallskip 
	
	%====================================================

	\section{The Conformal Fractional-Logarithmic Laplacian}
    
 In this section we introduce the conformal fractional--logarithmic Laplacian on the round sphere, establish its kernel
representation as well as its spectral characterization in terms of spherical harmonics.

\subsection{Stereographic Projection}\label{stere}

We regard $\mathbb{S}^N$ as the unit sphere in $\mathbb{R}^{N+1}$ and
denote by $-e_{N+1}=(0,\dots,0,-1)$ the south pole. The stereographic
projection with respect to the south pole
\[
  \sigma:\mathbb{S}^N\setminus\{-e_{N+1}\}\to \mathbb{R}^N
\]
is given by
\[  z=(z',z_{N+1})\in\mathbb{S}^N\setminus\{-e_{N+1}\}
  \quad\longmapsto\quad
  x=\sigma(z):=\frac{z'}{1+z_{N+1}}\in \mathbb{R}^N,\]
where $z'=(z_1,\dots,z_N)\in \mathbb{R}^N$.
Its inverse is
\[ x\in \mathbb{R}^N
  \quad\longmapsto\quad
  z=\sigma^{-1}(x)
  =\Bigl(
      \frac{2x}{1+|x|^2},
      \frac{1-|x|^2}{1+|x|^2}
    \Bigr)\in\mathbb{S}^N\setminus\{-e_{N+1}\}.\]
It is well known that $\sigma$ is a conformal diffeomorphism. In stereographic coordinates, the pullback of the round metric on $\mathbb{S}^N$ is given by
\[ g\big|_{\mathbb{S}^N\setminus\{-e_{N+1}\}}
  = \phi(x)^2\,\delta_{ij}\,dx_i\otimes dx_j,
  \qquad x=\sigma(z),\]
where
\[ \phi(x):=\frac{2}{1+|x|^2}\quad {\rm for}\ x\in \mathbb{R}^N.\]
Hence the round metric is conformal to the flat Euclidean metric, with conformal factor $\phi^2$. 

Motivated by this,  we introduce the conformal pullback  of a function \(u\) on \(\mathbb S^N\) to \(\mathbb R^N\) by
\begin{equation}\label{taus}
     \mathcal{T}_s[u](x):
  = \phi(x)^{\frac{N-2s}{2}}\,u\bigl(\sigma^{-1}(x)\bigr)
  \quad{\rm for}\ \, x\in \mathbb{R}^N,\   s\in(0,1).
\end{equation}
This is the natural transformation law associated with the conformal covariance of the operator of order \(2s\). Since stereographic projection is a diffeomorphism between
$\mathbb{S}^N\setminus\{-e_{N+1}\}$ and $\mathbb{R}^N$, the pullback $\mathcal{T}_s[u]$ only uses the
values of $u$ on $\mathbb{S}^N\setminus\{-e_{N+1}\}$. The pole $p=-e_{N+1}$ corresponds to the point
at infinity under stereographic projection; hence the regularity of $u$ at $p$ is encoded by the
asymptotic behavior of $v=\mathcal{T}_s[u]$ as $|x|\to\infty$. In particular, $u$ extends
continuously to $p$ if and only if
\[
v(x)\sim u(p)\,\phi(x)^{\frac{N-2s}{2}}
\qquad |x|\to\infty,
\]
since $\phi(x)=\frac{2}{1+|x|^2}\sim 2|x|^{-2}$ as $|x|\to\infty$. In particular, for \(u\in C^\infty(\mathbb S^N)\), the transplanted function $\mathcal T_s[u]$
belongs to \(C^\infty(\mathbb R^N)\cap L^\infty(\mathbb R^N)\), and in fact
\(\mathcal T_s[u](x)=o(|x|^{-(N-2s)})\) as \(|x|\to\infty\).

Moreover, if \(u\) solves the fractional Yamabe equation
\begin{equation}\label{fracsn}
     \mathscr{P}^s_gu=\lambda\,u^{\frac{N+2s}{N-2s}}
  \quad\text{on }\,\mathbb S^N,
\end{equation}
with \(u>0\) on $\mathbb{S}^N$, then \(v_s:=\mathcal T_s[u]\) solves
\begin{equation}\label{fracrn}
     (-\Delta)^s v_s=\lambda\,v_s^{\frac{N+2s}{N-2s}}
  \quad\text{in}\,\,\mathbb R^N.
\end{equation}
In particular, if \(u\equiv C\) is a positive constant solution of $\eqref{fracsn}$, then
\[
  v_s(x)=C\left(\frac{2}{1+|x|^2}\right)^{\frac{N-2s}{2}}
\]
is a radial bubble solution of the Euclidean critical equation $\eqref{fracrn}$.

	%====================================================
    
\subsection{Kernel Representation and Logarithmic Dini Regularity}

In this subsection we establish the singular-kernel representation of $\mathscr P^{s+\ln}_g$ and the
basic regularity consequences needed throughout the paper. The proof proceeds by differentiating
the integral formula for $\mathscr P_g^{t}$ with respect to the order $t$ at $t=s$, which produces a
logarithmic perturbation of the standard fractional kernel. In particular, the principal value
structure and the precise constants must be tracked carefully. We also verify that the resulting
operator is well-defined on Hölder classes $C^\beta(\mathbb S^N)$ with $\beta>2s$, and we establish
the convergence of the difference quotients in $L^p$ as well as the limiting behavior as $s\to0^+$,
which connects $\mathscr P_g^{s+\ln}$ with the conformal logarithmic Laplacian.

\medskip
\noindent{\bf Proof of Theorem \ref{prop:frac-log-kernel}. }
$\textbf{(i)}$
Let $u\in C^\beta(\mathbb{S}^N)$ with $\beta>2s$ and fix $z\in\mathbb{S}^N$.
Writing~\eqref{eq:frac-sphere} as
\[
  \mathscr{P}^t_g u(z)
  = c_{N,t}\,\I_t(z) + A_{N,t}\,u(z),
  \quad \text{where}\quad
  \I_t(z):=\mathrm{p.v.}\!\int_{\mathbb{S}^N}
            \frac{u(z)-u(\zeta)}{|z-\zeta|^{N+2t}}\,dV_g(\zeta),
\]
we compute the derivative at $t=s$ as
\[
  \mathscr{P}^{s+\ln}_g u(z)
  = \left.\frac{d}{dt}\right|_{t=s}\mathscr{P}^t_g u(z)
  = c'_{N,s}\,\I_s(z) + c_{N,s}\,\I'_s(z) + A'_{N,s}\,u(z).
\]

We first compute $\I'_s(z)$. Differentiating under the integral sign,
formally we obtain
\begin{align*}
  \I'_s(z)
  &= \mathrm{p.v.}\!\int_{\mathbb{S}^N}
      (u(z)-u(\zeta))\,\frac{d}{dt}\Bigl(
        |z-\zeta|^{-N-2t}\Bigr)\Big|_{t=s}\,dV_g(\zeta)
 \\[1mm]& = \mathrm{p.v.}\!\int_{\mathbb{S}^N}
      \frac{u(z)-u(\zeta)}{|z-\zeta|^{N+2s}}\,
      \bigl(-2\ln|z-\zeta|\bigr)\,dV_g(\zeta).
\end{align*}
Next we justify the differentiation under the integral sign. Since $u\in C^\beta(\mathbb{S}^N)$ and
$\beta>2s$, we have the Hölder estimate
\[
  |u(z)-u(\zeta)|
  \le C\,|z-\zeta|^\beta,
\]
so that near the diagonal the integrand is bounded in absolute value by
\[
  C\,|z-\zeta|^{\beta-N-2s}\,|\ln|z-\zeta||.
\]
In local geodesic polar coordinates around $z$ this behaves like
\(
  |\ln r|\,r^{\beta-2s-1}
\),
which is integrable at $r=0$ because $\beta-2s>0$. This shows that
$|\I'_s(z)|<\infty$ and that differentiation under the integral sign is
licensed by the dominated convergence theorem.

Substituting the expression for $I'_s$ into the formula for
$\mathscr{P}^{s+\ln}_g u(z)$, we obtain
\[
  \mathscr{P}^{s+\ln}_g u(z)
  = c_{N,s}\,\mathrm{p.v.}\!\int_{\mathbb{S}^N}
      \frac{u(z)-u(\zeta)}{|z-\zeta|^{N+2s}}
      \bigl(-2\ln|z-\zeta|\bigr)\,dV_g(\zeta)
    + c'_{N,s}\,\I_s(z) + A'_{N,s}\,u(z).
\]
Since
\[
  \I_s(z)
  = \mathrm{p.v.}\!\int_{\mathbb{S}^N}
      \frac{u(z)-u(\zeta)}{|z-\zeta|^{N+2s}}\,dV_g(\zeta),
\]
we can absorb the term $c'_{N,s}\I_s(z)$ into the kernel by writing
\[
  c'_{N,s}\,\I_s(z)
  = c_{N,s}\,\mathrm{p.v.}\!\int_{\mathbb{S}^N}
      \frac{u(z)-u(\zeta)}{|z-\zeta|^{N+2s}}\,
      \frac{c'_{N,s}}{c_{N,s}}\,dV_g(\zeta)
  = c_{N,s}\,\mathrm{p.v.}\!\int_{\mathbb{S}^N}
      \frac{u(z)-u(\zeta)}{|z-\zeta|^{N+2s}}\,b_{N,s}\,dV_g(\zeta),
\]
where we set $b_{N,s}:=c'_{N,s}/c_{N,s}$. Combining the two integral terms we arrive at~\eqref{eq:frac-log-kernel}.

It remains to compute $b_{N,s}$ and $A'_{N,s}$. From the definition
\[
  c_{N,s}
  = 4^s\pi^{-N/2}
    \frac{\Gamma\!\left(\tfrac N2+s\right)}{\Gamma(2-s)}\,s(1-s),
\]
we obtain
\[
  \ln c_{N,s}
  = s\ln 4 - \frac N2\ln\pi
    + \ln\Gamma\!\left(\tfrac N2+s\right)
    - \ln\Gamma(2-s)
    + \ln s + \ln(1-s).
\]
Differentiating with respect to $s$ gives
\[
  \frac{c'_{N,s}}{c_{N,s}}
  = \ln 4
    + \psi\!\Bigl(\tfrac N2+s\Bigr)
    + \psi(2-s)
    + \frac{1}{s}
    - \frac{1}{1-s},
\]
which is~\eqref{eq:frac-log-constants}. Finally, since
\[
  A_{N,s}
  =\frac{\Gamma\!\left(\tfrac N2+s\right)}
          {\Gamma\!\left(\tfrac N2-s\right)},
\]
we obtain
\[
  \ln A_{N,s}
  =\ln\Gamma\!\left(\tfrac N2+s\right)
   -\ln\Gamma\!\left(\tfrac N2-s\right),
\]
and therefore
\[
  \frac{A'_{N,s}}{A_{N,s}}
  = \psi\!\Bigl(\tfrac N2+s\Bigr)
    + \psi\!\Bigl(\tfrac N2-s\Bigr).
\]
This yields~\eqref{eq:A-Ns-prime}.\smallskip

$\textbf{(ii)}$ 
First, by the assumption $\beta>2s$ and the H\"older continuity of $u$
we have
\[
  |u(z)-u(\zeta)|
  \le C\,|z-\zeta|^\beta,
\]
so that the kernel in~\eqref{eq:frac-log-kernel} is bounded in absolute
value by
\[
  C\,|z-\zeta|^{\beta-N-2s}\bigl(1+|\ln|z-\zeta||\bigr).
\]
As observed in the proof of \textbf{(i)}, this is integrable near the
diagonal because $\beta-2s>0$, while away from the diagonal the kernel
is smooth and bounded. This implies that
$\mathscr{P}^{s+\ln}_g u(z)$ is well-defined for every $z\in\mathbb{S}^N$ and
that $\mathscr{P}^{s+\ln}_g u\in L^\infty(\mathbb{S}^N)$.
Since $\mathbb{S}^N$ is compact, the inclusion
$L^\infty(\mathbb{S}^N)\hookrightarrow L^p(\mathbb{S}^N)$ is continuous
for all $1\le p<\infty$, which yields $\mathscr{P}^{s+\ln}_g u\in L^p$
for all $1\le p\le\infty$.

To prove the convergence of the difference quotients in \(C(\mathbb S^N)\), i.e. uniformly on \(\mathbb S^N\), fix \(\varepsilon>0\) such that
\[
(s-\varepsilon,s+\varepsilon)\subset(0,1),
\qquad s+\varepsilon<\beta/2.
\]
Write
\[
\mathscr P_g^t u = c_{N,t}\,f_t + A_{N,t}u,
\qquad
f_t(z):=\mathrm{p.v.}\!\int_{\mathbb S^N}
\frac{u(z)-u(\zeta)}{|z-\zeta|^{N+2t}}\,dV_g(\zeta).
\]

For \(z,\zeta\in\mathbb S^N\), set
\[
K_t(z,\zeta):=|z-\zeta|^{-N-2t}.
\]
Then
\[
\partial_t K_t(z,\zeta)
=-2\ln|z-\zeta|\,|z-\zeta|^{-N-2t},
\]
and
\[
\partial_t^2 K_t(z,\zeta)
=4(\ln|z-\zeta|)^2\,|z-\zeta|^{-N-2t}.
\]
By Taylor's formula in \(t\), for \(t\in(s-\varepsilon,s+\varepsilon)\),
\[
\left|
\frac{K_t(z,\zeta)-K_s(z,\zeta)}{t-s}
-\partial_t K_s(z,\zeta)
\right|
\le
C|t-s|\,(\ln|z-\zeta|)^2\,|z-\zeta|^{-N-2(s+\varepsilon)},
\]
where \(C>0\) is independent of \(z,\zeta,t\).

Since \(u\in C^\beta(\mathbb S^N)\), we have
\[
|u(z)-u(\zeta)|\le C\,|z-\zeta|^\beta.
\]
Hence
\[
\begin{aligned}
&|u(z)-u(\zeta)|\,
\left|
\frac{K_t(z,\zeta)-K_s(z,\zeta)}{t-s}
-\partial_t K_s(z,\zeta)
\right|  \le
C|t-s|\,|z-\zeta|^{\beta-N-2(s+\varepsilon)}(\ln|z-\zeta|)^2.
\end{aligned}
\]
Therefore,
\[
\left|
\frac{f_t(z)-f_s(z)}{t-s}
+2\,\mathrm{p.v.}\!\int_{\mathbb S^N}
\frac{u(z)-u(\zeta)}{|z-\zeta|^{N+2s}}
\ln|z-\zeta|\,dV_g(\zeta)
\right|
\le C|t-s|\,J(z),
\]
where
\[
J(z):=\int_{\mathbb S^N}
|z-\zeta|^{\beta-N-2(s+\varepsilon)}(\ln|z-\zeta|)^2\,dV_g(\zeta).
\]

We claim that \(\sup_{z\in\mathbb S^N}J(z)<\infty\). Indeed, near \(\zeta=z\), in local geodesic polar coordinates \((r,\omega)\) centered at \(z\), one has
\[
dV_g(\zeta)\sim r^{N-1}\,dr\,d\omega,
\qquad |z-\zeta|\sim r,
\]
so the radial part of \(J(z)\) behaves like
\[
\int_0^1 r^{\beta-2(s+\varepsilon)-1}|\ln r|^2\,dr,
\]
which is finite since \(\beta>2(s+\varepsilon)\). Away from the diagonal, the integrand is continuous and bounded. By compactness of \(\mathbb S^N\), this yields
\[
\sup_{z\in\mathbb S^N}J(z)<\infty.
\]

Consequently, we conclude that 
\[
\left\|
\frac{f_t-f_s}{t-s}
-\partial_t f_t\big|_{t=s}
\right\|_{C(\mathbb S^N)}
\le C|t-s|\,\sup_{z\in\mathbb S^N}J(z)
\longrightarrow 0
\quad \text{as}\quad t\to s,
\]
where
\[
\partial_t f_t\big|_{t=s}(z)
=
-2\,\mathrm{p.v.}\!\int_{\mathbb S^N}
\frac{u(z)-u(\zeta)}{|z-\zeta|^{N+2s}}
\ln|z-\zeta|\,dV_g(\zeta).
\]
Thus \(t\mapsto f_t\) is differentiable in \(C(\mathbb S^N)\) at \(t=s\).

Since \(t\mapsto c_{N,t}\) and \(t\mapsto A_{N,t}\) are smooth, it follows that
\(t\mapsto \mathscr P_g^t u\in C(\mathbb S^N)\) is \(C^1\) on \((s-\varepsilon,s+\varepsilon)\), and
\[
\left.\frac{d}{dt}\mathscr P_g^t u\right|_{t=s}
=
\mathscr P_g^{s+\ln}u.
\]
Therefore,
\[
\frac{\mathscr P_g^t u-\mathscr P_g^s u}{t-s}
\longrightarrow
\mathscr P_g^{s+\ln}u
\quad\text{in }C(\mathbb S^N)\quad \text{as}\quad t\to s.
\]
Since \(\mathbb S^N\) has finite volume, this also implies convergence in \(L^p(\mathbb S^N)\) for every \(1\le p\le+\infty\).

$\textbf{(iii)}$ Let
\[
\cJ_su(z):=\mathrm{p.v.}\!\int_{\mathbb S^N}
\frac{u(z)-u(\zeta)}{|z-\zeta|^{N+2s}}\,dV_g(\zeta),
\qquad
\cL_su(z):=\mathrm{p.v.}\!\int_{\mathbb S^N}
\frac{u(z)-u(\zeta)}{|z-\zeta|^{N+2s}}\ln|z-\zeta|\,dV_g(\zeta).
\]
Then, by \eqref{eq:frac-log-kernel},
\[
\mathscr P_g^{s+\ln}u
=
-2c_{N,s}\cL_su + c_{N,s}b_{N,s}\cJ_su + A'_{N,s}u.
\]
We will show that, uniformly on \(\mathbb S^N\),
\[
-2c_{N,s}\cL_su\to 0,\qquad
c_{N,s}b_{N,s}\cJ_su\to c_N\cJ_0u,\qquad
A'_{N,s}u\to A_Nu,
\]
as \(s\to0^+\). This yields \(\mathscr P_g^{s+\ln}u\to \mathscr P_g^{\ln}u\) by \eqref{eq:log-def}.

\medskip
\noindent\textbf{  Uniform integrability bounds.}
Since \(u\in C^\beta(\mathbb S^N)\), there exists \(C>0\) such that
\[
|u(z)-u(\zeta)|\le C\,d_g(z,\zeta)^\beta\le C'|z-\zeta|^\beta,
\qquad \forall z,\zeta\in\mathbb S^N,
\]
where we used the equivalence of geodesic and chordal distances on the compact sphere.

Fix \(s_0\in(0,\beta/2)\). For \(0<s\le s_0\), we have
\[
\frac{|u(z)-u(\zeta)|}{|z-\zeta|^{N+2s}}
\le C |z-\zeta|^{\beta-N-2s_0},
\]
and
\[
\frac{|u(z)-u(\zeta)|\,|\ln|z-\zeta||}{|z-\zeta|^{N+2s}}
\le C |z-\zeta|^{\beta-N-2s_0}|\ln|z-\zeta||.
\]
In local geodesic polar coordinates around \(z\), the corresponding radial behaviors are
\[
r^{\beta-2s_0-1}
\quad\text{and}\quad
|\ln r|\,r^{\beta-2s_0-1},
\]
which are integrable near \(r=0\) because \(\beta-2s_0>0\). Away from the diagonal, the kernels are bounded. Hence
\[
\sup_{0<s\le s_0}\|\cJ_su\|_{L^\infty(\mathbb S^N)}<\infty,
\qquad
\sup_{0<s\le s_0}\|\cL_su\|_{L^\infty(\mathbb S^N)}<\infty.
\]
In particular, for \(0<s<\beta/2\), the principal values are actually absolutely convergent.

\medskip
\noindent\textbf{  The logarithmic kernel part vanishes.}
Since
\[
c_{N,s}=4^s\pi^{-N/2}s\,\frac{\Gamma(\frac N2+s)}{\Gamma(1-s)}
\longrightarrow 0
\quad{\rm as}\ \,  s\to0^+,
\]
and \(\|\cL_su\|_\infty\) is uniformly bounded for \(0<s\le s_0\), we obtain 
\[
\|-2c_{N,s}\cL_su\|_{L^\infty(\mathbb S^N)}\le 2c_{N,s}\|\cL_su\|_\infty\longrightarrow 0\quad {\rm as}\ \, s\to0^+.
\]

\medskip
\noindent\textbf{ Convergence of the principal kernel term.}
Set
\[
\alpha_s:=c_{N,s}b_{N,s}.
\]
From \eqref{eq:frac-log-constants} and the explicit formula for \(c_{N,s}\), one has
\[
c_{N,s}=c_N\,s+O(s^2),
\qquad
b_{N,s}=\frac1s+O(1),
\]
hence
\[
\alpha_s=c_{N,s}b_{N,s}\longrightarrow c_N
\quad{\rm as}\ \,  s\to0^+.
\]

Now we claim that
\[
\cJ_su\longrightarrow \cJ_0u
\quad\text{uniformly on }\mathbb S^N,
\]
where
\[
\cJ_0u(z):=\int_{\mathbb S^N}\frac{u(z)-u(\zeta)}{|z-\zeta|^N}\,dV_g(\zeta).
\]
Indeed, for \(0<s\le s_0\), by the mean value theorem in the parameter \(t\mapsto |z-\zeta|^{-N-2t}\),
\[
\bigl||z-\zeta|^{-N-2s}-|z-\zeta|^{-N}\bigr|
\le 2s\,|\ln|z-\zeta||\,|z-\zeta|^{-N-2s_0}.
\]
Therefore,
\[
\begin{aligned}
|\cJ_su(z)-\cJ_0u(z)|
&\le \int_{\mathbb S^N}|u(z)-u(\zeta)|\,
\bigl||z-\zeta|^{-N-2s}-|z-\zeta|^{-N}\bigr|\,dV_g(\zeta) \\
&\le C s \int_{\mathbb S^N}
|z-\zeta|^{\beta-N-2s_0}|\ln|z-\zeta||\,dV_g(\zeta).
\end{aligned}
\]
The last integral is uniformly bounded in \(z\) by the same integrability argument as in (i), so
\[
\|\cJ_su-\cJ_0u\|_{L^\infty(\mathbb S^N)}\le Cs .
\]
Hence
\[
\begin{aligned}
\|\alpha_s\cJ_su-c_N\cJ_0u\|_\infty
&\le |\alpha_s-c_N|\,\|\cJ_su\|_\infty + c_N\|\cJ_su-\cJ_0u\|_\infty \longrightarrow 0\quad {\rm as}\ \, s\to0^+.
\end{aligned}
\]

\medskip
\noindent\textbf{Convergence of the zero-order term.}
By \eqref{eq:A-Ns-prime},
\[
A'_{N,s}
=
A_{N,s}\Bigl(\psi\!\bigl(\tfrac N2+s\bigr)+\psi\!\bigl(\tfrac N2-s\bigr)\Bigr).
\]
Since \(A_{N,s}\to 1\) as \(s\to0^+\), and \(\psi\) is continuous on \((0,\infty)\), we get
\[
A'_{N,s}\longrightarrow 2\psi\!\left(\tfrac N2\right)=A_N.
\]
Therefore,
\[
\|A'_{N,s}u-A_Nu\|_\infty
\le |A'_{N,s}-A_N|\,\|u\|_\infty
\longrightarrow 0.
\]

As a consequence, we derive that
\[
\mathscr P_g^{s+\ln}u \longrightarrow c_N\cJ_0u+A_Nu=\mathscr P_g^{\ln}u
\quad\text{uniformly on }\mathbb S^N,
\]
as \(s\to0^+\), proving the first claim. Since \(\mathbb S^N\) has finite volume, uniform convergence implies convergence in \(L^p(\mathbb S^N)\) for every \(1\le p\le\infty\). This proves the second claim. 
\hfill$\Box$\bigskip

Before proving Proposition~\ref{prop:Dini-frac-log}, we note that the key point is to control the
logarithmic singularity in the kernel by a suitable Dini-type modulus of continuity, which ensures
the principal value integral is well-defined and yields the desired continuity estimates.

\begin{proof}[\textbf{Proof of Proposition \ref{prop:Dini-frac-log}}]
We use the kernel representation \eqref{eq:frac-log-kernel}:
\[
  \mathscr{P}^{s+\ln}_g u(z)
  = c_{N,s}\,\mathrm{p.v.}\!\int_{\mathbb{S}^N}
      \frac{u(z)-u(\zeta)}{|z-\zeta|^{N+2s}}
      \bigl(-2\ln|z-\zeta|+b_{N,s}\bigr)\,dV_g(\zeta)
    + A'_{N,s}\,u(z).
\]
Since \(s\in(0,1)\) is fixed,  the main issue is the singular integral term.

Throughout the proof we use the standard local equivalence, uniform on the compact sphere,
\begin{equation}\label{eq:dg-chordal-equiv}
  C_1\,d_g(z,\zeta)\le |z-\zeta|\le C_2\,d_g(z,\zeta),
  \qquad z,\zeta\in\mathbb S^N,
\end{equation}
and in geodesic polar coordinates \((r,\theta)\) centered at \(z\),
\begin{equation}\label{eq:polar-volume}
  dV_g(\zeta)=J_z(r,\theta)\,r^{N-1}\,dr\,d\theta,
\end{equation}
with \(J_z\) smooth and uniformly bounded above and below for \(0<r<1\).

\medskip
\noindent\textbf{(i) Pointwise well-definedness under local Dini condition.}

Fix \(z\in\mathbb S^N\), and assume \(u\) is fractional--logarithmic Dini continuous of order \(s\) at \(z\), i.e.
\[
  \int_0^1 \frac{\omega_{u,z}(r)}{r^{1+2s}}(1+|\ln r|)\,dr<\infty.
\]
We show that the singular integral at \(z\) is absolutely convergent, hence the principal value is well defined.

Split
\[
\int_{\mathbb S^N} = \int_{d_g(z,\zeta)\le 1} + \int_{d_g(z,\zeta)>1}.
\]
For the far part \(d_g(z,\zeta)>1\), the kernel is bounded since \(|z-\zeta|\ge c>0\), so
\[
\left|
\frac{u(z)-u(\zeta)}{|z-\zeta|^{N+2s}}
\bigl(-2\ln|z-\zeta|+b_{N,s}\bigr)
\right|
\le C\bigl(|u(z)|+|u(\zeta)|\bigr).
\]
Because \(u\in L^1(\mathbb S^N)\) and \(\mathbb S^N\) has finite volume, this part is finite.

For the near part \(d_g(z,\zeta)\le 1\), using \eqref{eq:dg-chordal-equiv},
\[
|u(z)-u(\zeta)|\le \omega_{u,z}(d_g(z,\zeta))
\le \omega_{u,z}(C|z-\zeta|),
\]
and
\[
\frac{| -2\ln|z-\zeta|+b_{N,s}|}{|z-\zeta|^{N+2s}}
\le C\frac{1+|\ln d_g(z,\zeta)|}{d_g(z,\zeta)^{N+2s}}.
\]
Hence
\[
\begin{aligned}
&\int_{d_g(z,\zeta)\le 1}
\frac{|u(z)-u(\zeta)|}{|z-\zeta|^{N+2s}}
\bigl| -2\ln|z-\zeta|+b_{N,s}\bigr|\,dV_g(\zeta)\\
&\qquad\le
C\int_{d_g(z,\zeta)\le 1}
\frac{\omega_{u,z}(d_g(z,\zeta))}{d_g(z,\zeta)^{N+2s}}
\bigl(1+|\ln d_g(z,\zeta)|\bigr)\,dV_g(\zeta).
\end{aligned}
\]
Passing to geodesic polar coordinates around \(z\) and using \eqref{eq:polar-volume}, the singular integral is absolutely convergent at \(z\), and \(\mathscr P_g^{s+\ln}u(z)\) is well defined.

\medskip
\noindent\textbf{(ii) Continuity under uniform Dini condition.}

Assume now that \(u\) is uniformly fractional--logarithmic Dini continuous of order \(s\), i.e.
\[
  \int_0^1 \frac{\omega_u(r)}{r^{1+2s}}(1+|\ln r|)\,dr<\infty,
  \qquad
  \omega_u(r):=\sup_{x\in\mathbb S^N}\omega_{u,x}(r).
\]
We prove that \(\mathscr P_g^{s+\ln} u\) is continuous on \(\mathbb S^N\). Let
\[
Ku(z,\zeta):=
\frac{u(z)-u(\zeta)}{|z-\zeta|^{N+2s}}
\bigl(-2\ln|z-\zeta|+b_{N,s}\bigr),
\]
so that
\[
\mathscr P_g^{s+\ln}u(z)=c_{N,s}\int_{\mathbb S^N}Ku(z,\zeta)\,dV_g(\zeta)+A'_{N,s}u(z),
\]
where the integral is absolutely convergent by part (i). It remains to prove continuity of the integral term
\[
I(z):=\int_{\mathbb S^N}Ku(z,\zeta)\,dV_g(\zeta).
\]

Fix \(z_0\in\mathbb S^N\), and let \(z\to z_0\). We split
\[
I(z)-I(z_0)=\int_{B_\rho(z_0)}\!\!\!Ku(z,\zeta)\,dV_g(\zeta)
-\int_{B_\rho(z_0)}\!\!\!Ku(z_0,\zeta)\,dV_g(\zeta)
+\int_{\mathbb S^N\setminus B_\rho(z_0)}\!\!\!\!\!\!\!\!\!\!\!\!\!\bigl(Ku(z,\zeta)-Ku(z_0,\zeta)\bigr)\,dV_g(\zeta)
\]
for some small \(\rho\in(0,1/4)\), where \(B_\rho(z_0):=\{\zeta:d_g(z_0,\zeta)<\rho\}\).

We estimate the near and far parts separately.

\smallskip
\noindent\textbf{Near part.}
If $d_g(z,z_0)<\rho/2$, then
\(B_\rho(z_0)\subset B_{2\rho}(z)\). Thus, by the same estimate as in part (i) and the uniform modulus \(\omega_u\),
\[
\int_{B_\rho(z_0)} |Ku(z,\zeta)|\,dV_g(\zeta)
\le C\int_0^{2\rho}\frac{\omega_u(r)}{r^{1+2s}}(1+|\ln r|)\,dr.
\]
Similarly,
\[
\int_{B_\rho(z_0)} |Ku(z_0,\zeta)|\,dV_g(\zeta)
\le C\int_0^{\rho}\frac{\omega_u(r)}{r^{1+2s}}(1+|\ln r|)\,dr.
\]
Hence
\[
\left|\int_{B_\rho(z_0)}Ku(z,\zeta)\,dV_g(\zeta)
-\int_{B_\rho(z_0)}Ku(z_0,\zeta)\,dV_g(\zeta)\right|
\le C\int_0^{2\rho}\frac{\omega_u(r)}{r^{1+2s}}(1+|\ln r|)\,dr.
\]
By the uniform Dini condition, the right-hand side tends to \(0\) as \(\rho\downarrow0\), uniformly in \(z\) near \(z_0\).

\smallskip
\noindent\textbf{Far part.}
On \(\mathbb S^N\setminus B_\rho(z_0)\), if \(z\) is close to \(z_0\) then \(d_g(z,\zeta)\ge \rho/2\), so the kernel factor
\[
(z,\zeta)\mapsto \frac{-2\ln|z-\zeta|+b_{N,s}}{|z-\zeta|^{N+2s}}
\]
is smooth and uniformly bounded there. Also \(u\) is continuous, hence \(u(z)\to u(z_0)\). Therefore for each fixed \(\zeta\in \mathbb S^N\setminus B_\rho(z_0)\),
\[
Ku(z,\zeta)\to Ku(z_0,\zeta)\qquad z\to z_0.
\]
Moreover, for \(z\) near \(z_0\),
\[
|Ku(z,\zeta)-Ku(z_0,\zeta)|
\le C_\rho\bigl(|u(z)|+|u(z_0)|+|u(\zeta)|\bigr)
\le C_\rho\bigl(1+|u(\zeta)|\bigr),
\]
since \(u\) is continuous on the compact sphere and hence bounded. The right-hand side is integrable on \(\mathbb S^N\setminus B_\rho(z_0)\). Thus, by dominated convergence,
\[
\int_{\mathbb S^N\setminus B_\rho(z_0)}\bigl(Ku(z,\zeta)-Ku(z_0,\zeta)\bigr)\,dV_g(\zeta)\to 0
\quad{\rm as}\ \,  z\to z_0,
\]
for each fixed \(\rho>0\).

Combining the near- and far-part estimates: first choose \(\rho>0\) so small that the near-part bound is less than \(\varepsilon\), then choose \(z\) close enough to \(z_0\) so that the far part is less than \(\varepsilon\). This proves \(I(z)\to I(z_0)\), hence \(I\in C(\mathbb S^N)\). Therefore \(\mathscr P_g^{s+\ln}u=c_{N,s}I+A'_{N,s}u\) is continuous on \(\mathbb S^N\), as claimed.
\end{proof}

\subsection{Eigenvalues and Eigenfunctions}

In this subsection we describe the spectral structure of $\mathscr P^{s+\ln}_g$ on the round sphere
$(\mathbb S^N,g)$. Since $\mathscr P^{s+\ln}_g$ is obtained by differentiating the conformal
fractional Laplacian $\mathscr P_g^{t}$ with respect to the order, it remains diagonal with respect
to the spherical harmonic decomposition. We first derive the explicit eigenvalue formula in terms
of the Laplace--Beltrami eigenvalues $\lambda_k=k(k+N-1)$ and the digamma function, and then prove
that every $C^\beta$-eigenfunction of $\mathscr P^{s+\ln}_g$ must belong to a unique spherical
harmonic subspace $\mathcal H_k$. A further key ingredient for the classification is the strict
monotonicity of the discrete sequence $\{\varphi^{s+\ln}_N(\lambda_k)\}_{k\ge0}$, established in
Proposition~\ref{pr spc func}.

 From the property of the digamma function,   there exists a unique value $a_0\in(1,2)$ such that 
$$\psi(a_0+1)  +\psi(a_0-1)=0. $$ 
Direct computations shows that 
$$a_0\approx 1.8473. $$ 
There exists a unique value $a_1\in(\frac12,2)$ such that 
$$\psi(a_1+\frac12)  +\psi(a_1-\frac12)=0. $$ 
Direct computation implies that 
$$a_1\approx 1.5703. $$

\begin{lemma}\label{lm es-21}
Let $N\ge 1$, $s\in(0,1)$ and $N>2s$. Define, for $\lambda\ge 0$,
\begin{equation}\label{eq:phi0-def-clean}
\phi_0(s,N;\lambda):=
\psi\!\Bigl(\tfrac12+s+\sqrt{\lambda+\tfrac14(N-1)^2}\Bigr)
+\psi\!\Bigl(\tfrac12-s+\sqrt{\lambda+\tfrac14(N-1)^2}\Bigr),
\end{equation}
where $\psi$ is the digamma function. Then:

\smallskip\noindent\textup{(i)}
The map $\lambda\mapsto \phi_0(s,N;\lambda)$ is strictly increasing on $(0,\infty)$.

\smallskip\noindent\textup{(ii)} 
Let $a_0\in(1,2)$ be the unique number such that
\begin{equation}\label{eq:a0-def-clean}
\psi(a_0+1)+\psi(a_0-1)=0.
\end{equation}
For every $N\ge 2$ one has
\begin{equation}\label{eq:phi0-large-lambda-clean}
\phi_0(s,N;\lambda)>0
\qquad\text{whenever}\qquad
\lambda>\Lambda_N:=\Bigl(a_0-\tfrac12\Bigr)^2-\tfrac14(N-1)^2.
\end{equation}
In particular, if $N\ge 4$ then $\Lambda_N<0$ and hence $\phi_0(s,N;\lambda)>0$ for all $\lambda\ge 0$.

Moreover, for $\lambda=0$ the sign is as follows:
\begin{itemize}
\item[(a)] If $N=2$, then $\phi_0(s,2;0)<0$ for all $s\in(0,1)$.
\item[(b)] If $N=3$, then there exists a unique $s_0\in(0,1)$ such that
\[
\phi_0(s_0,3;0)=0,\qquad
\phi_0(s,3;0)>0\ \text{for }s\in(0,s_0),\qquad
\phi_0(s,3;0)<0\ \text{for }s\in(s_0,1).
\]
\end{itemize}

\smallskip\noindent\textup{(iii)}
Let $a_1\in(\tfrac12,2)$ be the unique number such that
\begin{equation}\label{eq:a1-def-clean}
\psi\Bigl(a_1+\tfrac12\Bigr)+\psi\Bigl(a_1-\tfrac12\Bigr)=0.
\end{equation}
When $N=1$, then $s\in(0,\tfrac12)$, one has $\phi_0(s,1;0)<0$ and
\[
\phi_0(s,1;\lambda)>0\qquad\text{for all }\lambda>\Lambda_1:=\Bigl(a_1-\tfrac12\Bigr)^2.
\]
Furthermore, the map $s\mapsto \phi_0(s,1;1)$ is strictly decreasing on $(0,\tfrac12)$; hence there exists a unique
$s_1\in(0,\tfrac12)$ such that
\[
\phi_0(s_1,1;1)=0,\qquad
\phi_0(s,1;1)>0\ \text{for }s\in(0,s_1),\qquad
\phi_0(s,1;1)<0\ \text{for }s\in(s_1,\tfrac12).
\]
\end{lemma}

\begin{proof}
Set
\[
a(\lambda):=\sqrt{\lambda+\tfrac14(N-1)^2},\qquad
A(\lambda):=\tfrac12+a(\lambda),
\]
so that
\begin{equation}\label{eq:phi0-A-clean}
\phi_0(s,N;\lambda)=\psi\bigl(A(\lambda)+s\bigr)+\psi\bigl(A(\lambda)-s\bigr).
\end{equation}
Since $N>2s$, we have $A(\lambda)-s\ge \tfrac N2-s>0$, hence all terms are well-defined for $\lambda\ge 0$.

\medskip\noindent\textup{(i)}
For $\lambda>0$ we have $a(\lambda)>0$ and $A'(\lambda)=\frac1{2a(\lambda)}$. Differentiating
\eqref{eq:phi0-A-clean} yields
\[
\partial_\lambda\phi_0(s,N;\lambda)
=A'(\lambda)\Bigl[\psi'\bigl(A(\lambda)+s\bigr)+\psi'\bigl(A(\lambda)-s\bigr)\Bigr]
=\frac{1}{2a(\lambda)}\Bigl[\psi'\bigl(A+s\bigr)+\psi'\bigl(A-s\bigr)\Bigr]>0,
\]
because $\psi'(t)>0$ for $t>0$.

\medskip\noindent\textup{(ii)}
Fix $a>1$ and define $g_a(s):=\psi(a+s)+\psi(a-s)$ for $s\in(0,1]$. Then
\[
g_a'(s)=\psi'(a+s)-\psi'(a-s)<0,
\]
since the trigamma $\psi'$ is strictly decreasing on $(0,\infty)$ and $a+s>a-s>0$.
Thus $g_a$ is strictly decreasing and $g_a(s)\ge g_a(1)=\psi(a+1)+\psi(a-1)$ for $s\in(0,1]$.
Let $F(a):=\psi(a+1)+\psi(a-1)$ for $a>1$. Then $F'(a)=\psi'(a+1)+\psi'(a-1)>0$, hence $F$ is strictly increasing.
By \eqref{eq:a0-def-clean}, if $a>a_0$ then $F(a)>0$, and therefore $g_a(s)>0$.

Now assume $N\ge 2$ and $\lambda>\Lambda_N$ with $\Lambda_N$ as in \eqref{eq:phi0-large-lambda-clean}. Then
$A(\lambda)=\tfrac12+\sqrt{\lambda+\tfrac14(N-1)^2}>a_0$, hence
$\phi_0(s,N;\lambda)=g_{A(\lambda)}(s)>0$, proving \eqref{eq:phi0-large-lambda-clean}.
If $N\ge 4$, then $\Lambda_N<0$, so the conclusion holds for all $\lambda\ge 0$.

For $\lambda=0$, we have $A(0)=\tfrac N2$ and $\phi_0(s,N;0)=g_{N/2}(s)$.
If $N=2$, then $g_{1}(0)=2\psi(1)<0$ and since $g_1$ is strictly decreasing, $g_1(s)<0$ for all $s\in(0,1)$.
If $N=3$, then $g_{3/2}$ is strictly decreasing on $(0,1)$ and
\[
\lim_{s\downarrow 0} g_{3/2}(s)=2\psi\Bigl(\tfrac32\Bigr)>0,
\qquad
\lim_{s\uparrow 1} g_{3/2}(s)=\psi\Bigl(\tfrac52\Bigr)+\psi\Bigl(\tfrac12\Bigr)<0.
\]
Hence, by the intermediate value theorem and strict monotonicity, there exists a unique
$s_0\in(0,1)$ such that $g_{3/2}(s_0)=0$, and the stated sign property follows.

\medskip\noindent\textup{(iii)}
Let $N=1$ and $s\in(0,\tfrac12)$. Then $\tfrac12\pm s\in(0,1)$, and since $\psi<0$ on $(0,1]$, we get
$\phi_0(s,1;0)=\psi(\tfrac12+s)+\psi(\tfrac12-s)<0$.
If $\lambda>\Lambda_1:=(a_1-\tfrac12)^2$, then $A(\lambda)=\tfrac12+\sqrt{\lambda}>a_1$, and the same argument as above
(using \eqref{eq:a1-def-clean}) yields $\phi_0(s,1;\lambda)>0$.

Finally, for $\lambda=1$ we have $A(1)=\tfrac32$ and $\phi_0(s,1;1)=g_{3/2}(s)$ for $s\in(0,\tfrac12)$.
Since $g_{3/2}$ is strictly decreasing and
\[
g_{3/2}(0)=2\psi\Bigl(\tfrac32\Bigr)>0,\qquad
g_{3/2}\Bigl(\tfrac12\Bigr)=\psi(2)+\psi(1)<0,
\]
there exists a unique $s_1\in(0,\tfrac12)$ such that $g_{3/2}(s_1)=0$, and the claimed sign property follows.
\end{proof}

\begin{proposition}\label{pr spc func}
Let $N\ge 1$ be an integer, $s\in(0,1)$ and $N>2s$. Recall
\[
\lambda_k:=k(k+N-1),\qquad k\in\N_0.
\]
Then the mapping $k\mapsto \varphi^{s+\ln}_N(\lambda_k)$ is strictly increasing.
\end{proposition}

\begin{proof}
We split the proof into three cases according to the dimension:
\[
\text{(i) } N\ge 4,\qquad \text{(ii) } N=2,3,\qquad \text{(iii) } N=1.
\]

Recall that, for $\lambda\ge 0$,
\[
\varphi^{s+\ln}_N(\lambda)
=\varphi_{N,s}(\lambda)\,\phi_0(s,N;\lambda),
\]
where
\[
\varphi_{N,s}(\lambda)
=
\frac{\Gamma\!\big(\frac12+s+a(\lambda)\big)}
{\Gamma\!\big(\frac12-s+a(\lambda)\big)},
\qquad
\phi_0(s,N;\lambda)
=
\psi\!\big(\tfrac12+s+a(\lambda)\big)+\psi\!\big(\tfrac12-s+a(\lambda)\big),
\]
and $a(\lambda):=\sqrt{\lambda+\frac{(N-1)^2}{4}}$.
Under $N>2s$, both arguments of $\Gamma$ and $\psi$ are positive for every $\lambda\ge 0$, hence
\begin{equation}\label{eq:phi-varphi-pos}
\varphi_{N,s}(\lambda)>0 \quad{\rm for}\ \,  \lambda\ge 0.
\end{equation}
Moreover, $\lambda\mapsto \varphi_{N,s}(\lambda)$ is strictly increasing on $(0,\infty)$, since
\[
\frac{d}{d\lambda}\log\varphi_{N,s}(\lambda)
=\frac{1}{2a(\lambda)}\Bigl[\psi\!\big(\tfrac12+s+a(\lambda)\big)-\psi\!\big(\tfrac12-s+a(\lambda)\big)\Bigr]>0,
\qquad \lambda>0,
\]
where we used that $\psi$ is strictly increasing.
In addition, Lemma~\ref{lm es-21} (i) gives $\partial_\lambda\phi_0(s,N;\lambda)>0$ for $\lambda>0$.

Consequently, on any interval $I\subset(0,\infty)$ where $\phi_0(s,N;\lambda)\ge 0$ for all $\lambda\in I$, we have
\[
\frac{d}{d\lambda}\varphi^{s+\ln}_N(\lambda)
=\phi_0\,\varphi_{N,s}'+\varphi_{N,s}\,\phi_0'>0\qquad \text{on }I,
\]
because $\varphi_{N,s}>0$, $\varphi_{N,s}'>0$, and $\phi_0'>0$ there. Hence $\varphi^{s+\ln}_N$ is strictly increasing on $I$.

\medskip\noindent\textbf{(i) $N\ge 4$.}
By Lemma~\ref{lm es-21} (ii), $\phi_0(s,N;\lambda)>0$ for all $\lambda\ge 0$. Therefore $\varphi^{s+\ln}_N$ is strictly
increasing on $(0,\infty)$. Since $k\mapsto\lambda_k$ is strictly increasing on $\N_0$, it follows that
$k\mapsto \varphi^{s+\ln}_N(\lambda_k)$ is strictly increasing.

\medskip\noindent\textbf{(ii) $N=2,3$.}
Let $\Lambda_N$ be as in Lemma~\ref{lm es-21} (ii). Then $\phi_0(s,N;\lambda)>0$ for all $\lambda>\Lambda_N$.
Since $\lambda_k\to\infty$ and $\lambda_k$ is increasing, there exists $k_*=k_*(N)$ such that $\lambda_k>\Lambda_N$
for all $k\ge k_*$. Hence $\varphi^{s+\ln}_N(\lambda)$ is strictly increasing on $(\Lambda_N,\infty)$, and thus
\[
k\ge k_* \ \Longrightarrow\  \varphi^{s+\ln}_N(\lambda_{k+1})>\varphi^{s+\ln}_N(\lambda_k).
\]

It remains to compare the finitely many indices $0\le k<k_*$. For $N=2,3$ we have $\lambda_0=0$ and $\lambda_1=N$.
By Lemma~\ref{lm es-21} (ii), the sign of $\phi_0(s,N;0)$ is:
\[
\phi_0(s,2;0)<0\quad \forall s\in(0,1),\qquad
\exists\, s_0\in(0,1):\ \phi_0(s_0,3;0)=0,\ 
\phi_0(s,3;0)\ge 0\,(\le 0)\ \text{for } s\le s_0\,(\ge 0).
\]
Since $\varphi_{N,s}(0)>0$, the same holds for $\varphi^{s+\ln}_N(0)=\varphi_{N,s}(0)\phi_0(s,N;0)$.

On the other hand, $\lambda_1=N>\Lambda_N$ for $N=2,3$ ($k_*=1$), hence $\phi_0(s,N;\lambda_1)>0$ by Lemma~\ref{lm es-21} (ii) and thus
$\varphi^{s+\ln}_N(\lambda_1)>0$. To compare with $\lambda_0=0$, note that $\varphi_{N,s}(0)>0$ and
$\varphi^{s+\ln}_N(0)=\varphi_{N,s}(0)\phi_0(s,N;0)$.
If $N=2$, then $\phi_0(s,2;0)<0$ for all $s\in(0,1)$, hence $\varphi^{s+\ln}_2(\lambda_1)>\varphi^{s+\ln}_2(\lambda_0)$.
If $N=3$, let $s_0\in(0,1)$ be the unique number such that $\phi_0(s_0,3;0)=0$.
For $s\ge s_0$ we have $\phi_0(s,3;0)\le 0$, hence $\varphi^{s+\ln}_3(\lambda_0)\le 0<\varphi^{s+\ln}_3(\lambda_1)$.
For $s\in(0,s_0)$ we have $\phi_0(s,3;0)>0$, and since $\phi_0$ is increasing in $\lambda$,
$\phi_0(s,3;\lambda)\ge \phi_0(s,3;0)>0$ for all $\lambda\ge 0$; therefore $\varphi^{s+\ln}_3$ is strictly increasing on
$(0,\infty)$ and $\varphi^{s+\ln}_3(\lambda_1)>\varphi^{s+\ln}_3(\lambda_0)$.

Combining this with the strict increase on $(\Lambda_N,\infty)$ yields that
$k\mapsto \varphi^{s+\ln}_N(\lambda_k)$ is also strictly increasing.

\medskip\noindent\textbf{(iii) $N=1$.}
Here $\lambda_0=0$, $\lambda_1=1$, and $\lambda_k=k^2$ for $k\ge 0$.
Lemma~\ref{lm es-21} (iii) implies that for $s\in(0,\tfrac12)$,
\[
\phi_0(s,1;0)<0,\qquad \phi_0(s,1;1)\ge 0\,(\le 0)\ \text{for } s\le s_1\,(\ge s_1),
\]
for a unique $s_1\in(0,\tfrac12)$. In particular, $\varphi^{s+\ln}_1(0)<0$ for all $s\in(0,\tfrac12)$, while
$\varphi^{s+\ln}_1(1)=\varphi_{1,s}(1)\phi_0(s,1;1)$ may change sign at $s=s_1$.

Let $\Lambda_1$ be as in Lemma~\ref{lm es-21}(iii). Since $\lambda_2=4>\Lambda_1$, we have $\phi_0(s,1;\lambda)>0$
for all $\lambda\ge 4$, hence $\varphi^{s+\ln}_1$ is strictly increasing on $(4,\infty)$ and therefore
\[
k\ge 2\ \Longrightarrow\ \varphi^{s+\ln}_1(\lambda_{k+1})>\varphi^{s+\ln}_1(\lambda_k).
\]

It remains to check the first steps.
First, since $\varphi^{s+\ln}_1(0)<0$ and $\varphi^{s+\ln}_1(4)>0$, we get
\[
\varphi^{s+\ln}_1(\lambda_2)>\varphi^{s+\ln}_1(\lambda_0).
\]
Moreover, if $s\le s_1$ then $\phi_0(s,1;1)\ge 0$, hence $\varphi^{s+\ln}_1$ is strictly increasing on $(0,\infty)$ and thus
$\varphi^{s+\ln}_1(\lambda_1)>\varphi^{s+\ln}_1(\lambda_0)$ and $\varphi^{s+\ln}_1(\lambda_2)>\varphi^{s+\ln}_1(\lambda_1)$.

If $s\in(s_1,\tfrac12)$, note that
\[
\varphi_{1,s}(1)=\frac{\Gamma(\frac32+s)}{\Gamma(\frac32-s)}
=\frac{\frac12+s}{\frac12-s}\,\frac{\Gamma(\frac12+s)}{\Gamma(\frac12-s)}
=\frac{\frac12+s}{\frac12-s}\,\varphi_{1,s}(0),
\]
and, using $\psi(x+1)=\psi(x)+\frac1x$,
\[
\phi_0(s,1;1)=\psi(\tfrac32+s)+\psi(\tfrac32-s)
=\phi_0(s,1;0)+\frac{1}{\tfrac12+s}+\frac{1}{\tfrac12-s}.
\]
Hence
\begin{align*}
\varphi^{s+\ln}_1(1)-\varphi^{s+\ln}_1(0)
&=\varphi_{1,s}(0)\Bigg[
\frac{\frac12+s}{\frac12-s}\Bigl(\phi_0(s,1;0)+\frac{1}{\tfrac12+s}+\frac{1}{\tfrac12-s}\Bigr)
-\phi_0(s,1;0)\Bigg]\\
&=\varphi_{1,s}(0)\Bigg[
\frac{2s}{\frac12-s}\,\phi_0(s,1;0)
+\frac{1}{\frac12-s}+\frac{\frac12+s}{(\frac12-s)^2}
\Bigg]>0,
\end{align*}
which implies $\varphi^{s+\ln}_1(\lambda_1)>\varphi^{s+\ln}_1(\lambda_0)$. Therefore the strict increase holds also between $k=0$ and $k=1$.

Altogether, $k\mapsto \varphi^{s+\ln}_1(\lambda_k)$ is strictly increasing for $N=1$ as well.
\end{proof}

We now turn to the proof of the spectral characterization of $\mathscr P^{s+\ln}_g$, showing that it
is diagonalized by spherical harmonics and that its $C^\beta$-eigenfunctions are precisely the
elements of the eigenspaces $\mathcal H_k$.\medskip

\noindent {\bf Proof of Theorem \ref{thm:frac-log-basic}: }
\textbf{(i).} Since for every $t\in(0,1)$ we have
\[
  \mathscr{P}^t_g\Phi
  = \varphi_{N,t}(\lambda)\,\Phi
  \quad\text{on }\mathbb{S}^N,
\]
with $\varphi_{N,t}(\lambda)$ defined in $\ref{varphins}$.
Differentiating with respect to $t$ at $t=s$ gives
\[
  \mathscr{P}^{s+\ln}_g\Phi(z)
  = \left.\frac{d}{dt}\mathscr{P}^t_g\Phi(z)\right|_{t=s}
  = \left.\frac{d}{dt}\varphi_{N,t}(\lambda)\right|_{t=s}\,\Phi(z)
  = \varphi^{s+\ln}_N(\lambda)\,\Phi(z),
\]
where we set
\[
  \varphi^{s+\ln}_N(\lambda)
  := \left.\frac{d}{dt}\varphi_{N,t}(\lambda)\right|_{t=s}.
\]
To compute this derivative, write
\[
  \varphi_{N,t}(\lambda)
  = \frac{\Gamma(A_t)}{\Gamma(B_t)},
  \quad\text{with}\quad
  A_t:=\tfrac12+t+\rho_\lambda,\quad
  B_t:=\tfrac12-t+\rho_\lambda,
\]
where
\[
  \rho_\lambda:=\sqrt{\lambda+\tfrac14(N-1)^2}.
\]
Then
\[
  \ln\varphi_{N,t}(\lambda)
  = \ln\Gamma(A_t)-\ln\Gamma(B_t),
\]
and therefore
\[
  \frac{d}{dt}\ln\varphi_{N,t}(\lambda)
  = \psi(A_t)A'_t - \psi(B_t)B'_t
  = \psi(A_t)+\psi(B_t),
\]
since $A'_t=1$ and $B'_t=-1$. Hence
\[
  \frac{d}{dt}\varphi_{N,t}(\lambda)
  = \varphi_{N,t}(\lambda)\bigl[\psi(A_t)+\psi(B_t)\bigr],
\]
and evaluating at $t=s$ yields~\eqref{eq:frac-log-eigvalue}. This proves
that $\Phi$ is an eigenfunction of $\mathscr{P}^{s+\ln}_g$.

\textbf{(ii).} Suppose that $\Psi\in C^\infty(\mathbb{S}^N)$ satisfies
\eqref{eq:eig-frac-log} for some real number $\mu$, i.e.
\[
  \mathscr{P}^{s+\ln}_g\Psi = \mu\Psi\quad\text{on }\mathbb{S}^N.
\]

Let $\{Y_{k,j}\}_{k\ge0,\,1\le j\le d_k}$ be an orthonormal basis of
spherical harmonics on $\mathbb{S}^N$, so that
\[
  -\Delta_g Y_{k,j} = \lambda_k Y_{k,j},
  \quad \lambda_k = k(k+N-1),
\]
and each eigenspace
\(
  \mathcal{H}_k:=\mathrm{span}\{Y_{k,j}:1\le j\le d_k\}
\)
is finite-dimensional.
It is well known (see, e.g., \cite{chang2011fractional,chang2025fractional}) that
$\mathscr{P}^t_g$ commutes with $-\Delta_g$ and acts diagonally on
spherical harmonics:
\[
  \mathscr{P}^t_g Y_{k,j}
  = \varphi_{N,t}(\lambda_k)\,Y_{k,j}
  \qquad\text{for all }t\in(0,1),
\]
so that, by differentiation in $t$ at $t=s$,
\[
  \mathscr{P}^{s+\ln}_g Y_{k,j}
  = \varphi^{s+\ln}_N(\lambda_k)\,Y_{k,j}
  \qquad\text{for all }k\ge0,\ 1\le j\le d_k.
\]

Expanding $\Psi$ in spherical harmonics,
\[
  \Psi = \sum_{k=0}^\infty\sum_{j=1}^{d_k} a_{k,j} Y_{k,j},
\]
we obtain
\[
  \mathscr{P}^{s+\ln}_g\Psi
  = \sum_{k=0}^\infty\sum_{j=1}^{d_k}
      a_{k,j}\,\varphi^{s+\ln}_N(\lambda_k)\,Y_{k,j}.
\]
The eigenvalue equation
\(
  \mathscr{P}^{s+\ln}_g\Psi = \mu\Psi
\)
thus yields
\[
  \sum_{k,j} a_{k,j}\,\bigl(\varphi^{s+\ln}_N(\lambda_k)-\mu\bigr)\,Y_{k,j}
  = 0.
\]
By orthonormality of the family $\{Y_{k,j}\}$ we deduce that, for each
pair $(k,j)$,
\[
  a_{k,j}\,\bigl(\varphi^{s+\ln}_N(\lambda_k)-\mu\bigr) = 0.
\]

From Proposition \ref{pr spc func}, for any $\mu$, there exists a unique  $k\in\N_0$ such that 
\[
  \varphi^{s+\ln}_N(\lambda_k)=\mu.
\]
  Hence
\[
  \Psi
  \in \mathcal{H}_k,
\]
which is the desired conclusion.\hfill$\Box$\medskip

\section{Conformal Covariance and Yamabe-Type Problems}

In this section we develop the conformal geometric framework associated with the
$\mathscr P^{s+\ln}_g$. We first establish its conformal covariance
law under metric changes, then relate the spherical operator to its Euclidean counterpart via an
intertwining identity, and finally introduce the associated fractional--logarithmic $Q$-curvature
and the corresponding Yamabe-type equations on $\mathbb S^N$ and $\mathbb R^N$.

\subsection{Conformal Covariance Law}

In this subsection we first recall the conformal covariance of the conformal fractional Laplacian
$\mathscr P_g^{s}$, and then derive the corresponding conformal transformation law for the
fractional--logarithmic operator.

Let $\vartheta\in C^\infty(\mathbb{S}^N)$ be a positive function.
For each $s\in(0,1)$ we consider a conformal change of metric
\[
  \widetilde g = \vartheta^{\frac{4}{N-2s}} g,
  \qquad \vartheta>0.
\]
It is a standard fact (see, for instance,
\cite{chang2025fractional,gonzalez2016recent}) that the conformal
fractional Laplacian satisfies the covariance law
\begin{equation}\label{eq:frac-conformal-general}
  \mathscr{P}^s_{\widetilde g}(u)
  = \vartheta^{-\frac{N+2s}{N-2s}}\,
    \mathscr{P}^s_g(\vartheta\,u),
  \qquad u\in C^\infty(\mathbb{S}^N).
\end{equation}

In our setting it is convenient to work with a fixed conformal factor
$\eta$ for the metric, i.e.
\[
  \widetilde g = \eta\,g.
\]
For $N>2s$, there exists a unique positive function
$\vartheta_s=\vartheta_s(\eta)$ such that
\[
  \eta\,g = \vartheta_s^{\frac{4}{N-2s}}g,
  \qquad\text{namely}\qquad
  \vartheta_s := \eta^{\frac{N-2s}{4}}.
\]
Plugging $\vartheta=\vartheta_s$ into
\eqref{eq:frac-conformal-general} we obtain, for every
$u\in C^\infty(\mathbb{S}^N)$,
\begin{equation}\label{eq:frac-conformal-fixed-eta}
  \mathscr{P}^s_{\eta g}(u)
  = \vartheta_s^{-\frac{N+2s}{N-2s}}\,
    \mathscr{P}^s_g(\vartheta_s u)
  = \eta^{-\frac{N+2s}{4}}\,
    \mathscr{P}^s_g\bigl(\eta^{\frac{N-2s}{4}}u\bigr).
\end{equation}
This is the explicit conformal law for $\mathscr{P}^s_g$ under the metric
change $g\mapsto\eta g$, which we will differentiate in $s$ to obtain the
transformation rule for the conformal fractional--logarithmic operator.

Recall that for fixed $s\in(0,1)$, the conformal fractional--logarithmic operator is
defined by
\[
  \mathscr{P}^{s+\ln}_g u
  := \left.\frac{d}{dt}\mathscr{P}^t_g u\right|_{t=s}.
\]
In particular, with respect to the conformally related metric
$\eta g$ we define
\[
  \mathscr{P}^{s+\ln}_{\eta g} u
  := \left.\frac{d}{dt}\mathscr{P}^t_{\eta g} u\right|_{t=s},
  \quad u\in C^\infty(\mathbb{S}^N).
\]

Using \eqref{eq:frac-conformal-fixed-eta} with $t$ in place of $s$, we
can write, for $t$ in a neighbourhood of $s$,
\begin{equation}\label{eq:frac-conformal-t}
  \mathscr{P}^t_{\eta g}(u)
  = \eta^{-\frac{N+2t}{4}}\,
    \mathscr{P}^t_g\bigl(\eta^{\frac{N-2t}{4}}u\bigr).
\end{equation}
For convenience set
\[
  a(t):=\frac{N+2t}{4},\qquad
  b(t):=\frac{N-2t}{4},
\]
so that \eqref{eq:frac-conformal-t} becomes
\[
  \mathscr{P}^t_{\eta g}(u)
  = \eta^{-a(t)}\,\mathscr{P}^t_g\bigl(\eta^{b(t)}u\bigr).
\]

Differentiating with respect to $t$ at $t=s$ and using the definition of
$\mathscr{P}^{t+\ln}_g$, we obtain the following transformation law.

\begin{proof}[\textbf{Proof of Proposition \ref{prop:conf-frac-log}.}]
Starting from \eqref{eq:frac-conformal-t}, we write
\[
  \mathscr{P}^t_{\eta g}(u)
  = F(t)\,G(t),
\]
where
\[
  F(t):=\eta^{-a(t)},\qquad
  G(t):=\mathscr{P}^t_g\bigl(\eta^{b(t)}u\bigr).
\]
Differentiating and evaluating at $t=s$ we obtain
\[
  \mathscr{P}^{s+\ln}_{\eta g}u
  = F'(s)G(s)+F(s)G'(s).
\]

Since $a'(t)=\tfrac12$, we have
\[
  F'(t) = -a'(t)\ln(\eta)\,\eta^{-a(t)}
        = -\tfrac12\ln(\eta)\,\eta^{-a(t)},
\]
thus
\[
  F'(s)G(s)
  = -\frac12\ln(\eta)\,\eta^{-a(s)}
    \mathscr{P}^s_g\bigl(\eta^{b(s)}u\bigr).
\]

Next, we differentiate $G(t)$ using the definition of the
fractional--logarithmic operator and the chain rule for the factor
$\eta^{b(t)}$:
\[
  G'(t)
  = \mathscr{P}^{t+\ln}_g\bigl(\eta^{b(t)}u\bigr)
    + \mathscr{P}^t_g\bigl(b'(t)\ln(\eta)\,\eta^{b(t)}u\bigr).
\]
Since $b'(t)=-\tfrac12$, we have
\[
  G'(s)
  = \mathscr{P}^{s+\ln}_g\bigl(\eta^{b(s)}u\bigr)
    - \frac12\,\mathscr{P}^s_g\bigl((\ln\eta)\,\eta^{b(s)}u\bigr).
\]
Therefore
\[
  F(s)G'(s)
  = \eta^{-a(s)}
    \biggl[
      \mathscr{P}^{s+\ln}_g\bigl(\eta^{b(s)}u\bigr)
      - \frac12\,\mathscr{P}^s_g\bigl((\ln\eta)\,\eta^{b(s)}u\bigr)
    \biggr].
\]

Combining the expressions for $F'(s)G(s)$ and $F(s)G'(s)$, and recalling
that $a(s)=(N+2s)/4$ and $b(s)=(N-2s)/4$, we obtain
\[
  \mathscr{P}^{s+\ln}_{\eta g}u
  = \eta^{-\frac{N+2s}{4}}
    \biggl[
      \mathscr{P}^{s+\ln}_g\bigl(\eta^{\frac{N-2s}{4}}u\bigr)
      - \frac12\,\ln(\eta)\,
        \mathscr{P}^s_g\bigl(\eta^{\frac{N-2s}{4}}u\bigr)
      - \frac12\,\mathscr{P}^s_g\bigl((\ln\eta)\,\eta^{\frac{N-2s}{4}}u\bigr)
    \biggr],
\]
which is \eqref{eq:conf-frac-log-eta}.
The equivalent formulation \eqref{eq:conf-frac-log-phi} follows by the
setting $\varphi=\eta^{\frac{N-2s}{4}}u$.
\end{proof}

\subsection{Intertwining with the Euclidean Model}

In this subsection we explain how the spherical operators are related to their Euclidean
counterparts through stereographic projection. More precisely, the conformal pullback $\mathcal T_s$
intertwines  $\mathscr P_g^{s+\ln}$ on $\mathbb S^N$ with 
$(-\Delta)^{s+\ln}$ on $\mathbb R^N$, respectively. This correspondence allows us to transfer
definitions, mapping properties, and nonlinear Yamabe-type equations between $\mathbb S^N$ and
$\mathbb R^N$.

\begin{proof}[\textbf{Proof of Proposition \ref{prop:sphere-RN-frac-log}.}]
For each $t\in(0,1)$, define
\[
  v_t(x):=\mathcal{T}_t[u](x)
  = \phi(x)^{\frac{N-2t}{2}}\,u\bigl(\sigma^{-1}(x)\bigr),
  \qquad x\in \mathbb{R}^N.
\]
From the fractional correspondence we have, for all $t\in(0,1)$,
\begin{equation}\label{eq:frac-intertwining-t}
  \mathcal{T}_t\bigl[\mathscr{P}^t_g u\bigr](x)
  = \phi(x)^{-2t}\,(-\Delta)^t v_t(x),
  \qquad x\in \mathbb{R}^N.
\end{equation}
We differentiate this identity with respect to $t$ at $t=s$.

Set
\[
  F(t,x):=\mathcal{T}_t\bigl[\mathscr{P}^t_g u\bigr](x).
\]
Then
\[
  F'(s,x)
  = \Bigl(\frac{d}{dt}\mathcal{T}_t\Bigr)\Big|_{t=s}
        \bigl[\mathscr{P}^s_g u\bigr](x)
    + \mathcal{T}_s\bigl[\mathscr{P}^{s+\ln}_g u\bigr](x).
\]

We first compute the derivative of $\mathcal{T}_t$. For any
$w:\mathbb{S}^N\to\mathbb{R}$,
\[
  \mathcal{T}_t[w](x)
  = \phi(x)^{\frac{N-2t}{2}}\,w\bigl(\sigma^{-1}(x)\bigr),
\]
so
\[
  \frac{d}{dt}\mathcal{T}_t[w](x)\Big|_{t=s}
  = -\ln\phi(x)\,\phi(x)^{\frac{N-2s}{2}}\,w\bigl(\sigma^{-1}(x)\bigr)
  = -\ln\phi(x)\,\mathcal{T}_s[w](x).
\]
Hence
\[
  \Bigl(\frac{d}{dt}\mathcal{T}_t\Bigr)\Big|_{t=s}
      \bigl[\mathscr{P}^s_g u\bigr](x)
  = -\ln\phi(x)\,\mathcal{T}_s\bigl[\mathscr{P}^s_g u\bigr](x).
\]

From \eqref{eq:frac-intertwining-t} with $t=s$ we know that
\[
  \mathcal{T}_s\bigl[\mathscr{P}^s_g u\bigr](x)
  = \phi(x)^{-2s}\,(-\Delta)^s v_s(x),
\]
thus
\begin{equation}\label{eq:LHS-derivative}
  F'(s,x)
  = -\ln\phi(x)\,\phi(x)^{-2s}\,(-\Delta)^s v_s(x)
    + \mathcal{T}_s\bigl[\mathscr{P}^{s+\ln}_g u\bigr](x).
\end{equation}

Set
\[
  G(t,x):=\phi(x)^{-2t}\,(-\Delta)^t v_t(x).
\]
Then
\[
  G'(s,x)
  = \frac{d}{dt}\Bigl(\phi^{-2t}\Bigr)\Big|_{t=s}(-\Delta)^s v_s(x)
    + \phi(x)^{-2s}\frac{d}{dt}\Bigl((-\Delta)^t v_t\Bigr)\Big|_{t=s}(x).
\]
We have
\[
  \frac{d}{dt}\bigl(\phi^{-2t}\bigr)\Big|_{t=s}
  = -2\ln\phi(x)\,\phi(x)^{-2s},
\]
hence
\[
  \frac{d}{dt}\Bigl(\phi^{-2t}\Bigr)\Big|_{t=s}(-\Delta)^s v_s
  = -2\ln\phi(x)\,\phi(x)^{-2s}\,(-\Delta)^s v_s(x).
\]

Next, using the definition of the fractional--logarithmic Laplacian,
\[
  \frac{d}{dt}\Bigl((-\Delta)^t v_t\Bigr)\Big|_{t=s}
  = (-\Delta)^{s+\ln} v_s
    + (-\Delta)^s\Bigl(\frac{d}{dt}v_t\Big|_{t=s}\Bigr).
\]
Since
\[
  v_t(x)
  = \phi(x)^{\frac{N-2t}{2}}\,u(\sigma^{-1}(x)),
\]
we obtain
\[
  \frac{d}{dt}v_t(x)\Big|_{t=s}
  = -\ln\phi(x)\,\phi(x)^{\frac{N-2s}{2}}\,u(\sigma^{-1}(x))
  = -\ln\phi(x)\,v_s(x),
\]
so
\[
  (-\Delta)^s\Bigl(\frac{d}{dt}v_t\Big|_{t=s}\Bigr)(x)
  = -(-\Delta)^s\bigl((\ln\phi)\,v_s\bigr)(x).
\]

Putting these together gives
\begin{align}
  G'(s,x)
  &= -2\ln\phi(x)\,\phi(x)^{-2s}\,(-\Delta)^s v_s(x)  \notag\\
  &\quad + \phi(x)^{-2s}\Bigl[
      (-\Delta)^{s+\ln} v_s(x)
      - (-\Delta)^s\bigl((\ln\phi)\,v_s\bigr)(x)
    \Bigr].\label{eq:RHS-derivative}
\end{align}

Since $F(t,x)=G(t,x)$ for all $t\in(0,1)$ and $x\in \mathbb{R}^N$, we
have $F'(s,x)=G'(s,x)$.
Equating \eqref{eq:LHS-derivative} and \eqref{eq:RHS-derivative}, we get
\begin{align*}
  &-\ln\phi(x)\,\phi(x)^{-2s}\,(-\Delta)^s v_s(x)
    + \mathcal{T}_s\bigl[\mathscr{P}^{s+\ln}_g u\bigr](x) \\
  &\qquad
   = -2\ln\phi(x)\,\phi(x)^{-2s}\,(-\Delta)^s v_s(x) \\
  &\qquad\quad
     + \phi(x)^{-2s}\Bigl[
        (-\Delta)^{s+\ln} v_s(x)
        - (-\Delta)^s\bigl((\ln\phi)\,v_s\bigr)(x)
      \Bigr].
\end{align*}
Rearranging terms, this is equivalent to
\[
  \mathcal{T}_s\bigl[\mathscr{P}^{s+\ln}_g u\bigr](x)
  = \phi(x)^{-2s}\Bigl[
      (-\Delta)^{s+\ln} v_s(x)
      - (-\Delta)^s\bigl((\ln\phi)\,v_s\bigr)(x)
      - \bigl(\ln\phi(x)\bigr)\,(-\Delta)^s v_s(x)
    \Bigr],
\]
which is exactly \eqref{eq:intertwining-frac-log}.
\end{proof}

\subsection{Q-Curvature and Yamabe-Type Problems}

In this subsection we introduce the curvature prescription problem naturally associated with
$\mathscr P^{s+\ln}_g$. We formulate the constant fractional--logarithmic $Q$-curvature problem on
$\mathbb S^N$, derive the corresponding Yamabe-type equation, and explain its Euclidean counterpart
via stereographic projection. This framework will be used later to establish the sphere--Euclidean
equivalence (Theorem~\ref{thm:frac-log-correspondence}) and to construct explicit bubble solutions
(Theorem~\ref{prop:radial-bubble-frac-log}).

The constant fractional--logarithmic $Q$-curvature problem on $(\mathbb S^N,g)$ consists in finding a
conformal metric
\[
\tilde g = u^{\frac{4}{N-2s}}\,g
\]
such that $Q^{s+\ln}_{\tilde g}\equiv \mu$ is constant on $\mathbb S^N$ for some $\mu\in\mathbb R$.
By \eqref{eq:frac-log-Yamabe-eq}, this is equivalent to finding a positive solution $u>0$ of the
fractional--logarithmic Yamabe-type equation
\begin{equation}\label{eq:frac-log-Yamabe}
\mathscr P^{s+\ln}_g(u)
= \frac{2}{N-2s}\,(\ln u)\,\mathscr P^s_g(u)
+ \frac{2}{N-2s}\,\mathscr P^s_g\bigl((\ln u)\,u\bigr)
+ \mu\,u^{\frac{N+2s}{N-2s}}
\quad\text{on }\mathbb S^N.
\end{equation}

The curvature quantity $Q^{s+\ln}_{\tilde g}$ and the equation \eqref{eq:frac-log-Yamabe} should be
viewed as the natural fractional--logarithmic counterparts of the fractional $Q$-curvature and the
fractional Yamabe equation. Moreover, since $\mathscr P^{s+\ln}_g$ arises by differentiating
$\mathscr P_g^t$ with respect to the order, one expects that as $s\to0^+$ the equation
\eqref{eq:frac-log-Yamabe} converges (after a suitable normalization) to the logarithmic Yamabe
equation, with the additional terms involving $\ln u$ appearing as first-order corrections in the
parameter $s$.

\begin{proof}[\textbf{Proof of Proposition \ref{prop:limit-frac-log-yamabe-us}.}]
We pass to the limit in \eqref{eq:frac-log-Yamabe-s-family-us} term by term.

Since \(u_{s_k}\to u_0\) in \(C^\beta(\mathbb S^N)\) and \(u_0>0\), for \(k\) large enough we have
\[
u_{s_k}\ge c>0 \quad\text{on }\mathbb S^N,
\]
hence \(\ln u_{s_k}\) is well defined and
\[
\ln u_{s_k}\to \ln u_0
\quad\text{in }C^\beta(\mathbb S^N)
\]
and in particular uniformly.

By Proposition~\ref{prop:frac-log-kernel} $(ii)$, applied to \(u_{s_k}\), using the convergence in \(C^\beta\) and the continuity of the operator family under the same Hölder control, we have
\[
\mathscr P_g^{s_k+\ln}(u_{s_k}) \longrightarrow \mathscr P_g^{\ln}(u_0)
\qquad\text{uniformly on }\mathbb S^N.
\]

Next, for smooth functions \(v\), one has
\[
\mathscr P_g^{s_k}(v)\to v
\qquad\text{uniformly on }\mathbb S^N
\quad k\to\infty,
\]
because \(c_{N,s_k}\to0\) and \(A_{N,s_k}\to1\). Applying this to \(v=u_{s_k}\) and \(v=(\ln u_{s_k})u_{s_k}\), together with the convergence \(u_{s_k}\to u_0\) and \((\ln u_{s_k})u_{s_k}\to (\ln u_0)u_0\) in \(C^\beta\), yields
\[
\mathscr P_g^{s_k}(u_{s_k})\to u_0,
\qquad
\mathscr P_g^{s_k}\bigl((\ln u_{s_k})u_{s_k}\bigr)\to (\ln u_0)u_0
\]
uniformly on \(\mathbb S^N\).

Also,
\[
\frac{2}{N-2s_k}\to \frac{2}{N},
\qquad
u_{s_k}^{\frac{N+2s_k}{N-2s_k}}\to u_0
\]
uniformly on \(\mathbb S^N\), since \(u_{s_k}\to u_0>0\) uniformly and the exponent tends to \(1\). Hence
\[
\frac{2}{N-2s_k}\,\ln u_{s_k}\,\mathscr P_g^{s_k}(u_{s_k})
\to \frac{2}{N}\,u_0\ln u_0,
\]
\[
\frac{2}{N-2s_k}\,\mathscr P_g^{s_k}\bigl((\ln u_{s_k})u_{s_k}\bigr)
\to \frac{2}{N}\,u_0\ln u_0,
\]
and
\[
\mu_{s_k}\,u_{s_k}^{\frac{N+2s_k}{N-2s_k}}
\to \mu_0 u_0
\]
uniformly on \(\mathbb S^N\).

Passing to the limit in \eqref{eq:frac-log-Yamabe-s-family-us}, we obtain
\[
\mathscr P_g^{\ln}(u_0)
=
\frac{2}{N}\,u_0\ln u_0
+\frac{2}{N}\,u_0\ln u_0
+\mu_0u_0
=
\frac{4}{N}\,u_0\ln u_0+\mu_0u_0,
\]
which is exactly \eqref{eq:log-yamabe-limit-us}.
\end{proof}

\bigskip

We now prove the equivalence between the spherical and Euclidean fractional--logarithmic
Yamabe-type problems. The argument relies on the intertwining identity in
Proposition~\ref{prop:sphere-RN-frac-log} together with the conformal pullback $\mathcal T_s$:
starting from a solution on $\mathbb S^N$ we obtain a solution on $\mathbb R^N$ by applying
$\mathcal T_s$, and conversely we recover the spherical equation by inverting the pullback and
using the correspondence of the conformal factors.

\begin{proof}[\textbf{Proof of Theorem \ref{thm:frac-log-correspondence}.}]
$(a)$ Assume first that $u$ satisfies
\eqref{eq:frac-log-Yamabe-SN} on $\mathbb{S}^N$.
We apply the transform $\mathcal{T}_s$ to both sides of
\eqref{eq:frac-log-Yamabe-SN} and use the sphere--Euclidean
correspondence for $\mathscr{P}^{s+\ln}_g$ and $\mathscr{P}^s_g$.

By Proposition~\ref{prop:sphere-RN-frac-log} we have
\begin{equation}\label{eq:proof-LHS}
  \mathcal{T}_s\bigl[\mathscr{P}^{s+\ln}_g u\bigr]
  = \phi^{-2s}\Bigl[
      (-\Delta)^{s+\ln} v_s
      - (-\Delta)^s\bigl((\ln\phi)\,v_s\bigr)
      - (\ln\phi)\,(-\Delta)^s v_s
    \Bigr].
\end{equation}

We treat each term in \eqref{eq:frac-log-Yamabe-SN} separately. First, using the fractional correspondence
\(
  \mathcal{T}_s[\mathscr{P}^s_g w]=\phi^{-2s}(-\Delta)^s\mathcal{T}_s[w]
\)
we obtain
\begin{equation}\label{eq:Ts-Ps-u}
  \mathcal{T}_s\bigl[\mathscr{P}^s_g u\bigr]
  = \phi^{-2s}(-\Delta)^s v_s.
\end{equation}
Next, since
\[
  u(\sigma^{-1}(x))
  = \phi(x)^{-\frac{N-2s}{2}}v_s(x),
\]
we have
\[
  \ln u(\sigma^{-1}(x))
  = \ln v_s(x) - \frac{N-2s}{2}\,\ln\phi(x).
\]

For the term $\ln u\,\mathscr{P}^s_g u$ we compute
\begin{align}
  \mathcal{T}_s\bigl[\ln u\,\mathscr{P}^s_g u\bigr](x)
  &= \phi(x)^{\frac{N-2s}{2}}\,
      \ln u(\sigma^{-1}(x))\,
      \mathscr{P}^s_g u(\sigma^{-1}(x)) \notag\\
  &= \phi(x)^{\frac{N-2s}{2}}
     \Bigl(\ln v_s(x)-\frac{N-2s}{2}\ln\phi(x)\Bigr)
     \phi(x)^{-\frac{N+2s}{2}}(-\Delta)^s v_s(x) \notag\\
  &= \phi(x)^{-2s}
     \Bigl(\ln v_s(x)-\frac{N-2s}{2}\ln\phi(x)\Bigr)
     (-\Delta)^s v_s(x).\label{eq:Ts-lnu-Psu}
\end{align}
Here we used again that
\(
  \mathscr{P}^s_g u(\sigma^{-1})
  = \phi^{-\frac{N+2s}{2}}(-\Delta)^s v_s
\).

For the term $\mathscr{P}^s_g((\ln u)u)$ we first compute
\[
  (\ln u)u\bigl(\sigma^{-1}(x)\bigr)
  = \phi(x)^{-\frac{N-2s}{2}}v_s(x)
    \Bigl(\ln v_s(x)-\frac{N-2s}{2}\ln\phi(x)\Bigr),
\]
and hence
\begin{equation}\label{eq:Ts-ln-u-u}
  \mathcal{T}_s\bigl[(\ln u)u\bigr](x)
  = v_s(x)\Bigl(
      \ln v_s(x)-\frac{N-2s}{2}\ln\phi(x)
    \Bigr).
\end{equation}
Applying the fractional correspondence to
$\mathscr{P}^s_g((\ln u)u)$ gives
\begin{equation}\label{eq:Ts-Ps-lnu-u}
  \mathcal{T}_s\bigl[\mathscr{P}^s_g((\ln u)u)\bigr]
  = \phi^{-2s}
      (-\Delta)^s\Bigl(
        v_s\Bigl(\ln v_s-\tfrac{N-2s}{2}\ln\phi\Bigr)
      \Bigr).
\end{equation}

Finally, for the power nonlinearity we use
\[
  u^{\frac{N+2s}{N-2s}}(\sigma^{-1}(x))
  = \phi(x)^{-\frac{N+2s}{2}}v_s(x)^{\frac{N+2s}{N-2s}},
\]
so that
\begin{equation}\label{eq:Ts-power}
  \mathcal{T}_s\Bigl[u^{\frac{N+2s}{N-2s}}\Bigr](x)
  = \phi(x)^{\frac{N-2s}{2}}\,
     \phi(x)^{-\frac{N+2s}{2}}\,
     v_s(x)^{\frac{N+2s}{N-2s}}
  = \phi(x)^{-2s}v_s(x)^{\frac{N+2s}{N-2s}}.
\end{equation}

Gathering \eqref{eq:Ts-lnu-Psu}, \eqref{eq:Ts-Ps-lnu-u}
and \eqref{eq:Ts-power}, we obtain
\begin{align}
  &\mathcal{T}_s\Bigl[
      \frac{2}{N-2s}\,\ln u\,\mathscr{P}^s_g u
      + \frac{2}{N-2s}\,\mathscr{P}^s_g\bigl((\ln u)u\bigr)
      + \mu\,u^{\frac{N+2s}{N-2s}}
    \Bigr] \notag\\
  &= \phi^{-2s}\Biggl\{
        \frac{2}{N-2s}\Bigl[
          \Bigl(\ln v_s-\tfrac{N-2s}{2}\ln\phi\Bigr)
            (-\Delta)^s v_s
          + (-\Delta)^s\Bigl(
              v_s\Bigl(\ln v_s-\tfrac{N-2s}{2}\ln\phi\Bigr)
            \Bigr)
        \Bigr]
        + \mu\,v_s^{\frac{N+2s}{N-2s}}
      \Biggr\}.\label{eq:proof-RHS-raw}
\end{align}

We now compare \eqref{eq:proof-LHS} with
$\mathcal{T}_s$ applied to the right-hand side of
\eqref{eq:frac-log-Yamabe-SN}.
Using \eqref{eq:frac-log-Yamabe-SN} and
\eqref{eq:proof-RHS-raw}, the identity
\[
  \mathcal{T}_s\bigl[\mathscr{P}^{s+\ln}_g u\bigr]
  = \mathcal{T}_s\Bigl[
      \frac{2}{N-2s}\,\ln u\,\mathscr{P}^s_g u
      + \frac{2}{N-2s}\,\mathscr{P}^s_g\bigl((\ln u)u\bigr)
      + \mu\,u^{\frac{N+2s}{N-2s}}
    \Bigr]
\]
is equivalent to
\begin{align*}
  &(-\Delta)^{s+\ln} v_s
   - (-\Delta)^s\bigl((\ln\phi)\,v_s\bigr)
   - (\ln\phi)\,(-\Delta)^s v_s \\
  &\quad =
  \frac{2}{N-2s}\Bigl[
      \Bigl(\ln v_s-\tfrac{N-2s}{2}\ln\phi\Bigr)\,(-\Delta)^s v_s
      + (-\Delta)^s\Bigl(
          v_s\Bigl(\ln v_s-\tfrac{N-2s}{2}\ln\phi\Bigr)
        \Bigr)
    \Bigr]
    + \mu\,v_s^{\frac{N+2s}{N-2s}}.
\end{align*}
Expanding the right-hand side and grouping the terms with $\ln\phi$ we
find
\begin{align*}
  &\frac{2}{N-2s}\Bigl[
      (\ln v_s)(-\Delta)^s v_s
      + (-\Delta)^s(v_s\ln v_s)
    \Bigr] \\
  &\qquad
    - \Bigl[
        (\ln\phi)\,(-\Delta)^s v_s
        + (-\Delta)^s\bigl(v_s\ln\phi\bigr)
      \Bigr]
    + \mu\,v_s^{\frac{N+2s}{N-2s}}.
\end{align*}
Hence all terms involving $\ln\phi$ cancel out, and we are left with
\[
  (-\Delta)^{s+\ln} v_s
  = \frac{2}{N-2s}\Bigl[
        (\ln v_s)\,(-\Delta)^s v_s
        + (-\Delta)^s(v_s\ln v_s)
    \Bigr]
    + \mu\,v_s^{\frac{N+2s}{N-2s}},
\]
which is exactly \eqref{eq:frac-log-Yamabe-RN}. This proves (i) $\Rightarrow$ (ii)
for smooth positive solutions.

Conversely, suppose $v_s>0$ solves \eqref{eq:frac-log-Yamabe-RN} in
$ \mathbb{R}^N$. Define
\[
  u(z):=\phi(x)^{-\frac{N-2s}{2}}v_s(x),
  \qquad x=\sigma(z).
\]
Then $u>0$ on $\mathbb{S}^N$ and $v_s=\mathcal{T}_s[u]$. Reversing the
previous computations (now applying $\mathcal{T}_s^{-1}$) shows that
$u$ satisfies \eqref{eq:frac-log-Yamabe-SN} on $\mathbb{S}^N$.
We omit the algebraic details, which are entirely analogous to the
above.
\end{proof}

We conclude this subsection by constructing an explicit family of positive ``bubble'' solutions.\medskip

\noindent{\bf Proof of Theorem \ref{prop:radial-bubble-frac-log}. }
Since \(u\equiv C\) is constant on \(\mathbb S^N\), the singular-integral parts vanish, and thus
\[
\mathscr P_g^s(u)=A_{N,s}C,
\qquad
\mathscr P_g^{s+\ln}(u)=A'_{N,s}C.
\]
Also, \((\ln u)\,u=(\ln C)\,C\) is constant, so
\[
\mathscr P_g^s\bigl((\ln u)\,u\bigr)=A_{N,s}(\ln C)\,C.
\]
Substituting these identities into \eqref{eq:frac-log-Yamabe-SN}, we obtain
\[
A'_{N,s}C
=
\frac{2}{N-2s}(\ln C)(A_{N,s}C)
+\frac{2}{N-2s}(A_{N,s}(\ln C)C)
+\mu\,C^{\frac{N+2s}{N-2s}}.
\]
Hence
\[
A'_{N,s}C
=
\frac{4}{N-2s}A_{N,s}(\ln C)C
+\mu\,C^{\frac{N+2s}{N-2s}},
\]
and therefore
\[
\mu
=
C^{1-\frac{N+2s}{N-2s}}
\left(
A'_{N,s}
-\frac{4}{N-2s}A_{N,s}\ln C
\right)
=
C^{-\frac{4s}{N-2s}}
\left(
A'_{N,s}
-\frac{4}{N-2s}A_{N,s}\ln C
\right).
\]
Thus \(u\equiv C\) is indeed a positive solution of \eqref{eq:frac-log-Yamabe-SN}.

Finally, by the definition of \(\mathcal T_s\),
\[
v_s(x)=\mathcal T_s[u](x)
=\phi(x)^{\frac{N-2s}{2}}u(\sigma^{-1}(x))
=
C\left(\frac{2}{1+|x|^2}\right)^{\frac{N-2s}{2}},
\]
which is positive and radial. The conclusion that \(v_s\) solves
\eqref{eq:frac-log-Yamabe-RN} follows directly from
Theorem~\ref{thm:frac-log-correspondence}.
\hfill$\Box$\medskip

\section{Applications to Sharp Sobolev Inequalities}

\subsection{Sharp Logarithmic Sobolev Inequality}\label{sharplog}

To transfer the sharp Euclidean logarithmic Sobolev inequality to the sphere, we first record the behavior of the \(L^2\)-norm, the entropy functional, and the logarithmic energy under the endpoint stereographic pullback.

\begin{lemma}\label{confcore}
    Let $(\mathbb S^N,g)$ be the round sphere and let
$\sigma:\mathbb S^N\setminus\{-e_{N+1}\}\to\mathbb R^N$ be the stereographic projection.
Set
\[
\phi(x):=\frac{2}{1+|x|^2},\quad x\in\mathbb R^N.
\]
Define the endpoint pullback
\[
(\mathcal T_0 u)(x):=\phi(x)^{\frac N2}\,u(\sigma^{-1}(x)).
\]
Then, for every $u\in C^\infty(\mathbb S^N)\setminus\{0\}$ and $v:=\mathcal T_0 u$, we have
\begin{enumerate}
\item[(i)] $\|v\|_{L^2(\mathbb R^N)}=\|u\|_{L^2(\mathbb S^N)}$;
\item[(ii)] 
\[
\int_{\mathbb R^N}\frac{v^2}{\|v\|_2^2}\ln\!\left(\frac{v^2}{\|v\|_2^2}\right)\,dx
=
\int_{\mathbb S^N}\frac{u^2}{\|u\|_2^2}\ln\!\left(\frac{u^2}{\|u\|_2^2}\right)\,dV_g
+N\int_{\mathbb S^N}\frac{u^2}{\|u\|_2^2}\ln\phi\,dV_g;
\]
\item[(iii)] The logarithmic energies satisfy
\[
\frac{\langle v,(-\Delta)^{\ln}v\rangle}{\|v\|_2^2}
=
\frac{\int_{\mathbb S^N}u\,\mathscr P_g^{\ln}u\,dV_g}{\|u\|_2^2}
+2\int_{\mathbb S^N}\frac{u^2}{\|u\|_2^2}\ln\phi\,dV_g.
\]
\end{enumerate}
\end{lemma}

\begin{proof}
\textbf{(i).} Using $dV_g=\phi^N dx$ and $v^2=\phi^N(u\circ\sigma^{-1})^2$, we obtain
\[
\|v\|_{L^2(\mathbb R^N)}^2
=\int_{\mathbb R^N}\phi^N\,u(\sigma^{-1}(x))^2\,dx
=\int_{\mathbb S^N}u^2\,dV_g
=\|u\|_{L^2(\mathbb S^N)}^2.
\]

\textbf{(ii).} Note that
\[
\ln\!\left(\frac{v^2}{\|v\|_2^2}\right)
=\ln\!\left(\frac{(u\circ\sigma^{-1})^2}{\|u\|_2^2}\right)+N\ln\phi
\qquad\text{on }\mathbb R^N,
\]
hence
\[
\int_{\mathbb R^N}\frac{v^2}{\|v\|_2^2}\ln\!\left(\frac{v^2}{\|v\|_2^2}\right)\,dx
=
\int_{\mathbb S^N}\frac{u^2}{\|u\|_2^2}\ln\!\left(\frac{u^2}{\|u\|_2^2}\right)\,dV_g
+N\int_{\mathbb S^N}\frac{u^2}{\|u\|_2^2}\ln\phi\,dV_g.
\]

\textbf{(iii).} By \cite[Proposition 1.3]{fernandez2025conformal}, we konw that
\[\mathcal T_0\!\big[\mathscr P_g^{\ln}u\big]
=
(-\Delta)^{\ln}v-2(\ln\phi)\,v
\qquad\text{in }\mathbb R^N.\]
Thus, we obtain
\[
\int_{\mathbb S^N}u\,\mathscr P_g^{\ln}u\,dV_g
=
\int_{\mathbb R^N} v\Bigl((-\Delta)^{\ln}v-2(\ln\phi)\,v\Bigr)\,dx
=
\langle v,(-\Delta)^{\ln}v\rangle-2\int_{\mathbb R^N}(\ln\phi)\,v^2\,dx.
\]
Dividing by $\|v\|_2^2=\|u\|_2^2$ and using $v^2dx=u^2dV_g$ gives the claimed relation in (iii).
\end{proof}

Fix $s\in [0,1)$. We recall the sharp fractional Sobolev inequality on $\mathbb R^N$ (see \cite[Theorem 1.1]{cotsiolis2004best})
\begin{equation}\label{eq:Sobolev-RN-frac-kappa}
\|v\|_{L^{2_s^*}(\mathbb R^N)}^2
\le \kappa_{N,s}\,\|v\|_{\dot H^s(\mathbb R^N)}^2,
\quad v\in \dot H^s(\mathbb R^N),
\end{equation}
where $2_s^*=\frac{2N}{N-2s}$ and
\[
\|v\|_{\dot H^s(\mathbb R^N)}^2
:=\int_{\mathbb R^N}|\xi|^{2s}\,|\widehat v(\xi)|^2\,d\xi,
\]
and the Fourier transform is defined by
\[
\widehat{f}(\xi):=(2\pi)^{-N/2}\int_{\mathbb R^N}e^{-ix\cdot \xi}f(x)\,dx.
\]
The sharp constant is
\begin{equation}\label{eq:kappaNs}
\kappa_{N,s}
=2^{-2s}\pi^{-s}\,
\frac{\Gamma\!\bigl(\frac{N-2s}{2}\bigr)}{\Gamma\!\bigl(\frac{N+2s}{2}\bigr)}
\left(\frac{\Gamma(N)}{\Gamma(\frac N2)}\right)^{\!\frac{2s}{N}}.
\end{equation}

\begin{proof}[\textbf{Proof of Proposition \ref{prop:logS-beckner-equivalence}.}]
For a fixed nonzero $v\in C_c^{\infty}(\mathbb R^N)$, define
\[
F_v(s):=\kappa_{N,s}\,\|v\|_{\dot H^s(\mathbb R^N)}^2-\|v\|_{L^{2_s^*}(\mathbb R^N)}^2.
\]
Then \eqref{eq:Sobolev-RN-frac-kappa} is exactly the statement that $F_v(s)\ge 0$ for all
$s\in [0,1)$, while $F_v(0)=0$ since $\kappa_{N,0}=1$, $2_0^*=2$, and $\|v\|_{\dot H^0}^2=\|v\|_2^2$.
Hence the right-hand derivative at $0$ satisfies
\[
F_v'(0^+):=\left.\frac{d}{ds}\right|_{s=0^+}F_v(s)\ge\;0.
\]

We now compute $F_v'(0^+)$ using the first-order expansions as $s\to0^+$.
First, $\kappa_{N,s}=1+s\,a_N+o(s)$ with
\begin{equation}\label{eq:aN}
a_N
=\left.\frac{d}{ds}\kappa_{N,s}\right|_{s=0}
=\frac{2}{N}\ln\!\left(\frac{\Gamma(N)}{\Gamma(\frac N2)}\right)-\ln(4\pi)-2\psi\!\left(\frac N2\right).
\end{equation}
Next,
\[
\|v\|_{\dot H^s(\mathbb R^N)}^2
=\|v\|_2^2+s\,\langle v,(-\Delta)^{\ln} v\rangle+o(s),
\qquad
(-\Delta)^{\ln}:=\left.\frac{d}{ds}(-\Delta)^s\right|_{s=0},
\]
and
\[
\|v\|_{L^{2_s^*}(\mathbb R^N)}^2
=\|v\|_2^2+\frac{2s}{N}\Bigl(\int_{\mathbb R^N}|v|^2\ln|v|^2\,dx-\|v\|_2^2\ln\|v\|_2^2\Bigr)+o(s).
\]
Substituting these expansions into $F_v(s)$ and comparing the coefficients of $s$ gives
\[
0\le F_v'(0^+)
=a_N\|v\|_2^2+\langle v,(-\Delta)^{\ln} v\rangle
-\frac{2}{N}\Bigl(\int_{\mathbb R^N}|v|^2\ln|v|^2\,dx-\|v\|_2^2\ln\|v\|_2^2\Bigr).
\]
Rearranging yields the sharp Euclidean logarithmic inequality gives the
entropy form
\begin{equation}\label{eq:logS-RN-entropy}
\frac{2}{N}\int_{\mathbb{R}^N}\frac{|v|^2}{\|v\|_2^2}
\ln\!\left(\frac{|v|^2}{\|v\|_2^2}\right)\,dx
\le a_N+\frac{\langle v,(-\Delta)^{\ln} v\rangle}{\|v\|_2^2}
\quad{\rm for}\ \,  v\in C_c^{\infty}(\mathbb{R}^N)\setminus\{0\}.
\end{equation}

We now transfer \eqref{eq:logS-RN-entropy} to the sphere via stereographic projection and the
conformal pullback. By Lemma \ref{confcore}, we obtain the sharp logarithmic inequality on $\mathbb{S}^N$:
\begin{equation}\label{eq:logS-SN}
\frac{2}{N}\int_{\mathbb{S}^N}\frac{|u|^2}{\|u\|_2^2}
\ln\!\left(\frac{|u|^2}{\|u\|_2^2}\right)\,dV_g
\le a_N+\frac{\int_{\mathbb{S}^N}u\,\mathscr P_g^{\ln}u\,dV_g}{\|u\|_2^2}
\quad{\rm for}\ \, u\in C^\infty(\mathbb{S}^N)\setminus\{0\},
\end{equation}
where $a_N$ is the explicit constant in \eqref{eq:aN}.  Using the representation \eqref{eq:log-def} we have the energy identity
\begin{equation}\label{eq:Pln-energy-kernel}
\int_{\mathbb S^N}u\,\mathscr P_g^{\ln}u\,dV_g
=\frac{c_N}{2}\iint_{\mathbb S^N\times\mathbb S^N}
\frac{(u(z)-u(\zeta))^2}{|z-\zeta|^N}\,dV_g(z)\,dV_g(\zeta)
+A_N\int_{\mathbb S^N}u^2\,dV_g,
\end{equation}
where
\[
c_N=\pi^{-N/2}\Gamma\!\left(\frac N2\right),
\qquad
A_N=2\,\psi\!\left(\frac N2\right).
\]

Substituting \eqref{eq:Pln-energy-kernel} into \eqref{eq:logS-SN} yields
\begin{align*}
&\frac{2}{N}\left(
\int_{\mathbb S^N}|u|^2\ln|u|^2\,dV_g
-\|u\|_2^2\ln\|u\|_2^2
\right)\notag\\
&\qquad\le
(a_N+A_N)\,\|u\|_2^2
+\frac{c_N}{2}\iint_{\mathbb S^N\times\mathbb S^N}
\frac{(u(z)-u(\zeta))^2}{|z-\zeta|^N}\,dV_g(z)\,dV_g(\zeta).\label{eq:logS-substituted}
\end{align*}
Using \eqref{eq:aN} and $A_N=2\psi(\frac N2)$, we obtain
\[\begin{aligned}
&\frac{2}{N}\left(
\int_{\mathbb S^N}|u|^2\ln|u|^2\,dV_g
-\|u\|_2^2\ln\|u\|_2^2
\right)\notag\\
&\qquad\le
\Biggl[\frac{2}{N}\ln\!\left(\frac{\Gamma(N)}{\Gamma(\frac N2)}\right)-\ln(4\pi)\Biggr]\|u\|_2^2
+\frac{c_N}{2}\iint_{\mathbb S^N\times\mathbb S^N}
\frac{(u(z)-u(\zeta))^2}{|z-\zeta|^N}\,dV_g(z)\,dV_g(\zeta).\label{eq:logS-simplified}
\end{aligned}\]
Recall the surface measure of the unit sphere $|\mathbb S^N|
=\frac{2\pi^{\frac{N+1}{2}}}{\Gamma(\frac{N+1}{2})}.$
Adding and subtracting $\frac{2}{N}\|u\|_2^2\ln|\mathbb S^N|$ on the left-hand side of
\eqref{eq:logS-simplified}, we obtain the equivalent form
\begin{align}
\frac{2}{N}\int_{\mathbb S^N}|u|^2
\ln\!\left(\frac{|u|^2\,|\mathbb S^N|}{\|u\|_2^2}\right)\,dV_g
&\le
\frac{c_N}{2}\iint_{\mathbb S^N\times\mathbb S^N}
\frac{(u(z)-u(\zeta))^2}{|z-\zeta|^N}\,dV_g(z)\,dV_g(\zeta)\notag\\
&\quad+
\Biggl[\frac{2}{N}\ln\!\left(\frac{\Gamma(N)}{\Gamma(\frac N2)}\right)
-\ln(4\pi)
+\frac{2}{N}\ln|\mathbb S^N|
\Biggr]\|u\|_2^2.\label{eq:logS-with-area}
\end{align}

If we normalize $\|u\|_2^2=1$, then the last term in \eqref{eq:logS-with-area} is an additive
constant, then
we arrive at the celebrated  Beckner-type inequality
\[\iint_{\mathbb S^N\times\mathbb S^N}
\frac{(u(z)-u(\zeta))^2}{|z-\zeta|^N}\,dV_g(z)\,dV_g(\zeta)
\ge
C_N\int_{\mathbb S^N}|u|^2
\ln\!\left(|u|^2\,|\mathbb S^N|\right)\,dV_g,
\qquad \|u\|_2=1,\]
where
\[C_N:=\frac{4}{N}\,\frac{\pi^{N/2}}{\Gamma(\frac N2)}.\]
Equivalently, 
\[\iint_{\mathbb S^N\times\mathbb S^N}
\frac{(u(z)-u(\zeta))^2}{|z-\zeta|^N}\,dV_g(z)\,dV_g(\zeta)
\ge
C_N\int_{\mathbb S^N}|u|^2
\ln\!\left(\frac{|u|^2\,|\mathbb S^N|}{\|u\|_2^2}\right)\,dV_g,
\quad{\rm for}\ \,  u\in C^\infty(\mathbb S^N)\setminus \{0\} .\]
\end{proof}

\subsection{Failure of Sharp Fractional--Logarithmic Sobolev Inequality}

We first derive a fractional--logarithmic identity by differentiating the sharp fractional Sobolev
inequality along the family of extremals.

Let $N\ge1$ and $s\in(0,1)$ with $N>2s$, and set $p(s):=2_s^*=\frac{2N}{N-2s}$.
Let $\kappa_{N,s}$ be the sharp constant in \eqref{eq:Sobolev-RN-frac-kappa}. Equivalently,
\begin{equation}\label{eq:kappa-variational}
\frac{1}{\kappa_{N,s}}
=\inf_{v\in \dot H^s(\mathbb R^N)\setminus\{0\}}
\frac{\|v\|_{\dot H^s(\mathbb R^N)}^2}{\|v\|_{L^{p(s)}(\mathbb R^N)}^2}.
\end{equation}
Moreover, equality in  \eqref{eq:kappa-variational}  is attained precisely by the fractional Talenti bubbles,
\[
u_s(x):
= \left(\frac{1}{1+|x|^2}\right)^{\frac{N-2s}{2}}\ \ {\rm for}\ \,  x\in \R^N\quad{\rm with}\ \,  s\in (0,1)
\]
i.e.
\begin{equation}\label{eq:kappa-attained}
  \frac{1}{\kappa_{N,s}}
=
\frac{\|u_s\|_{\dot H^s(\mathbb R^N)}^2}{\|u_s\|_{L^{p(s)}(\mathbb R^N)}^2}.  
\end{equation}

We recall that \(K_\nu\) denotes the modified Bessel function of the second kind, see \cite[p.~415]{aronszajn1961theory}.
We shall use the following asymptotic behaviors:
for \(\nu>0\),
\[
K_\nu(r)\sim 2^{\nu-1}\Gamma(\nu)\,r^{-\nu}
\qquad \text{as } r\downarrow 0,
\]
while
\[
K_\nu(r)\sim \sqrt{\frac{\pi}{2r}}\,e^{-r}
\qquad \text{as } r\to\infty.
\]

\begin{lemma}\label{lem:fourier-bubble}
Let
\begin{equation}\label{extremal}
    u_s(x):=(1+|x|^2)^{-\frac{N-2s}{2}},
\qquad x\in\mathbb R^N,\quad 0<s<1.
\end{equation}
Then
\[
\widehat{u_s}(\xi)
=
\frac{2^{\,1-\frac{N-2s}{2}}}{\Gamma\!\left(\frac{N-2s}{2}\right)}
\,|\xi|^{-s}K_s(|\xi|),
\qquad \xi\in\mathbb R^N\setminus\{0\},
\]
where \(K_s\) denotes the modified Bessel function of the second kind.
\end{lemma}

\begin{proof}
Set
\[
\alpha:=N-2s.
\]
Then \(0<\alpha<N\) and
\[
u_s(x)=(1+|x|^2)^{-\alpha/2}.
\]
By \cite[p.~414]{aronszajn1961theory}, the Bessel kernel
\[
G_\alpha(x)
=
\frac{1}{2^{\frac{N+\alpha-2}{2}}\pi^{N/2}\Gamma(\alpha/2)}
\,|x|^{\frac{\alpha-N}{2}}K_{\frac{N-\alpha}{2}}(|x|)
\]
satisfies
\[
\widehat{G_\alpha}(\xi)=(2\pi)^{-N/2}(1+|\xi|^2)^{-\alpha/2}.
\]
Since \((1+|\xi|^2)^{-\alpha/2}\) is radial and even, Fourier inversion gives
\[
\cF((1+|x|^2)^{-\alpha/2})(\xi)=(2\pi)^{N/2}G_\alpha(\xi).
\]
Therefore,
\[
\widehat{u_s}(\xi)
=
(2\pi)^{N/2}
\frac{1}{2^{\frac{N+\alpha-2}{2}}\pi^{N/2}\Gamma(\alpha/2)}
\,|\xi|^{\frac{\alpha-N}{2}}K_{\frac{N-\alpha}{2}}(|\xi|)=\frac{2^{\,1-\frac{N-2s}{2}}}{\Gamma\!\left(\frac{N-2s}{2}\right)}
\,|\xi|^{-s}K_s(|\xi|).
\]
This completes the proof.
\end{proof}

As a consequence of Lemma \ref{lem:fourier-bubble}, we have
\[
|\xi|^{2s}\,|\widehat{u_s}(\xi)|^2
=
C_{N,s}^2\,K_s(|\xi|)^2,
\]
where
\[
C_{N,s}:=\frac{2^{\,1-\frac{N-2s}{2}}}{\Gamma\!\left(\frac{N-2s}{2}\right)}.
\]
Hence
\[
|\xi|^{2s}\,|\ln |\xi|^2|\,|\widehat{u_s}(\xi)|^2
=
C_{N,s}^2\,|\ln |\xi|^2|\,K_s(|\xi|)^2.
\]

We claim that
\[
\int_{\mathbb R^N} |\xi|^{2s}\,|\ln |\xi|^2|\,|\widehat{u_s}(\xi)|^2\,d\xi<+\infty.
\]
Indeed, by radial symmetry it is enough to show
\[
\int_0^\infty r^{2s+N-1} |\ln r^2|\, K_s(r)^2\,dr<+\infty.
\]

As \(r\downarrow 0\), since \(0<s<1\), we have
\[
K_s(r)\sim r^{-s},
\]
and therefore
\[
r^{N-1} |\ln r^2|\,K_s(r)^2
\sim r^{N-1-2s}|\ln r|.
\]
Since \(N-2s>0\), it follows that
\[
\int_0^1 r^{N-1-2s}|\ln r|\,dr<\infty.
\]

As \(r\to\infty\), we use
\[
K_s(r)\sim \sqrt{\frac{\pi}{2r}}\,e^{-r},
\]
so that
\[
r^{N-1} |\ln r^2|\,K_s(r)^2
\sim r^{N-2}(\ln r)e^{-2r},
\]
which is integrable on \((1,\infty)\).

Thus
\[
\int_{\mathbb R^N} |\xi|^{2s}\,|\ln |\xi|^2|\,|\widehat{u_s}(\xi)|^2\,d\xi<\infty.
\]
In particular,
\[
\int_{\mathbb R^N} |\xi|^{2s}\,\ln |\xi|^2\,|\widehat{u_s}(\xi)|^2\,d\xi
\]
is well defined and finite.

Define the two normalized quantities
\[
\mathcal E_s(v):=\frac{\int_{\mathbb R^N}|\xi|^{2s}\,|\widehat v(\xi)|^2\,d\xi}{\|v\|_{L^{p(s)}}^2},
\qquad
\mathcal L_s(v):=\frac{\int_{\mathbb R^N}|\xi|^{2s}\ln|\xi|^2\,|\widehat v(\xi)|^2\,d\xi}{\|v\|_{L^{p(s)}}^2}.
\]

\begin{proof}[\textbf{Proof of Proposition \ref{prop:fraclog-identity-variational}.}]
Let
\[
A(s):=\int_{\mathbb R^N}|\xi|^{2s}\,|\widehat u_s(\xi)|^2\,d\xi,
\qquad
B(s):=\|u_s\|_{L^{p(s)}}^2,
\]
so that
\[
\mathcal E_s(u_s)=\frac{A(s)}{B(s)}.
\]
By \eqref{eq:kappa-variational}--\eqref{eq:kappa-attained}, we have
\[
\mathcal E_s(u_s)=\frac{1}{\kappa_{N,s}}.
\]
Differentiating with respect to \(s\), we obtain
\begin{equation}\label{eq:diff-Es}
\frac{d}{ds}\mathcal E_s(u_s)
=
-\frac{\kappa'_{N,s}}{\kappa_{N,s}^2}.
\end{equation}

Since
\[
\mathcal E_s(u_s)=\frac{A(s)}{B(s)},
\]
the quotient rule yields
\begin{equation}\label{eq:quotient-derivative-new}
\frac{d}{ds}\mathcal E_s(u_s)
=
\frac{A'(s)}{B(s)}-\frac{A(s)B'(s)}{B(s)^2}
=
\mathcal E_s(u_s)\left(\frac{A'(s)}{A(s)}-\frac{B'(s)}{B(s)}\right).
\end{equation}

\medskip
\noindent\textbf{Step 1: Derivative of the numerator.}
Since
\[
A(s)=\int_{\mathbb R^N}|\xi|^{2s}\,|\widehat u_s(\xi)|^2\,d\xi,
\]
we may differentiate under the integral sign to obtain
\begin{equation}\label{eq:Aprime-fourier}
A'(s)
=
\int_{\mathbb R^N} |\xi|^{2s}\ln|\xi|^2\,|\widehat u_s(\xi)|^2\,d\xi
+
2\int_{\mathbb R^N} |\xi|^{2s}\,\widehat u_s(\xi)\,\partial_s\widehat u_s(\xi)\,d\xi.
\end{equation}
Here the differentiation under the integral sign is justified by the explicit representation of \(\widehat u_s\) together with the standard asymptotic bounds for \(K_s\); we omit the verification. By the definition of \(\mathcal L_s(u_s)\), this becomes
\begin{equation}\label{eq:Aprime-fourier-L}
A'(s)
=
\mathcal L_s(u_s)\,\|u_s\|_{L^{p(s)}}^2
+
2\int_{\mathbb R^N} |\xi|^{2s}\,\widehat u_s(\xi)\,\partial_s\widehat u_s(\xi)\,d\xi.
\end{equation}

\medskip
\noindent\textbf{Step 2: Euler--Lagrange equation.}
Since \(u_s\) is an extremal for \(\mathcal E_s\), there exists a Lagrange multiplier
\(\Lambda_s\in\mathbb R\) such that
\begin{equation}\label{eq:EL-fixed-s-new}
(-\Delta)^s u_s=\Lambda_s\,|u_s|^{p(s)-2}u_s
\quad\text{in the weak sense on }\mathbb R^N.
\end{equation}
Testing \eqref{eq:EL-fixed-s-new} against \(u_s\), we obtain
\begin{equation}\label{eq:test-us-new}
A(s)
=
\int_{\mathbb R^N} |\xi|^{2s}|\widehat u_s(\xi)|^2\,d\xi
=
\Lambda_s\int_{\mathbb R^N}|u_s|^{p(s)}\,dx.
\end{equation}
Testing \eqref{eq:EL-fixed-s-new} against \(\dot u_s:=\partial_s u_s\), and using Plancherel, gives
\begin{equation}\label{eq:test-dot-new}
\int_{\mathbb R^N} |\xi|^{2s}\,\widehat u_s(\xi)\,\partial_s\widehat u_s(\xi)\,d\xi
=
\Lambda_s\int_{\mathbb R^N}|u_s|^{p(s)-2}u_s\,\dot u_s\,dx.
\end{equation}

\medskip
\noindent\textbf{Step 3: Derivative of the denominator.}
Let
\[
I(s):=\int_{\mathbb R^N}|u_s|^{p(s)}\,dx,
\qquad\text{so that}\qquad
B(s)=I(s)^{\frac{2}{p(s)}}.
\]
Since
\[
|u_s(x)|^{p(s)}
=
\left((1+|x|^2)^{-\frac{N-2s}{2}}\right)^{\frac{2N}{N-2s}}
=
(1+|x|^2)^{-N},
\]
the quantity \(I(s)\) is independent of \(s\). Hence
\[
I'(s)=0
\qquad\text{and}\qquad
\frac{B'(s)}{B(s)}=-\frac{2p'(s)}{p(s)^2}\ln I(s).
\]
Moreover,
\[
\Ent_{p(s)}(u_s)
=
\frac{\int_{\mathbb R^N}|u_s|^{p(s)}\ln |u_s|^{p(s)}\,dx}{I(s)}-\ln I(s).
\]
Using again that \(|u_s|^{p(s)}=(1+|x|^2)^{-N}\), we see that \(\Ent_{p(s)}(u_s)\) is independent of \(s\). Combining \eqref{eq:test-us-new} and \eqref{eq:test-dot-new}, we obtain
\[
\frac{2\int_{\mathbb R^N} |\xi|^{2s}\widehat u_s\,\partial_s\widehat u_s\,d\xi}{A(s)}
=
\frac{2\Lambda_s\int_{\mathbb R^N}|u_s|^{p(s)-2}u_s\,\dot u_s\,dx}{\Lambda_s\int_{\mathbb R^N}|u_s|^{p(s)}\,dx}
=
\frac{2}{I(s)}\int_{\mathbb R^N}|u_s|^{p(s)-2}u_s\,\dot u_s\,dx.
\]
Therefore, dividing \eqref{eq:Aprime-fourier-L} by \(A(s)\), we get
\[
\frac{A'(s)}{A(s)}
=
\frac{\mathcal L_s(u_s)}{\mathcal E_s(u_s)}
+
\frac{2}{I(s)}\int_{\mathbb R^N}|u_s|^{p(s)-2}u_s\,\dot u_s\,dx.
\]
On the other hand, differentiating
\[
I(s)=\int_{\mathbb R^N}|u_s|^{p(s)}\,dx
\]
and using \(I'(s)=0\), we find
\[
0
=
p'(s)\int_{\mathbb R^N}|u_s|^{p(s)}\ln|u_s|\,dx
+
p(s)\int_{\mathbb R^N}|u_s|^{p(s)-2}u_s\,\dot u_s\,dx.
\]
Hence
\[
\frac{2}{I(s)}\int_{\mathbb R^N}|u_s|^{p(s)-2}u_s\,\dot u_s\,dx
=
-\frac{2p'(s)}{p(s)I(s)}
\int_{\mathbb R^N}|u_s|^{p(s)}\ln|u_s|\,dx.
\]
Since
\[
\ln|u_s|^{p(s)}=p(s)\ln|u_s|,
\]
it follows that
\[
\frac{2}{I(s)}\int_{\mathbb R^N}|u_s|^{p(s)-2}u_s\,\dot u_s\,dx
=
-\frac{2p'(s)}{p(s)^2I(s)}
\int_{\mathbb R^N}|u_s|^{p(s)}\ln|u_s|^{p(s)}\,dx.
\]
Therefore
\[
\frac{A'(s)}{A(s)}
=
\frac{\mathcal L_s(u_s)}{\mathcal E_s(u_s)}
-
\frac{2p'(s)}{p(s)^2I(s)}
\int_{\mathbb R^N}|u_s|^{p(s)}\ln|u_s|^{p(s)}\,dx.
\]
Subtracting \(\frac{B'(s)}{B(s)}=-\frac{2p'(s)}{p(s)^2}\ln I(s)\), we arrive at
\begin{equation}\label{eq:AminusB-final}
\frac{A'(s)}{A(s)}-\frac{B'(s)}{B(s)}
=
\frac{\mathcal L_s(u_s)}{\mathcal E_s(u_s)}
-
\frac{2p'(s)}{p(s)^2}
\left(
\frac{\int_{\mathbb R^N}|u_s|^{p(s)}\ln|u_s|^{p(s)}\,dx}{I(s)}
-\ln I(s)
\right).
\end{equation}
By the definition of \(\Ent_{p(s)}(u_s)\), this becomes
\begin{equation}\label{eq:AminusB-final-entropy}
\frac{A'(s)}{A(s)}-\frac{B'(s)}{B(s)}
=
\frac{\mathcal L_s(u_s)}{\mathcal E_s(u_s)}
-
\frac{2p'(s)}{p(s)^2}\,\Ent_{p(s)}(u_s).
\end{equation}

\medskip
\noindent\textbf{Step 4: Conclusion.}
Since
\[
p(s)=\frac{2N}{N-2s},
\qquad
p'(s)=\frac{4N}{(N-2s)^2},
\]
we have
\[
\frac{2p'(s)}{p(s)^2}=\frac{2}{N}.
\]
Substituting \eqref{eq:AminusB-final-entropy} into \eqref{eq:quotient-derivative-new}, and using
\[
\mathcal E_s(u_s)=\frac1{\kappa_{N,s}},
\]
we obtain
\[
\frac{d}{ds}\mathcal E_s(u_s)
=
\frac1{\kappa_{N,s}}
\left(
\frac{\mathcal L_s(u_s)}{\mathcal E_s(u_s)}
-
\frac{2}{N}\Ent_{p(s)}(u_s)
\right).
\]
Combining this with \eqref{eq:diff-Es}, we arrive at
\[
-\frac{\kappa'_{N,s}}{\kappa_{N,s}^2}
=
\frac1{\kappa_{N,s}}
\left(
\frac{\mathcal L_s(u_s)}{\mathcal E_s(u_s)}
-
\frac{2}{N}\Ent_{p(s)}(u_s)
\right).
\]
Multiplying both sides by \(\kappa_{N,s}\) gives
\[
\frac{2}{N}\Ent_{p(s)}(u_s)
=
\frac{\kappa'_{N,s}}{\kappa_{N,s}}
+
\frac{\mathcal L_s(u_s)}{\mathcal E_s(u_s)},
\]
using \eqref{eq:kappa-attained}, we otain \eqref{eq:sharp-fraclog-identity-expanded}.
\end{proof}

\begin{rmk}\label{examd}
Formally letting \(s\to0^+\) in \eqref{eq:sharp-fraclog-identity-expanded}, one recovers the sharp logarithmic identity. Indeed, as \(s\to0^+\), one has
\[
p(s)\to2,
\qquad
u_s(x)=\left(\frac{1}{1+|x|^2}\right)^{\frac{N-2s}{2}}
\longrightarrow
v(x):=\left(\frac{1}{1+|x|^2}\right)^{\frac N2},
\]
and correspondingly \(\Ent_{p(s)}(u_s)\) tends formally to the \(L^2\)-entropy
\[
\Ent_2(v)
:=
\int_{\mathbb R^N}\frac{|v|^2}{\|v\|_2^2}
\ln\!\left(\frac{|v|^2}{\|v\|_2^2}\right)\,dx.
\]
Moreover, by the normalization of the sharp fractional Sobolev constant,
\[
\kappa_{N,0}=1,
\qquad
\kappa'_{N,0}=a_N.
\]
Since
\[
\mathcal E_0(v)
=
\frac{\int_{\mathbb R^N}|\widehat v(\xi)|^2\,d\xi}{\|v\|_2^2}
=1,
\]
identity \eqref{eq:sharp-fraclog-identity-expanded} reduces formally, at \(s=0\), to
\[
\frac{2}{N}\,\Ent_2(v)
=
a_N+\frac{\int_{\mathbb R^N}\ln|\xi|^2\,|\widehat v(\xi)|^2\,d\xi}{\|v\|_2^2}.
\]
This is precisely the logarithmic entropy identity corresponding to the equality case of \eqref{eq:logS-RN-entropy}. 
\end{rmk}

\begin{rmk}\label{rem:F-derivative-normalized}
The identity
\[
F_v'(s)
=\kappa'_{N,s}\,\int_{\mathbb R^N}|\xi|^{2s}\,|\widehat v(\xi)|^2\,d\xi
+\kappa_{N,s}\,\int_{\mathbb R^N}|\xi|^{2s}\ln|\xi|^2\,|\widehat v(\xi)|^2\,d\xi
-\frac{2}{N}\,\|v\|_{L^{p(s)}(\mathbb R^N)}^2\,\Ent_{p(s)}(v)
\]
can be rewritten in normalized form. Dividing by \(\|v\|_{L^{p(s)}(\mathbb R^N)}^2\), we obtain
\begin{equation}\label{eq:Fprime-normalized}
\frac{2}{N}\,\Ent_{p(s)}(v)
-\kappa'_{N,s}\,\frac{\int_{\mathbb R^N}|\xi|^{2s}\,|\widehat v(\xi)|^2\,d\xi}{\|v\|_{L^{p(s)}(\mathbb R^N)}^2}
-\kappa_{N,s}\,\frac{\int_{\mathbb R^N}|\xi|^{2s}\ln|\xi|^2\,|\widehat v(\xi)|^2\,d\xi}{\|v\|_{L^{p(s)}(\mathbb R^N)}^2}
=
-\frac{F_v'(s)}{\|v\|_{L^{p(s)}(\mathbb R^N)}^2}.
\end{equation}
In particular, fix \(s_0\in(0,1)\) and let \(v\) be an extremal for the sharp inequality \eqref{eq:Sobolev-RN-frac-kappa} at order \(s_0\). Then
\[
\frac{\int_{\mathbb R^N}|\xi|^{2s_0}|\widehat v(\xi)|^2\,d\xi}{\|v\|_{L^{p(s_0)}(\mathbb R^N)}^2}
=\frac1{\kappa_{N,s_0}}.
\]
Moreover, since \(v\) is fixed and \(F_v(s)\ge 0\) for all \(s\), while \(F_v(s_0)=0\), the point \(s_0\) is a minimum of \(F_v\). Hence
\[
F_v'(s_0)=0.
\]
Evaluating \eqref{eq:Fprime-normalized} at \(s=s_0\), we recover exactly the sharp fractional--logarithmic identity in Proposition~\ref{prop:fraclog-identity-variational}.
\end{rmk}

We next prove that the naive sharp fractional--logarithmic Sobolev inequality fails in the form stated in Theorem \ref{thm:naive-fraclog-fails}.

\begin{proof}[\textbf{Proof of Theorem \ref{thm:naive-fraclog-fails}.}]
By Remark~\ref{rem:F-derivative-normalized}, the inequality \eqref{eq:naive-fraclog-ineq} is equivalent to
\eqref{eq:Fprime-nonneg-thm}. We argue by contradiction and assume that
\[
F_v'(s)\ge 0
\qquad\text{for all }s\in(0,1),\ \forall\, v\in \dot H^s(\mathbb R^N)\setminus\{0\}.
\]

Fix \(s_0\in(0,1)\), and let \(u_{s_0}\) be a Talenti bubble, i.e. an extremal for the sharp Sobolev inequality at order \(s_0\). Then
\[
F_{u_{s_0}}(s_0)=0.
\]
Moreover, by Proposition~\ref{prop:fraclog-identity-variational}, the sharp fractional--logarithmic identity holds at \(s_0\), and therefore
\[
F'_{u_{s_0}}(s_0)=0.
\]

Since \(u_{s_0}\) is now fixed, the assumption \eqref{eq:Fprime-nonneg-thm} implies that the function
\[
s\longmapsto F_{u_{s_0}}(s)
\]
is nondecreasing on \((0,1)\). Because \(F_{u_{s_0}}(s_0)=0\), it follows that
\[
F_{u_{s_0}}(s)\le 0
\qquad\text{for every }s\in(0,s_0).
\]
On the other hand, by the sharp fractional Sobolev inequality we always have
\[
F_{u_{s_0}}(s)\ge 0
\qquad\text{for every }s\in(0,1).
\]
Hence
\[
F_{u_{s_0}}(s)=0
\qquad\text{for every }s\in(0,s_0).
\]

Therefore the fixed function \(u_{s_0}\) is an extremal for the sharp Sobolev inequality not only at the order \(s_0\), but for every \(s\in(0,s_0)\), this is impossible. This contradiction shows that \eqref{eq:naive-fraclog-ineq} cannot hold for all \(s\in(0,1)\), and hence the naive inequality \eqref{eq:Fprime-nonneg-thm} fails in general.
\end{proof}

We now transfer the Euclidean extremal identity to the sphere via stereographic projection and the conformal pullback.

\begin{proof}[\textbf{Proof of Proposition \ref{prop:fraclog-identity-sphere-correct}.}]
We derive the spherical identity by transporting the Euclidean extremal identity through stereographic projection.

Let \(x=\sigma(\omega)\), where
\[
\sigma:\mathbb S^N\setminus\{-e_{N+1}\}\to\mathbb R^N
\]
is the stereographic projection, and let
\[
\phi(x)=\frac{2}{1+|x|^2}.
\]
By definition of the conformal pullback,
\[
u_s(x)=\phi(x)^{\frac{N-2s}{2}}U_s(\omega),
\qquad \omega=\sigma^{-1}(x),
\qquad dV_g(\omega)=\phi(x)^N\,dx.
\]

Since
\[
\frac{N-2s}{2}\,p(s)=N,
\]
we have
\[
|u_s(x)|^{p(s)}=\phi(x)^N|U_s(\omega)|^{p(s)},
\]
and therefore
\[
|u_s(x)|^{p(s)}\,dx=|U_s(\omega)|^{p(s)}\,dV_g(\omega).
\]
It follows that
\[
\|u_s\|_{L^{p(s)}(\mathbb R^N)}
=
\|U_s\|_{L^{p(s)}(\mathbb S^N)}.
\]

Next, from
\[
\frac{|u_s(x)|^{p(s)}}{\|u_s\|_{L^{p(s)}(\mathbb R^N)}^{p(s)}}
=
\phi(x)^N
\frac{|U_s(\omega)|^{p(s)}}{\|U_s\|_{L^{p(s)}(\mathbb S^N)}^{p(s)}},
\]
we obtain
\begin{align*}
\Ent_{p(s)}(u_s)
&=
\int_{\mathbb R^N}
\frac{|u_s(x)|^{p(s)}}{\|u_s\|_{L^{p(s)}(\mathbb R^N)}^{p(s)}}
\ln\!\left(
\frac{|u_s(x)|^{p(s)}}{\|u_s\|_{L^{p(s)}(\mathbb R^N)}^{p(s)}}
\right)\,dx
\\
&=
\Ent_{p(s)}^{\mathbb S^N}(U_s)
+
N\int_{\mathbb S^N}
\frac{|U_s(\omega)|^{p(s)}}{\|U_s\|_{L^{p(s)}(\mathbb S^N)}^{p(s)}}
\ln\phi(\sigma(\omega))\,dV_g(\omega).
\end{align*}
Hence
\begin{equation}\label{eq:entropy-correction-sphere-new}
\Ent_{p(s)}^{\mathbb S^N}(U_s)
=
\Ent_{p(s)}(u_s)
-
N\int_{\mathbb S^N}
\frac{|U_s(\omega)|^{p(s)}}{\|U_s\|_{L^{p(s)}(\mathbb S^N)}^{p(s)}}
\ln\phi(\sigma(\omega))\,dV_g(\omega).
\end{equation}

We now treat the energy terms. By the conformal intertwining relation for \(\mathscr P_g^s \),
\[
\langle u_s,(-\Delta)^s u_s\rangle_{L^2(\mathbb R^N)}
=
\langle U_s,\mathscr P_g^s U_s\rangle_{\mathbb S^N}.
\]
Equivalently,
\begin{equation}\label{eq:Ps-energy-sphere}
\int_{\mathbb R^N}|\xi|^{2s}|\widehat u_s(\xi)|^2\,d\xi
=
\langle U_s,\mathscr P_g^s U_s\rangle_{\mathbb S^N}.
\end{equation}

For the fractional--logarithmic operator, arguing as in Proposition~\ref{prop:sphere-RN-frac-log}, the weak intertwining formula gives
\begin{align}
\langle U_s,\mathscr P^{s+\ln}_g U_s\rangle_{\mathbb S^N}
&=
\langle u_s,(-\Delta)^{s+\ln}u_s\rangle_{L^2(\mathbb R^N)}
\notag\\
&\quad
-2\big\langle (-\Delta)^{s/2}u_s,\,
(-\Delta)^{s/2}\bigl((\ln\phi)\,u_s\bigr)\big\rangle_{L^2(\mathbb R^N)}.
\label{eq:weak-intertwining-frac-log-new}
\end{align}
Hence
\begin{equation}\label{eq:log-energy-correction-new}
\langle u_s,(-\Delta)^{s+\ln}u_s\rangle_{L^2(\mathbb R^N)}
=
\langle U_s,\mathscr P^{s+\ln}_g U_s\rangle_{\mathbb S^N}
+
2\big\langle (-\Delta)^{s/2}u_s,\,
(-\Delta)^{s/2}\bigl((\ln\phi)\,u_s\bigr)\big\rangle_{L^2(\mathbb R^N)}.
\end{equation}
By the weak form of the intertwining relation for \(\mathscr P_g^s \),
\begin{equation}\label{eq:ps-correction-sphere-new}
\big\langle (-\Delta)^{s/2}u_s,\,
(-\Delta)^{s/2}\bigl((\ln\phi)\,u_s\bigr)\big\rangle_{L^2(\mathbb R^N)}
=
\int_{\mathbb S^N}
\ln\phi(\sigma(\omega))\,U_s(\omega)\,\mathscr P_g^s U_s(\omega)\,dV_g(\omega).
\end{equation}

Now we start from the Euclidean extremal identity in the form
\[
\frac{2}{N}\,\Ent_{p(s)}(u_s)
=
\kappa'_{N,s}\,
\frac{\int_{\mathbb R^N}|\xi|^{2s}|\widehat u_s(\xi)|^2\,d\xi}
{\|u_s\|_{L^{p(s)}(\mathbb R^N)}^2}
+
\kappa_{N,s}\,
\frac{\int_{\mathbb R^N}|\xi|^{2s}\ln|\xi|^2\,|\widehat u_s(\xi)|^2\,d\xi}
{\|u_s\|_{L^{p(s)}(\mathbb R^N)}^2}.
\]
Using
\[
\int_{\mathbb R^N}|\xi|^{2s}|\widehat u_s(\xi)|^2\,d\xi
=
\langle u_s,(-\Delta)^s u_s\rangle_{L^2(\mathbb R^N)},
\]
and
\[
\int_{\mathbb R^N}|\xi|^{2s}\ln|\xi|^2\,|\widehat u_s(\xi)|^2\,d\xi
=
\langle u_s,(-\Delta)^{s+\ln}u_s\rangle_{L^2(\mathbb R^N)},
\]
together with \eqref{eq:Ps-energy-sphere}, \eqref{eq:log-energy-correction-new}, and
\[
\|u_s\|_{L^{p(s)}(\mathbb R^N)}
=
\|U_s\|_{L^{p(s)}(\mathbb S^N)},
\]
we get
\begin{align*}
\frac{2}{N}\,\Ent_{p(s)}(u_s)
&=
\kappa'_{N,s}\,
\frac{\langle U_s,\mathscr P_g^s U_s\rangle_{\mathbb S^N}}
{\|U_s\|_{L^{p(s)}(\mathbb S^N)}^2}
+
\kappa_{N,s}\,
\frac{\langle U_s,\mathscr P^{s+\ln}_g U_s\rangle_{\mathbb S^N}}
{\|U_s\|_{L^{p(s)}(\mathbb S^N)}^2}
\\
&\qquad
+\frac{2\kappa_{N,s}}{\|U_s\|_{L^{p(s)}(\mathbb S^N)}^2}
\big\langle (-\Delta)^{s/2}u_s,\,
(-\Delta)^{s/2}\bigl((\ln\phi)\,u_s\bigr)\big\rangle_{L^2(\mathbb R^N)}.
\end{align*}
Using \eqref{eq:ps-correction-sphere-new}, this becomes
\begin{align*}
\frac{2}{N}\,\Ent_{p(s)}(u_s)
&=
\kappa'_{N,s}\,
\frac{\langle U_s,\mathscr P_g^s U_s\rangle_{\mathbb S^N}}
{\|U_s\|_{L^{p(s)}(\mathbb S^N)}^2}
+
\kappa_{N,s}\,
\frac{\langle U_s,\mathscr P^{s+\ln}_g U_s\rangle_{\mathbb S^N}}
{\|U_s\|_{L^{p(s)}(\mathbb S^N)}^2}
\\
&\qquad
+\frac{2\kappa_{N,s}}{\|U_s\|_{L^{p(s)}(\mathbb S^N)}^2}
\int_{\mathbb S^N}
\ln\phi(\sigma(\omega))\,U_s(\omega)\,\mathscr P_g^s U_s(\omega)\,dV_g(\omega).
\end{align*}
Finally, combining this identity with \eqref{eq:entropy-correction-sphere-new}, we obtain
\eqref{eq:sharp-fraclog-identity-sphere-correct}.
\end{proof}

\subsection{New Sharp Fractional--Logarithmic Sobolev Inequalities}

\subsubsection{First version using Pitt's inequality}
% Convention A (with 2\pi in the phase; unitary on L^2):
\newcommand{\Ftwopi}{\mathcal F_{2\pi}}
\newcommand{\Finfty}{\mathcal F}

% (A) Fourier transform used in the quoted theorem:
%   (\Ftwopi f)(\xi)=\int_{\R^n} e^{-2\pi i x\cdot \xi} f(x)\,dx.
% It satisfies Plancherel:  \|f\|_2=\|\Ftwopi f\|_2.

% Convention B (no 2\pi in the phase; also chosen unitary on L^2):
%   (\Finfty f)(\xi)=(2\pi)^{-n/2}\int_{\R^n} e^{- i x\cdot \xi} f(x)\,dx,
% so that Plancherel holds: \|f\|_2=\|\Finfty f\|_2.
%
% Relation between the two unitary conventions:
%   (\Ftwopi f)(\xi)=(2\pi)^{n/2}(\Finfty f)(2\pi\xi).

We begin with the following classical logarithmic uncertainty principle due to Beckner, which will serve as a basic point of comparison for the fractional--logarithmic inequalities considered below.

\begin{theorem}\cite[Theorem 1.3]{beckner1995pitt}\label{beckner}
Let $u\in\mathcal S(\R^N)$ with $\|u\|_{2}=1$.
Then
\[\frac{n}{2}\int_{\R^N}\ln|\xi|\,|\widehat u(\xi)|^{2}\,d\xi
\ \ge\
\int_{\R^N}\ln|u(x)|\,|u(x)|^{2}\,dx
 +B_N,\]
where the sharp constant is
\[B_N
=
\frac{N}{2}\psi\,\Big(\frac{N}{2}\Big)
-\frac{N}{4}\ln\pi
-\frac12\ln\!\Big(\frac{\Gamma(N)}{\Gamma(\frac N2)}\Big)+\frac{N}{2}\ln(2\pi).\]
Moreover, up to conformal automorphisms, the extremals are of the form
\[
u(x)=A\,(1+|x|^2)^{-N/2}.
\]
\end{theorem}

We next establish the convergence of the entropy functional along a smooth regularization of \((-\Delta)^{s/2}u\).

\begin{lemma}\label{lem:entropy-convergence-feps}
Let \(0<s<1,N>2s\) and let \(u\in \mathcal S(\mathbb R^N)\). For \(\varepsilon\in(0,1)\), define
\(f_\varepsilon\in \mathcal S(\mathbb R^N)\) by
\[
\widehat f_\varepsilon(\xi):=(|\xi|^2+\varepsilon)^{s/2}\widehat u(\xi),
\qquad \xi\in\mathbb R^N,
\]
and define \(f\in L^2(\mathbb R^N)\) by
\[
\widehat f(\xi):=|\xi|^s\widehat u(\xi).
\]
Then
\[
\int_{\mathbb R^N} |f_\varepsilon(x)|^2 \ln |f_\varepsilon(x)|\,dx
\longrightarrow
\int_{\mathbb R^N} |f(x)|^2 \ln |f(x)|\,dx
\qquad\text{as }\varepsilon\downarrow0.
\]

Moreover, if \(f\not\equiv 0\), then
\[
h_\varepsilon:=\frac{f_\varepsilon}{\|f_\varepsilon\|_2},
\qquad
h:=\frac{f}{\|f\|_2},
\]
satisfy
\[
\int_{\mathbb R^N} |h_\varepsilon(x)|^2 \ln |h_\varepsilon(x)|\,dx
\longrightarrow
\int_{\mathbb R^N} |h(x)|^2 \ln |h(x)|\,dx
\qquad\text{as }\varepsilon\downarrow0.
\]
\end{lemma}

\begin{proof}
We divide the proof into several steps.

\medskip
\noindent
\textbf{Step 1: \(f_\varepsilon\to f\) in \(L^\infty(\mathbb R^n)\) and in \(L^2(\mathbb R^n)\).}

Since \(u\in \mathcal S(\mathbb R^n)\), we have \(\widehat u\in \mathcal S(\mathbb R^n)\). Moreover,
for every \(\xi\in\mathbb R^n\) and every \(\varepsilon\in(0,1)\),
\[
\bigl|(|\xi|^2+\varepsilon)^{s/2}\widehat u(\xi)\bigr|
\le
(1+|\xi|^2)^{s/2}|\widehat u(\xi)|.
\]
Because \((1+|\xi|^2)^{s/2}|\widehat u(\xi)|\in L^1(\mathbb R^n)\), the dominated convergence theorem yields
\[
\widehat f_\varepsilon \to \widehat f
\qquad\text{in }L^1(\mathbb R^n).
\]
Hence, by Fourier inversion,
\[
\|f_\varepsilon-f\|_{L^\infty(\mathbb R^n)}
\le
\|\widehat f_\varepsilon-\widehat f\|_{L^1(\mathbb R^n)}
\longrightarrow 0.
\]
Similarly, since
\[
\bigl|(|\xi|^2+\varepsilon)^{s/2}\widehat u(\xi)\bigr|^2
\le
(1+|\xi|^2)^s |\widehat u(\xi)|^2
\in L^1(\mathbb R^n),
\]
the dominated convergence theorem also gives
\[
\widehat f_\varepsilon \to \widehat f
\qquad\text{in }L^2(\mathbb R^n),
\]
and therefore, by Plancherel,
\[
\|f_\varepsilon-f\|_{L^2(\mathbb R^n)}\to0.
\]

\medskip
\noindent
\textbf{Step 2: Uniform \(L^\infty\) and \(L^2\) bounds.}

From the previous estimates,
\[
\sup_{\varepsilon\in(0,1)} \|\widehat f_\varepsilon\|_{L^1(\mathbb R^n)}
\le
\int_{\mathbb R^n} (1+|\xi|^2)^{s/2}|\widehat u(\xi)|\,d\xi
=:A_\infty<\infty,
\]
hence
\[
\sup_{\varepsilon\in(0,1)} \|f_\varepsilon\|_{L^\infty(\mathbb R^n)}
\le A_\infty.
\]
Likewise,
\[
\sup_{\varepsilon\in(0,1)} \|f_\varepsilon\|_{L^2(\mathbb R^n)}^2
\le
\int_{\mathbb R^n} (1+|\xi|^2)^s |\widehat u(\xi)|^2\,d\xi<\infty.
\]

\medskip
\noindent
\textbf{Step 3: A uniform fractional moment bound.}

Fix \(\alpha\in(0,s/2)\). We claim that
\begin{equation}\label{eq:uniform-moment-feps}
\sup_{\varepsilon\in(0,1)}
\int_{\mathbb R^n} |x|^{2\alpha}|f_\varepsilon(x)|^2\,dx
<\infty.
\end{equation}

To prove this, we use the standard fractional Fourier moment estimate: for every
\(v\in \mathcal S(\mathbb R^n)\) and every \(\alpha\in(0,1)\), see \cite[Proposition 3.4]{di2012hitchhiker}
\[
\||x|^\alpha v\|_{L^2(\mathbb R^n)}^2
\le
C_{n,\alpha}\,[\widehat v]_{\dot H^\alpha(\mathbb R^n)}^2,
\]
where
\[
[\phi]_{\dot H^\alpha(\mathbb R^n)}^2
:=
\iint_{\mathbb R^n\times\mathbb R^n}
\frac{|\phi(\xi)-\phi(\eta)|^2}{|\xi-\eta|^{n+2\alpha}}\,d\xi d\eta.
\]
Therefore it suffices to prove that
\[
\sup_{\varepsilon\in(0,1)}
[\widehat f_\varepsilon]_{\dot H^\alpha(\mathbb R^n)}
<\infty.
\]

Write
\[
\widehat f_\varepsilon(\xi)=g_\varepsilon(\xi)\widehat u(\xi),
\qquad
g_\varepsilon(\xi):=(|\xi|^2+\varepsilon)^{s/2}.
\]
We split the seminorm into the regions
\[
D_{\mathrm{out}}:=\{(\xi,\eta)\in\mathbb R^n\times\mathbb R^n:\ |\xi-\eta|>1\},
\]
and
\[
D_{\mathrm{in}}:=\{(\xi,\eta)\in\mathbb R^n\times\mathbb R^n:\ |\xi-\eta|\le1\}.
\]

For the outer region, using \(|a-b|^2\le 2|a|^2+2|b|^2\) and symmetry,
\[
\begin{aligned}
\iint_{D_{\mathrm{out}}}
\frac{|\widehat f_\varepsilon(\xi)-\widehat f_\varepsilon(\eta)|^2}{|\xi-\eta|^{n+2\alpha}}
\,d\xi d\eta
&\le
4 \int_{\mathbb R^n} |\widehat f_\varepsilon(\xi)|^2
\left(\int_{|\eta-\xi|>1}\frac{d\eta}{|\xi-\eta|^{n+2\alpha}}\right)d\xi \\
&\le
C_\alpha \int_{\mathbb R^n} (|\xi|^2+\varepsilon)^s |\widehat u(\xi)|^2\,d\xi \\
&\le
C_\alpha \int_{\mathbb R^n} (1+|\xi|^2)^s |\widehat u(\xi)|^2\,d\xi
\le C.
\end{aligned}
\]

For the inner region, write \(\xi=\eta+w\), with \(|w|\le1\). Then
\[
\widehat f_\varepsilon(\eta+w)-\widehat f_\varepsilon(\eta)
=
g_\varepsilon(\eta+w)\bigl(\widehat u(\eta+w)-\widehat u(\eta)\bigr)
+
\widehat u(\eta)\bigl(g_\varepsilon(\eta+w)-g_\varepsilon(\eta)\bigr).
\]
Hence
\[
|\widehat f_\varepsilon(\eta+w)-\widehat f_\varepsilon(\eta)|^2
\le
2I_{1,\varepsilon}(\eta,w)+2I_{2,\varepsilon}(\eta,w),
\]
where
\[
I_{1,\varepsilon}(\eta,w)
:=
|g_\varepsilon(\eta+w)|^2\,|\widehat u(\eta+w)-\widehat u(\eta)|^2,
\]
and
\[
I_{2,\varepsilon}(\eta,w)
:=
|\widehat u(\eta)|^2\,|g_\varepsilon(\eta+w)-g_\varepsilon(\eta)|^2.
\]

We first estimate the contribution of \(I_{1,\varepsilon}\).
Since \(\widehat u\in\mathcal S(\mathbb R^n)\), for every \(N>0\) there exists \(C_N>0\) such that
\[
|\widehat u(\eta+w)-\widehat u(\eta)|
\le
C_N |w|(1+|\eta|)^{-N}
\qquad\text{for all }|w|\le1.
\]
Moreover, since \(|w|\le1\),
\[
|g_\varepsilon(\eta+w)|^2
=
(|\eta+w|^2+\varepsilon)^s
\le
C(1+|\eta|^2)^s.
\]
Therefore,
\[
I_{1,\varepsilon}(\eta,w)
\le
C_N |w|^2 (1+|\eta|)^{2s-2N}.
\]
Consequently,
\[
\begin{aligned}
\iint_{D_{\mathrm{in}}}
\frac{I_{1,\varepsilon}(\eta,w)}{|w|^{n+2\alpha}}\,d\eta dw
&\le
C_N \int_{|w|\le1}|w|^{2-n-2\alpha}\,dw
\int_{\mathbb R^n}(1+|\eta|)^{2s-2N}\,d\eta.
\end{aligned}
\]
Since \(\alpha<1\), the \(w\)-integral is finite, and choosing \(N\) sufficiently large makes the \(\eta\)-integral finite. Hence this term is bounded uniformly in \(\varepsilon\).

We next estimate the contribution of \(I_{2,\varepsilon}\).
Since \(0<s<1\), the map \(t\mapsto t^{s/2}\) is Hölder continuous of order \(s/2\) on \([0,\infty)\). Thus
\[
|a^{s/2}-b^{s/2}|\le |a-b|^{s/2}
\qquad\text{for all }a,b\ge0.
\]
Applying this with \(a=|\eta+w|^2+\varepsilon\) and \(b=|\eta|^2+\varepsilon\), we obtain
\[
\begin{aligned}
|g_\varepsilon(\eta+w)-g_\varepsilon(\eta)|
&\le
\bigl||\eta+w|^2-|\eta|^2\bigr|^{s/2} \\
&\le
C(|\eta|+|w|)^{s/2}|w|^{s/2}
\le
C(1+|\eta|)^{s/2}|w|^{s/2},
\end{aligned}
\]
for \(|w|\le1\). Hence
\[
I_{2,\varepsilon}(\eta,w)
\le
C |\widehat u(\eta)|^2 (1+|\eta|)^s |w|^s.
\]
Therefore,
\[
\begin{aligned}
\iint_{D_{\mathrm{in}}}
\frac{I_{2,\varepsilon}(\eta,w)}{|w|^{n+2\alpha}}\,d\eta dw
&\le
C \int_{\mathbb R^n} |\widehat u(\eta)|^2 (1+|\eta|)^s\,d\eta
\int_{|w|\le1}|w|^{s-n-2\alpha}\,dw.
\end{aligned}
\]
Since \(\alpha<s/2\), the \(w\)-integral is finite, and the \(\eta\)-integral is finite because
\(\widehat u\in\mathcal S(\mathbb R^n)\). This bound is again uniform in \(\varepsilon\).

Combining the estimates over \(D_{\mathrm{out}}\) and \(D_{\mathrm{in}}\), we conclude that
\[
\sup_{\varepsilon\in(0,1)}
[\widehat f_\varepsilon]_{\dot H^\alpha(\mathbb R^n)}
<\infty.
\]
Hence \eqref{eq:uniform-moment-feps} follows.

\medskip
\noindent
\textbf{Step 4: Uniform integrability of the entropy integrand.}

Define
\[
\Phi(t):=|t|^2|\ln|t||,\qquad t\in\mathbb R,
\]
with the convention \(\Phi(0)=0\).

We first show local uniform integrability. Since
\[
\sup_{\varepsilon\in(0,1)}\|f_\varepsilon\|_{L^\infty(\mathbb R^n)}\le A_\infty,
\]
the function \(\Phi\) is bounded on \([-A_\infty,A_\infty]\). Hence
\[
\sup_{\varepsilon\in(0,1)}\int_E \Phi(f_\varepsilon(x))\,dx
\le
\sup_{|t|\le A_\infty}\Phi(t)\,|E|
\]
for every measurable set \(E\subset\mathbb R^n\). It follows that
\(\{\Phi(f_\varepsilon)\}_{\varepsilon\in(0,1)}\) is uniformly integrable on sets of finite measure.

We next control the tails. Fix \(\delta\in(0,\min\{1,2\alpha n^{-1}\cdot\infty\})\) so small that
\[
\frac{2\alpha(2-\delta)}{\delta}>n.
\]
Using the elementary inequality
\[
t^2|\ln t|
\le
C_\delta\bigl(t^{2-\delta}+t^{2+\delta}\bigr),
\qquad t\ge0,
\]
we obtain
\[
\int_{|x|>R}\Phi(f_\varepsilon(x))\,dx
\le
C_\delta
\left(
\int_{|x|>R}|f_\varepsilon(x)|^{2-\delta}\,dx
+
\int_{|x|>R}|f_\varepsilon(x)|^{2+\delta}\,dx
\right).
\]

For the first term, by Hölder's inequality and \eqref{eq:uniform-moment-feps},
\[
\begin{aligned}
\int_{|x|>R}|f_\varepsilon(x)|^{2-\delta}\,dx
&=
\int_{|x|>R}
\bigl(|x|^{2\alpha}|f_\varepsilon(x)|^2\bigr)^{\frac{2-\delta}{2}}
|x|^{-\alpha(2-\delta)}\,dx \\
&\le
\left(\int_{|x|>R}|x|^{2\alpha}|f_\varepsilon(x)|^2\,dx\right)^{\frac{2-\delta}{2}}
\left(\int_{|x|>R}|x|^{-\frac{2\alpha(2-\delta)}{\delta}}\,dx\right)^{\frac{\delta}{2}} \\
&\le
C
\left(\int_{|x|>R}|x|^{-\frac{2\alpha(2-\delta)}{\delta}}\,dx\right)^{\frac{\delta}{2}},
\end{aligned}
\]
and the last quantity tends to \(0\) as \(R\to\infty\), uniformly in \(\varepsilon\), because
\(\frac{2\alpha(2-\delta)}{\delta}>n\).

For the second term, using the uniform \(L^\infty\) bound together with \eqref{eq:uniform-moment-feps},
\[
\begin{aligned}
\int_{|x|>R}|f_\varepsilon(x)|^{2+\delta}\,dx
&\le
\|f_\varepsilon\|_{L^\infty(\mathbb R^n)}^\delta
\int_{|x|>R}|f_\varepsilon(x)|^2\,dx \\
&\le
A_\infty^\delta R^{-2\alpha}
\int_{\mathbb R^n}|x|^{2\alpha}|f_\varepsilon(x)|^2\,dx
\le
C R^{-2\alpha},
\end{aligned}
\]
which also tends to \(0\) as \(R\to\infty\), uniformly in \(\varepsilon\).

Thus
\[
\lim_{R\to\infty}
\sup_{\varepsilon\in(0,1)}
\int_{|x|>R}\Phi(f_\varepsilon(x))\,dx
=0.
\]
Combining this tail estimate with the local uniform integrability proved above, we conclude that
\(\{\Phi(f_\varepsilon)\}_{\varepsilon\in(0,1)}\) is uniformly integrable in \(L^1(\mathbb R^n)\).

\medskip
\noindent
\textbf{Step 5: Passage to the limit.}

Since \(f_\varepsilon\to f\) uniformly on \(\mathbb R^n\), we have
\[
|f_\varepsilon(x)|^2\ln|f_\varepsilon(x)|
\longrightarrow
|f(x)|^2\ln|f(x)|
\qquad\text{for every }x\in\mathbb R^n,
\]
where we use the convention \(0\ln0=0\). By Step 4, the family
\[
\bigl\{|f_\varepsilon|^2\ln|f_\varepsilon|\bigr\}_{\varepsilon\in(0,1)}
\]
is uniformly integrable. Hence, by Vitali's theorem,
\[
\int_{\mathbb R^n}|f_\varepsilon(x)|^2\ln|f_\varepsilon(x)|\,dx
\longrightarrow
\int_{\mathbb R^n}|f(x)|^2\ln|f(x)|\,dx.
\]

Finally, assume \(f\not\equiv0\), and set
\[
h_\varepsilon=\frac{f_\varepsilon}{\|f_\varepsilon\|_2},
\qquad
h=\frac{f}{\|f\|_2}.
\]
Since \(\|f_\varepsilon-f\|_2\to0\), we have \(\|f_\varepsilon\|_2\to\|f\|_2>0\). Moreover,
\[
\int_{\mathbb R^n}|h_\varepsilon|^2\ln|h_\varepsilon|\,dx
=
\frac{1}{\|f_\varepsilon\|_2^2}
\int_{\mathbb R^n}|f_\varepsilon|^2\ln|f_\varepsilon|\,dx
-
\ln\|f_\varepsilon\|_2,
\]
and similarly
\[
\int_{\mathbb R^n}|h|^2\ln|h|\,dx
=
\frac{1}{\|f\|_2^2}
\int_{\mathbb R^n}|f|^2\ln|f|\,dx
-
\ln\|f\|_2.
\]
Passing to the limit yields
\[
\int_{\mathbb R^n}|h_\varepsilon(x)|^2\ln|h_\varepsilon(x)|\,dx
\longrightarrow
\int_{\mathbb R^n}|h(x)|^2\ln|h(x)|\,dx.
\]
This completes the proof.
\end{proof}

% Sketch of the constant shift:
% Since ln|2\pi\xi| = ln|\xi| + ln(2\pi) and \int|(\Finfty f)|^2=1,
% passing from \Ftwopi to \Finfty shifts the constant by +(n/2)ln(2\pi).

%------------------------------------------------------------
% Relation between Beckner's log-uncertainty and the Euclidean
% logarithmic Laplacian inequality
%------------------------------------------------------------

%------------------------------------------------------------
% A Beckner-type inequality for (-Δ)^{s+ln} via the substitution
% f = (-Δ)^{s/2}u
%------------------------------------------------------------
Combining Beckner's logarithmic uncertainty principle with the regularization argument from Lemma~\ref{lem:entropy-convergence-feps}, we obtain the following sharp fractional--logarithmic inequality.

\begin{proof}[\textbf{Proof of Proposition \ref{prop:fraclog-beckner}.}]
Set
\[
f:=(-\Delta)^{s/2}u,
\qquad
\widehat f(\xi)=|\xi|^s\widehat u(\xi).
\]
Since \(u\in\mathcal S(\mathbb R^N)\), for \(\varepsilon\in(0,1)\) define
\[
\widehat f_\varepsilon(\xi):=(|\xi|^2+\varepsilon)^{s/2}\widehat u(\xi),
\qquad
g_\varepsilon:=\frac{f_\varepsilon}{\|f_\varepsilon\|_2}.
\]
Then \(\widehat f_\varepsilon\in\mathcal S(\mathbb R^N)\), hence \(f_\varepsilon\in\mathcal S(\mathbb R^N)\), and therefore
\(g_\varepsilon\in\mathcal S(\mathbb R^N)\) with
\[
\|g_\varepsilon\|_{L^2(\mathbb R^N)}=1.
\]
Thus Beckner's logarithmic uncertainty principle applies to \(g_\varepsilon\), giving
\begin{equation}\label{eq:beckner-geps}
\frac n2\int_{\mathbb R^N}\ln|\xi|\,|\widehat g_\varepsilon(\xi)|^2\,d\xi
\ge
\int_{\mathbb R^N}\ln|g_\varepsilon(x)|\,|g_\varepsilon(x)|^2\,dx
+B_N.
\end{equation}

We now pass to the limit as \(\varepsilon\downarrow0\). First, since
\[
\bigl|(|\xi|^2+\varepsilon)^{s/2}\widehat u(\xi)\bigr|^2
\le
(1+|\xi|^2)^s|\widehat u(\xi)|^2,
\]
and \((1+|\xi|^2)^s|\widehat u(\xi)|^2\in L^1(\mathbb R^N)\), the dominated convergence theorem yields
\[
\|f_\varepsilon\|_2^2
=
\int_{\mathbb R^N} (|\xi|^2+\varepsilon)^s |\widehat u(\xi)|^2\,d\xi
\longrightarrow
\int_{\mathbb R^N} |\xi|^{2s} |\widehat u(\xi)|^2\,d\xi
=
\|f\|_2^2
=
1.
\]
Hence
\[
\widehat g_\varepsilon(\xi)
=
\frac{(|\xi|^2+\varepsilon)^{s/2}\widehat u(\xi)}{\|f_\varepsilon\|_2}.
\]
Next, again by dominated convergence,
\[
\int_{\mathbb R^N}\ln|\xi|\,|\widehat g_\varepsilon(\xi)|^2\,d\xi
=
\frac{1}{\|f_\varepsilon\|_2^2}
\int_{\mathbb R^N}\ln|\xi|\,(|\xi|^2+\varepsilon)^s|\widehat u(\xi)|^2\,d\xi
\longrightarrow
\int_{\mathbb R^N}\ln|\xi|\,|\xi|^{2s}|\widehat u(\xi)|^2\,d\xi.
\]
Indeed, the integrand is dominated by
\[
|\ln|\xi||\,(1+|\xi|^2)^s|\widehat u(\xi)|^2\in L^1(\mathbb R^N),
\]
since \(\widehat u\in\mathcal S(\mathbb R^N)\).

On the other hand, by Lemma~\ref{lem:entropy-convergence-feps}, since \(\|f\|_2=1\), we have
\[
\int_{\mathbb R^N}|g_\varepsilon(x)|^2\ln|g_\varepsilon(x)|\,dx
\longrightarrow
\int_{\mathbb R^N}|f(x)|^2\ln|f(x)|\,dx.
\]
Therefore, letting \(\varepsilon\downarrow0\) in \eqref{eq:beckner-geps}, we obtain
\[
\frac n2\int_{\mathbb R^N}\ln|\xi|\,|\xi|^{2s}|\widehat u(\xi)|^2\,d\xi
\ge
\int_{\mathbb R^N}|f(x)|^2\ln|f(x)|\,dx
+
B_N.
\]
Since \(f=(-\Delta)^{s/2}u\), this becomes
\[
\frac n2\int_{\mathbb R^N}\ln|\xi|\,|\xi|^{2s}|\widehat u(\xi)|^2\,d\xi
\ge
\int_{\mathbb R^N}\big|(-\Delta)^{s/2}u(x)\big|^2
\ln\big|(-\Delta)^{s/2}u(x)\big|\,dx
+
B_N,
\]
we get
\[
\int_{\mathbb R^N}\ln|\xi|\,|\xi|^{2s}|\widehat u(\xi)|^2\,d\xi
=
\frac12
\int_{\mathbb R^N}|\xi|^{2s}\ln|\xi|^2\,|\widehat u(\xi)|^2\,d\xi
=
\frac12\,\langle u,(-\Delta)^{s+\ln}u\rangle.
\]
Substituting this into the previous inequality yields \eqref{eq:fraclog-ineq-normalized}.

Finally, if \((-\Delta)^{s/2}u\) is an extremal for Beckner's inequality, then the right-hand side and left-hand side above coincide, and therefore equality holds in \eqref{eq:fraclog-ineq-normalized}.
\end{proof}

%------------------------------------------------------------
% A moment-bound inequality for (-Δ)^{s+ln} via Shannon Entropy
%------------------------------------------------------------
\subsubsection{Second version using the Shannon Entropy}

The inequality obtained above already relates the fractional--logarithmic energy to an entropy functional. However, the term
\[
\int_{\mathbb R^N}\big|(-\Delta)^{s/2}u(x)\big|^2
\ln\big|(-\Delta)^{s/2}u(x)\big|\,dx
\]
is still too implicit for many purposes. A natural way to further estimate it is to invoke the sharp Shannon entropy--moment inequality for probability densities, which bounds the entropy from below in terms of the second moment.

Indeed, let \(\rho(x)=|f(x)|^2\), where \(\|f\|_{L^2(\mathbb R^N)}=1\), so that \(\rho\) is a probability density on \(\mathbb R^N\). By the Gaussian maximization property of the differential entropy
(see \cite[Theorem 8.6.5]{cover2006elements}),
\[
-\int_{\mathbb R^N}\rho(x)\ln \rho(x)\,dx
\le
\frac12\ln\big((2\pi e)^N \det K\big),
\]
where \(K\) is the covariance matrix of \(\rho\). Using the arithmetic--geometric mean inequality,
\[
\det K \le \left(\frac{\operatorname{tr}K}{N}\right)^N,
\qquad
\operatorname{tr}K:=\int_{\mathbb R^N}|x|^2\rho(x)\,dx,
\]
we obtain
\begin{equation}\label{eq:entropy-moment-f}
\int_{\mathbb R^N}|f(x)|^2\ln|f(x)|\,dx
\ge
-\frac{N}{4}\ln\!\Bigg(\frac{2\pi e}{N}\int_{\mathbb R^N}|x|^2|f(x)|^2\,dx\Bigg),
\end{equation}
whenever \(\|f\|_{L^2(\mathbb R^N)}=1\).

Equality holds if and only if \(\rho=|f|^2\) is a centered Gaussian density, namely,
\[
|f(x)|^2=\frac{1}{(2\pi \sigma^2)^{N/2}}
\exp\!\left(-\frac{|x|^2}{2\sigma^2}\right)
\]
for some \(\sigma>0\).

\begin{proof}[\textbf{Proof of Proposition \ref{prop:fraclog_moment}.}]
Let
\[
f:=(-\Delta)^{s/2}u.
\]
By Proposition~\ref{prop:fraclog-beckner}, we obtain
\begin{equation}\label{eq:beckner-fraclog-applied}
\frac{N}{4}\,\langle u,(-\Delta)^{s+\ln}u\rangle
\ge
\int_{\mathbb R^N}|f(x)|^2\ln|f(x)|\,dx
+
 B_N.
\end{equation}
By \eqref{eq:entropy-moment-f}, we obtain
\[
\frac{N}{4}\,\langle u,(-\Delta)^{s+\ln}u\rangle
\ge
-\frac{N}{4}\ln\!\Bigg(\frac{2\pi e}{N}\int_{\mathbb R^N}|x|^2\, \big|(-\Delta)^{s/2}u(x)\big|^2\,dx\Bigg)
+B_N.
\]

It remains to show that the inequality is strict. Equality in the above argument would require equality simultaneously in Proposition~\ref{prop:fraclog-beckner} and in \eqref{eq:entropy-moment-f}. By Proposition~\ref{prop:fraclog-beckner}, equality in the first inequality holds if and only if \(f\) is, up to conformal automorphisms, a bubble of the form
\[
f(x)=A(1+|x|^2)^{N/2}.
\]
On the other hand, equality in the sharp entropy--moment inequality \eqref{eq:entropy-moment-f} holds if and only if \(|f|^2\) is Gaussian, equivalently,
\[
f(x)=C e^{-\alpha |x|^2}
\]
up to the obvious normalization constants. These two families do not intersect nontrivially. Therefore the inequality is strict, and \eqref{eq:fraclog-moment-normalized} follows.
\end{proof}

%------------------------------------------------------------
% An L^q-norm lower bound for (-Δ)^{s+ln} via Jensen's inequality
%------------------------------------------------------------
\subsubsection{Third version using Jensen's inequality}
A different way to handle the entropy term
\[
\int_{\mathbb R^N}\big|(-\Delta)^{s/2}u(x)\big|^2
\ln\big|(-\Delta)^{s/2}u(x)\big|\,dx
\]
is to apply Jensen's inequality directly, rather than passing through the Shannon entropy--moment inequality. This provides another explicit estimate of the entropy functional and yields a third sharp fractional--logarithmic Sobolev-type inequality.

\begin{proof}[\textbf{Proof of Proposition \ref{prop:fraclog-Lq-bound}.}]
Set
\[
f:=(-\Delta)^{s/2}u.
\]
Then by Proposition~\ref{prop:fraclog-beckner}, we obtain
\begin{equation}\label{eq:beckner-fraclog-Lq}
\frac{N}{4}\,\langle u,(-\Delta)^{s+\ln}u\rangle
\ge
\int_{\mathbb R^N}|f(x)|^2\ln|f(x)|\,dx
+
 B_N.
\end{equation}

We now estimate the entropy term from below in terms of the \(L^q\)-norm of \(f\). Since \(\|f\|_2=1\), the measure
\[
d\mu(x):=|f(x)|^2\,dx
\]
is a probability measure on \(\mathbb R^N\). Note that
\[
\int_{\mathbb R^N}|f(x)|^2\ln|f(x)|\,dx
=
\frac{1}{q-2}\int_{\mathbb R^N}\ln\bigl(|f(x)|^{q-2}\bigr)\,d\mu(x),
\]
since \(\ln\) is concave, Jensen's inequality gives
\[
\int_{\mathbb R^N}\ln\bigl(|f|^{q-2}\bigr)\,d\mu
\le
\ln\!\left(\int_{\mathbb R^N}|f|^{q-2}\,d\mu\right).
\]
Because \(q-2<0\), multiplying by \(\frac{1}{q-2}\) reverses the inequality, and hence
\[
\int_{\mathbb R^N}|f(x)|^2\ln|f(x)|\,dx
\ge
\frac{1}{q-2}
\ln\!\left(\int_{\mathbb R^N}|f(x)|^{q-2}|f(x)|^2\,dx\right).
\]
That is,
\[
\int_{\mathbb R^N}|f(x)|^2\ln|f(x)|\,dx
\ge
\frac{1}{q-2}\ln\!\left(\int_{\mathbb R^N}|f(x)|^q\,dx\right).
\]
Substituting this bound into \eqref{eq:beckner-fraclog-Lq}, and recalling that \(f=(-\Delta)^{s/2}u\), we obtain
\[
\frac{n}{4}\,\langle u,(-\Delta)^{s+\ln}u\rangle
\ge
-\frac{1}{2-q}\ln\Big\|(-\Delta)^{s/2}u\Big\|_{L^q(\mathbb R^N)}^q
+B_N.
\]

It remains to show that the inequality is strict. Equality in the Jensen step would require $|f(x)|^{q-2}$
to be constant \(\mu\)-almost everywhere, hence \(|f|\) itself to be constant on the support of \(\mu\), that is, on the set where \(f\neq 0\). Since \(f\in \mathcal S(\mathbb R^N)\subset L^2(\mathbb R^N)\) and \(\|f\|_2=1\), this is impossible unless \(f=0\), which is excluded. Therefore equality cannot occur in the Jensen inequality, and the above estimate is in fact strict. This proves \eqref{eq:fraclog-Lq-normalized}.
\end{proof}

	% ================= Acknowledgements =================

		% ================= Bibliography =================
		\printbibliography

		% ================= Author's Information =================

	\vspace{2em}

\small
	\noindent\textit{Huyuan Chen}:  Center for Mathematics and Interdisciplinary Sciences,\\[1mm]
		Fudan University, Shanghai 200433, PR China.\\[2mm]
Shanghai Institute for Mathematics and Interdisciplinary Sciences,\\[1mm]
		 Shanghai 200433, PR China \\[1mm]  
         \noindent\emph{Email:} \texttt{chenhuyuan@simis.cn} \\[3mm]

	\noindent\textit{Rui Chen}: School of Mathematical Sciences, Fudan University,\\[1mm]
		Shanghai 200433,  China\\[2mm]
		Brandenburg University of Technology Cottbus--Senftenberg,\\[1mm]
		Cottbus 03046, Germany\\[1mm]
		\noindent\emph{Email:} \texttt{chenrui23@m.fudan.edu.cn}\\[3mm]

	\noindent\textit{Daniel Hauer }:
         Brandenburg University of Technology Cottbus–Senftenberg,\\[1mm]
		Platz der Deutschen Einheit 1, 03046 Cottbus, Germany\\[2mm]
		 School of Mathematics and Statistics, The University of Sydney,\\[1mm]
		NSW 2006, Australia \\[1mm]
      	\noindent\emph{Email:} \texttt{daniel.hauer@b-tu.de}

	\end{document}
    \newpage

		\vspace{10em}

$$ (-\Delta)^{s+\ln}v- \frac{2}{N-2s}\Big(\,(\ln v)\,(-\Delta)^sv
 +  (-\Delta)^s\bigl((\ln v)\,v\bigr)\Big)\quad {\rm  in}\ \R^N$$

% Assume 0<s<1 and v\in C_c^2(\R^N) with v>0 (so that \ln v is well-defined).

% Recall the pointwise integral representations
\[
(-\Delta)^s v(x)=c_{N,s}\,\int_{\R^N}\frac{v(x)-v(y)}{|x-y|^{N+2s}}\,dy,
\]
\[
(-\Delta)^{s+\ln}v(x)
=c_{N,s}\,\int_{\R^N}\frac{v(x)-v(y)}{|x-y|^{N+2s}}
\Bigl(b_{N,s}-2\ln|x-y|\Bigr)\,dy,
\qquad 
b_{N,s}:=\frac{d}{ds}c_{N,s}.
\]

% Start from
% \int vE(v)=<v,(-\Delta)^{s+\ln}v>-\frac{4}{N-2s}<(\ln v)v,(-\Delta)^s v>.
% Expanding each pairing:

\[
\int_{\R^N} v(x)\,E(v)(x)\,dx
=
c_{N,s}\int_{\R^N} v(x)\,
\int_{\R^N}\frac{v(x)-v(y)}{|x-y|^{N+2s}}
\Bigl(b_{N,s}-2\ln|x-y|\Bigr)\,dy\,dx
\]
\[
\hspace{30mm}
-\frac{4}{N-2s}\,c_{N,s}\int_{\R^N} (\ln v(x))\,v(x)\,
\int_{\R^N}\frac{v(x)-v(y)}{|x-y|^{N+2s}}\,dy\,dx.
\]

% If you prefer a symmetric (Gagliardo) double-integral expansion, use
% <\phi,(-\Delta)^s\psi>=\frac{c_{N,s}}2\iint \frac{(\phi(x)-\phi(y))(\psi(x)-\psi(y))}{|x-y|^{N+2s}} dxdy,
% and the analogous identity with the kernel (b_{N,s}-2\ln|x-y|) for (-\Delta)^{s+\ln}:

\[
\big\langle v,(-\Delta)^{s+\ln}v\big\rangle
=
\frac{c_{N,s}}{2}\iint_{\R^N\times\R^N}
\frac{\bigl(v(x)-v(y)\bigr)^2}{|x-y|^{N+2s}}
\Bigl(b_{N,s}-2\ln|x-y|\Bigr)\,dx\,dy,
\]
\[
\big\langle (\ln v)\,v,(-\Delta)^s v\big\rangle
=
\frac{c_{N,s}}{2}\iint_{\R^N\times\R^N}
\frac{\bigl(v(x)\ln v(x)-v(y)\ln v(y)\bigr)\bigl(v(x)-v(y)\bigr)}{|x-y|^{N+2s}}\,dx\,dy.
\]

% Hence
\[
\int_{\R^N} v\,E(v)\,dx
=
\frac{c_{N,s}}{2}\iint_{\R^N\times\R^N}
\frac{\bigl(v(x)-v(y)\bigr)^2}{|x-y|^{N+2s}}
\Bigl(b_{N,s}-2\ln|x-y|\Bigr)\,dx\,dy
\]
\[
\hspace{30mm}
-\frac{4}{N-2s}\cdot\frac{c_{N,s}}{2}\iint_{\R^N\times\R^N}
\frac{\bigl(v(x)\ln v(x)-v(y)\ln v(y)\bigr)\bigl(v(x)-v(y)\bigr)}{|x-y|^{N+2s}}\,dx\,dy.
\]
	
	\end{document}

\begin{rmk}
If the map
\[
  \lambda\mapsto\varphi^{s+\ln}_N(\lambda)
\]
is injective on the spectrum
\(
  \{\lambda_k=k(k+N-1):k\ge0\}
\)
of $-\Delta_g$. Then,  \eqref{eq:eig-lap} and \eqref{eq:eig-frac-log} are
equivalent. More precisely, if $\Phi\not\equiv0$ satisfies
\(
  \mathscr{P}^{s+\ln}_g\Phi = \mu\Phi
\),
then there exists a unique $k_0\ge0$ such that
\(
  \mu=\varphi^{s+\ln}_N(\lambda_{k_0})
\),
we have
\(
  \Phi\in\mathcal{H}_{k_0}
\),
and consequently
\(
  -\Delta_g\Phi=\lambda_{k_0}\Phi.
\)
\end{rmk}

\begin{proof}
Let $\Phi$ be as in the statement and expand it in spherical harmonics
as in the proof of Theorem~\ref{thm:frac-log-basic}(ii). The argument
there shows that $\Phi$ belongs to
\(
  \bigoplus_{k\in K}\mathcal{H}_k
\)
with
\(
  K=\{k:\varphi^{s+\ln}_N(\lambda_k)=\mu\}.
\)
By the injectivity assumption, the set $K$ contains at most one index.
If $\Phi\not\equiv0$, then $K$ is nonempty and there exists a unique
$k_0$ with $K=\{k_0\}$. Therefore $\Phi\in\mathcal{H}_{k_0}$ and
\(
  -\Delta_g\Phi=\lambda_{k_0}\Phi
\),
which yields the claimed equivalence between
\eqref{eq:eig-lap} and \eqref{eq:eig-frac-log}.
\end{proof}